\documentclass{imanum}
\usepackage{graphicx}
\usepackage{amsmath,multirow}
\usepackage{amssymb}
\usepackage{amscd}
\usepackage{amsfonts}
\usepackage{enumerate}
\usepackage{mathrsfs}
\usepackage{xypic}
\usepackage{stmaryrd}
\usepackage{graphicx}
\usepackage{fancybox}
\usepackage{color}
\usepackage{amsmath}
\usepackage{exscale}
\usepackage{enumitem,array}%
\usepackage{amscd}
\usepackage{amsfonts}
\usepackage{enumerate}
\usepackage{mathrsfs}
\usepackage{xy}
\usepackage{stmaryrd}
\usepackage{graphicx}

\usepackage{color}

\usepackage{exscale}
\usepackage{enumitem,array}%
\usepackage{subfigure}
\usepackage{extarrows}
\usepackage{amsfonts}
\usepackage{xypic}
\usepackage{float}

\jno{drnxxx}

\usepackage{amsmath}
\usepackage{amssymb}
\usepackage{amscd}
\usepackage{amsfonts}
\usepackage{enumerate}
\usepackage{mathrsfs}
\usepackage{xy}
\usepackage{stmaryrd}
\usepackage{graphicx}
\usepackage{fancybox}
\usepackage{color}
\usepackage{multirow}
\usepackage{exscale}
\usepackage{enumitem,array}%
\usepackage{subfigure}
\usepackage{extarrows}
\usepackage{amsfonts}
\usepackage{xypic}

\usepackage{algorithm}
\usepackage{algpseudocode} 
\algrenewcommand\algorithmicrequire{\textbf{Input:}}
\algrenewcommand\algorithmicensure{\textbf{Output:}}

\algdef{SE}[FOR]{MyFor}{MyEndFor}[1]{\algorithmicfor\ #1\ \algorithmicdo}{\textbf{end for}}

\algdef{SE}[WHILE]{MyWhile}{MyEndWhile}[1]
{\algorithmicwhile\ #1\ \algorithmicdo}
{\textbf{end while}}

\usepackage{tikz}
\usepackage{epstopdf}

\def\pdt2{\partial_t^2}
\def\pdx2{\partial_x^2}

\newcommand{\absmm}[1]{{\left\vert\kern-0.25ex\left\vert\kern-0.25ex\left\vert #1
    \right\vert\kern-0.25ex\right\vert\kern-0.25ex\right\vert}}

\def\Im{\mathrm{Im}}

\def\pdt2{\partial_t^2}
\def\pdx2{\partial_x^2}

\newcommand{\abs}[1]{\left\vert#1\right\vert}

\def\Im{\mathrm{Im}}

\def\Im{{\mathrm{Im}\,}}

\newtheorem{thm}{Theorem}[section]
\newtheorem{mytheo}{Theorem}[section]

\newtheorem{lem}[thm]{Lemma}

\newtheorem{algodef}[thm]{Algorithm}

\newtheorem{rem}[thm]{Remark}
\begin{document}

\title{A  filtered two-step variational integrator for charged-particle dynamics in a  moderate or strong  magnetic field}
\shorttitle{A filtered  variational integrator for CPD}


\author{%
	{\sc Ting Li\thanks{Email: 1009587520@qq.com}} \\[2pt]
	School of Mathematics and Statistics, Xi'an Jiaotong University, 710049 Xi'an, China\\Mathematisches Institut, University of T\"{u}bingen, Auf der Morgenstelle 10, 72076 T\"{u}bingen, Germany\\[6pt]
	{\sc and}\\[6pt]
	{\sc  Bin Wang}\thanks{Corresponding author. Email: wangbinmaths@xjtu.edu.cn}\\[2pt]
	School of Mathematics and Statistics, Xi'an Jiaotong University, 710049 Xi'an, China}
\shortauthorlist{T. Li and B. Wang}

\maketitle

\begin{abstract}
{This article is concerned with a new  filtered two-step variational integrator for solving the charged-particle dynamics in a mildly non-uniform moderate or strong magnetic field with a dimensionless parameter $\varepsilon$ inversely
proportional to the strength of the magnetic field. In the case of a moderate magnetic field ($\varepsilon=1$), second-order error bounds and long-time near-conservation of energy and momentum are obtained. Moreover, the proof of the long-term analysis is accomplished by the backward error analysis. For $0<\varepsilon \ll 1$, the proposed integrator achieves uniform second-order accuracy in the position and the parallel velocity for large step sizes, while attaining first-order accuracy with respect to the small parameter $\varepsilon$ for smaller step sizes. The error bounds are derived from a comparison of the modulated Fourier expansions of the exact and numerical solutions. Moreover, long-time near-conservation of the energy and the magnetic moment is established using modulated Fourier expansion and backward error analysis. All the theoretical results of the error behavior and long-time near-conservation are numerically demonstrated by four numerical experiments.}
 {Charged particle dynamics, two-step filtered variational integrator, backward error analysis, modulated Fourier expansion, convergence, long-time behavior}
\end{abstract}

\section{Introduction}\label{intro}
The dynamics of charged particles in electromagnetic fields play a fundamental role in extreme cosmic environments, and their comprehensive study is of both theoretical and practical significance (\cite{Arnold, Northrop}).  This paper focuses on the time integration of  the equations governing  a particle's motion in a mildly non-uniform moderate or strong  magnetic field. The particle motion can be described by the differential equation (\cite{Lubich2022})
\begin{equation}\label{charged-particle}
	\begin{aligned}
		&\ddot{x}(t)=\dot{x}(t) \times  B(x(t)) +F(x(t)),\ \ t\in[0,T]\ \
		  \textmd{with} \ \  B(x)=\frac{1}{\varepsilon}B_0+B_1(x),
	\end{aligned}
\end{equation}
where  $x(t)\in \mathbb{R}^3$ represents the  position of a particle moving
in an electromagnetic field, $B_0$ is a fixed vector in $\mathbb{R}^3$ with the uniform bound $\abs{B_0}=1$,  with $\abs{\cdot}$ denoting the Euclidean norm, $B_1(x): \mathbb{R}^3\rightarrow \mathbb{R}^3$ is a non-constant magnetic
field which is assumed to have a known vector potential $ A_1(x) $
and the nonlinear function $F(x): \mathbb{R}^3\rightarrow \mathbb{R}^3$ describes the electric field. We assume that $B_1(x)$ and $F(x)$ are  smooth and uniformly bounded w.r.t. $\varepsilon$. It is noted that the magnetic field $B(x)$ can be expressed as $B(x)=\nabla_x \times A(x)$ by  the  vector potential $A(x)=-\frac{1}{2}x \times \frac{B_0}{\varepsilon}+ A_1(x)$ and its  strength is determined  by the parameter $1/\varepsilon$.
The initial position and velocity are assumed to have an $\varepsilon$-independent bound
\begin{equation}\label{velocity}
	\abs{x(0)} \leq \hat{C},\quad \abs{\dot{x}(0)} \leq \hat{C}.
\end{equation}

In addition to the mildly non-uniform strong magnetic field \eqref{charged-particle} described above,  the analysis can also be extended to the so-called maximal ordering scaling case (\cite{Brizard,Lubich2020}), where  
$
B(x) =B(\varepsilon x)/\varepsilon.
$
This representation can be decomposed as  
$
B(x) =B(\varepsilon x_0)/\varepsilon + (B(\varepsilon x) - B(\varepsilon x_0))/\varepsilon,
$
in which the second term remains uniformly bounded with respect to $\varepsilon$.  Consequently, it can be regarded as the  component $B_1(x)$ in \eqref{charged-particle}, 
so that the framework developed in this paper naturally extends to this scaling as well. By contrast, for a general non-uniform strong magnetic field of the form $B(x)=\hat B(x)/\varepsilon$ with $\hat B(x)$ independent of $\varepsilon$, the derivatives of the magnetic field are no longer uniformly bounded with respect to $\varepsilon$. This contrasts with the mildly non-uniform and maximal ordering cases considered above. Accordingly, a different analytical approach is required, which will be pursued in future work.

The charged-particle dynamics (CPD) \eqref{charged-particle}  possesses  many important properties. If the nonlinear function  has a scalar potential, i.e., $F(x)=-\nabla_{x}U(x)$ with scalar potential $U(x)$, the CPD \eqref{charged-particle} is an Euler--Lagrange equation for the Lagrange function
\begin{equation}\label{L(x,v)}
	L(x,\dot{x})=\frac{1}{2}\abs{\dot{x}}^2+A(x)^{\intercal}\dot{x}-U(x).
	\end{equation}  Moreover, the system has conservation laws. The energy $H(x,v)=\abs{v}^2/2+U(x)$ is exactly conserved  along the solution of  \eqref{charged-particle}, where $v:=\dot{x}$.
For moderate magnetic fields ($\varepsilon=1$), if the scalar and vector potentials possess certain invariance properties, then the Lagrangian \eqref{L(x,v)} satisfies the relation $L(e^{\tau S}x,e^{\tau S}v)=L(x,v)$ with a skew-symmetric matrix $S$. By Noether’s theorem, this implies the conservation of the momentum $M(x,v) = (v + A(x))^{\intercal} Sx,$
which is identified in \cite{Hairer2017-2,Hairer2017-1} as an invariant of the charged-particle dynamics \eqref{charged-particle}.
Meanwhile, for the strong regime ($0< \varepsilon \ll 1$), from  \cite{Northrop,Benettin1994,Arnold}, we know that the magnetic moment $I(x,v)=\frac{1}{2\varepsilon}\frac{\abs{v \times  B(x) }^2}{\abs{B(x)}^3}$
 is an adiabatic invariant of \eqref{charged-particle}.

In this paper, we explore two different regimes for the parameter
$\varepsilon$.  \textit{In the strong regime, where $0< \varepsilon \ll 1$,}  the solution of equation \eqref{charged-particle} exhibits highly oscillatory behavior, which is significant in applications like plasma physics. Conversely, \textit{in the moderate regime $\varepsilon=1$,} the solution does not display such oscillations, and the system behaves as a typical dynamic system. Various time discretization methods have been extensively developed and researched in recent decades for solving  CPD with a moderate magnetic field. One of the most well-known methods is the Boris method \cite{Boris}, which is widely used for its simplicity and favorable long-time conservation properties. Further theoretical analysis of its properties has been provided by \cite{Qin2013}  and  \cite{Hairer2017-1}. In the recent few decades, structure-preserving methods for differential equations have gained increasing attention. Several structure-preserving methods have been  investigated and applied to CPD. Symmetric methods have been explored by \cite{Ostermann15,Hairer2017-2,wangwu2020}.
Energy/volume-preserving integrators have been proposed by \cite{He2015,Brugnano2019,Brugnano2020,Chacon}.  Furthermore, symplectic and K-symplectic methods have been studied in works such as   \cite{PRL1,Xiao,Webb2014,Tao2016,Zhang2016,He2017,Y. Shi,Xiao21,li2023}.

For  the CPD under  a strong  magnetic field ($0<\varepsilon\ll 1$), numerous numerical methods have been proposed. Generally speaking, the existing methods can be  divided into two  categories. The first focuses on long-time near-conservation of invariants, as studied in \cite{Hairer2018,wangwu2020}.
These methods can have good long-time near-conservation but the accuracy becomes badly when $\varepsilon$ is small. The second category consists of integrators designed to improve accuracy in the strong magnetic field regime. Asymptotic preserving  schemes have been developed  in \cite{VP4,VP5,Chacon} but their errors are not uniformly accurate with respect to   the parameter $\varepsilon$.  Two filtered Boris algorithms were designed   in \cite{Lubich2020} to have uniform   second order accuracy for  CPD under the scaling as $B(x,t)=B_0(\varepsilon x)/\varepsilon+B_1(x,t)$. Three splitting methods with first-order uniform error bound  were proposed  in \cite{WZ}. A recent work \cite{WJ23} developed a class of algorithms with improved accuracy 
for the two dimensional CPD and they become uniformly accurate methods when applied to the three  dimensional system under  the scaling $B=B(\varepsilon x)$.  Some other uniformly accurate methods have been presented in \cite{VP3,CPC,VP2,VP1} for the Vlasov--Poisson system.  Although these methods have nice accuracy when solving CPD under a strong  magnetic field, they usually do not have long-time conservation. In a very recent paper (see, e.g. \cite{Lubich2022}), several large-stepsize integrators were studied for charged-particle dynamics in a mildly strong non-uniform magnetic field, and one scheme was shown to exhibit good accuracy and long-time near-conservation properties for large step sizes. The present work develops a filtered two-step variational integrator and provides a comprehensive analysis covering both moderate and strong magnetic field regimes.  Favorable accuracy and long-time near-conservation properties are established for all considered magnetic field strengths and step-size regimes.

In this paper, a new filtered two-step variational integrator is constructed by introducing two filter functions, thereby achieving favorable accuracy and long-term behavior.
The main contributions of this paper are as follows.
\begin{itemize}
	
	\item A symmetric filtered variational integrator is developed for charged-particle dynamics, together with a comprehensive analysis for both moderate and strong magnetic fields, establishing accuracy and long-time near-conservation properties.

	\item For the moderate magnetic field regime ($\varepsilon=1$), the proposed integrator is shown to achieve second-order accuracy together with long-time near-conservation of energy and momentum. The long-term behavior is rigorously established by backward error analysis.
	
\item For the strong magnetic field regime ($0<\varepsilon\ll1$),  uniform second-order accuracy for large step sizes and first-order accuracy in the small parameter for smaller step sizes are rigorously established, together with long-time near-conservation of the energy and magnetic moment, 
based on modulated Fourier expansion and backward error analysis.
\end{itemize}

The  rest of this article is organized as follows.  In Section~\ref{sec2}, we construct a filtered two-step variational integrator based on a discrete Lagrangian and two filter functions. The main results for mildly non-uniform moderate and strong magnetic fields, along with four numerical experiments, are provided in Section \ref{sec3}.  Section~\ref{sec4} analyzes the error bounds and long-time near-conservation of energy and momentum for a moderate magnetic field.  Moreover, error bounds and a long-time analysis of the energy and magnetic moment for the strong magnetic field regime are presented in Section~\ref{sec5}. The last section is devoted to the conclusions of this paper.

\section{Numerical integrator}\label{sec2}
This section presents the construction of a filtered variational integrator for \eqref{charged-particle}. The method is based on a discrete Lagrangian combined with suitable filter functions, and its formulation proceeds in three steps. Firstly, we introduce a discrete Lagrangian. Applying the discrete Hamilton principle yields the associated discrete Euler–Lagrange equations, which define a variational integrator. As established by Theorem~6.1 in \cite{hairer2006}, such integrators are symplectic and possess favorable long-term energy behavior, which motivates the use of the discrete Lagrangian framework. Secondly, to address the highly oscillatory nature of charged-particle dynamics in a strong magnetic field, appropriate filter functions are incorporated to enhance numerical accuracy and stability. Finally, these filter functions are explicitly given.	 
	 
\begin{itemize}
	\item \textbf{Step 1: Discrete Lagrangian.} 
	Let $x_n \approx x(t_n)$ denote the numerical approximation at $t_n=nh$, with stepsize $h>0$. Based on the Lagrangian function $L(x,\dot{x})$ given in \eqref{L(x,v)},  the discrete Lagrangian is constructed by approximating the action integral
	over each time interval using the midpoint rule (see, e.g., Example~6.3 in Chapter~VI of \cite{hairer2006}). This leads to
	\begin{equation*}\label{ L_h}
		L_{h}(x_n,x_{n+1})=\frac{h}{2}\big(\hat{v}_{n+\frac{1}{2}}\big)^{\intercal}\hat{v}_{n+\frac{1}{2}}
		+hA\big({x}_{n+\frac{1}{2}}\big)^{\intercal}\hat{v}_{n+\frac{1}{2}}-hU({x}_{n+\frac{1}{2}}) \approx \int_{t_n}^{t_{n+1}} L(x(t),\dot{x}(t))dt,
	\end{equation*}
	with the abbreviations $\hat{v}_{n+1/2} := (x_{n+1} - x_n)/h$ 
	and $x_{n+1/2} := (x_n + x_{n+1})/2$, respectively.
	Following the discrete Hamilton principle,
	the associated discrete momenta $p_n$ and $p_{n+1}$,
	corresponding to $p(t)=v(t)+A(x(t))$, are given by
	\begin{equation}\label{Met-p}
		\begin{aligned}
			&p_{n}=-\frac{\partial L_h}{\partial x_n}(x_n,x_{n+1})=\hat{v}_{n+\frac{1}{2}}
			-\frac{h}{2}A'\big({x}_{n+\frac{1}{2}}\big)^{\intercal}\hat{v}_{n+\frac{1}{2}}
			+A\big({x}_{n+\frac{1}{2}}\big)
			-\frac{h}{2}F\big({x}_{n+\frac{1}{2}}\big),\\
			&p_{n+1}=\frac{\partial L_h}{\partial x_{n+1}}(x_n,x_{n+1})=\hat{v}_{n+\frac{1}{2}}
			+\frac{h}{2}A'\big({x}_{n+\frac{1}{2}}\big)^{\intercal}\hat{v}_{n+\frac{1}{2}}
			+A\big({x}_{n+\frac{1}{2}}\big)
			+\frac{h}{2}F\big({x}_{n+\frac{1}{2}}\big),
		\end{aligned}
	\end{equation} 	
	with $A'(x)=(\partial_{j} A_{i}(x))_{i,j=1}^{3}$. By Theorem~6.1 of \cite{hairer2006}, the method \eqref{Met-p} is symplectic.
	Moreover, exchanging $x_n \leftrightarrow x_{n+1}$ and $h \leftrightarrow -h$
	leaves \eqref{Met-p} invariant, and hence the method is symmetric.
	This symmetry is also known to play an important role in favorable long-time behavior  \cite{hairer2006}.
	\item \textbf{Step 2: Filtered modification.} To capture the high oscillations in the velocity more accurately, a modification
	is needed. Reformulating the first equation of \eqref{Met-p} and incorporating a filter matrix $\Psi$ leads to
	$
	\hat{v}_{n+\frac{1}{2}}= \Psi\big(p_{n}+\frac{h}{2} A'\big({x}_{n+\frac{1}{2}}\big)^{\intercal}\hat{v}_{n+\frac{1}{2}} - A\big({x}_{n+\frac{1}{2}}\big)
	+\frac{h}{2} F\big({x}_{n+\frac{1}{2}}\big)\big).
	$
	The form of the filter matrix $\Psi$ will be specified in the subsequent part. Using the notation of $\hat{v}_{n+\frac{1}{2}}$, it follows that 
	\begin{equation}\label{onesx}
		x_{n+1}
		=x_n+h\hat{v}_{n+\frac{1}{2}}
		=x_n+h\Psi\Big(p_n
		+\frac{1}{2}A'\big({x}_{n+\frac{1}{2}}\big)^{\intercal}(x_{n+1}-x_{n})
		-A\big({x}_{n+\frac{1}{2}}\big)	+\frac{h}{2}F\big({x}_{n+\frac{1}{2}}\big)\Big).
	\end{equation}
	With the symmetry property, we arrive at $x_{n-1}
	=x_n-h\Psi\big(p_n
	+\frac{1}{2}A'\big({x}_{n-\frac{1}{2}}\big)^{\intercal}(x_{n-1}-x_{n})
	-A\big({x}_{n-\frac{1}{2}}\big)	-\frac{h}{2}F\big({x}_{n-\frac{1}{2}}\big)\big),$
	where $x_{n - 1/2}:=(x_{n}+x_{n-1})/2$.	Combining these two relations, we obtain a filtered two-step symmetric formulation
	\begin{equation}\label{variational integrator-ori}
		\begin{aligned}
			x_{n+1}-2x_{n}+x_{n-1}
			=&\Psi\Big(\frac{h}{2}A'(x_{n+\frac{1}{2}})^{\intercal}(x_{n+1}-x_{n})
			+\frac{h}{2}A'(x_{n-\frac{1}{2}})^{\intercal}(x_{n}-x_{n-1})-h\big(A(x_{n+\frac{1}{2}})-A(x_{n-\frac{1}{2}})\big)\\
			&+\frac{h^2}{2}\big( F(x_{n+\frac{1}{2}})+F(x_{n-\frac{1}{2}}) \big)\Big).
		\end{aligned}
	\end{equation}
	Together with the position update \eqref{variational integrator-ori}, we introduce the following filtered approximation for the velocity:
	\begin{equation*}\label{}
		\frac{v_{n+1}+v_n}{2}
		=
		\Phi\frac{x_{n+1}-x_n}{h}.
	\end{equation*}
	The choice of the filter matrices $\Psi$ and $\Phi$ will be detailed in the following analysis. 
	\item \textbf{Step 3: Choice of the filter functions.}  We introduce the following filter functions:
\begin{equation*}\label{filtered}
\psi(\zeta)=\operatorname{tanch}(\zeta/ 2)=\frac{\tanh(\zeta/2)}{\zeta/2},\qquad
\phi(\zeta)=\frac{1}{\operatorname{sinch}(\zeta/2)}=\frac{\zeta/2}{\sinh(\zeta/2)},
\end{equation*}
with the understanding that $\psi(0)=\phi(0)=1$. The specific choice of these filter functions is motivated by the theoretical analysis required for the strong magnetic field regime. In particular, Theorem~\ref{error bound 2} in Section~\ref{sec3} demonstrates that $\psi$ is necessary to achieve a second-order error bound in position, while $\phi$ plays a crucial role in the long-time near-conservation of energy and magnetic moment, as established in Theorems~\ref{conservation-H2} and~\ref{conservation-M2}. The detailed analysis is presented in Section \ref{sec5}.
 Subsequently, we define the filtered matrices
\begin{equation}\label{MVF-phi}
\Psi := \psi\Big(-\frac{h}{\varepsilon} \tilde{B}_0\Big)
= I + \Big(1-\operatorname{tanc}\Big(\frac{h}{2 \varepsilon}\Big)\Big){\tilde{B}_0}^2, \
\Phi := \phi\Big(-\frac{h}{\varepsilon} \tilde{B}_0\Big)
= I + \Big(1-\operatorname{sinc}\Big(\frac{h}{2 \varepsilon}\Big)^{-1}\Big){\tilde{B}_0}^2,
\end{equation}
where $\operatorname{tanc}(\zeta)=\tan (\zeta)/ \zeta$ and $\operatorname{sinc}(\zeta)=\sin (\zeta)/ \zeta$. Here, $\tilde{B}_0$ denotes the skew-symmetric matrix defined by $-\tilde{B}_0 v = v\times B_0$ for $v\in\mathbb{R}^3$. The identities above follow from a Rodrigues-type formula (see \cite{Lubich2020}).
\end{itemize}  

For clarity, the filtered two-step variational integrator constructed above is summarized in the following algorithm.
\begin{algodef}\label{defM}
	A filtered two-step variational integrator for
	solving charged-particle dynamics \eqref{charged-particle} is
	defined as
	\begin{equation}\label{Method}
		\begin{aligned}
			x_{n+1}-2x_{n}+x_{n-1}
			&=\Psi\Big(\frac{h}{2}A'(x_{n+\frac{1}{2}})^{\intercal}(x_{n+1}-x_{n})
			+\frac{h}{2}A'(x_{n-\frac{1}{2}})^{\intercal}(x_{n}-x_{n-1})\\
			&\qquad \ \  -h\big(A(x_{n+\frac{1}{2}})-A(x_{n-\frac{1}{2}})\big)
			+\frac{h^2}{2}\big( F(x_{n+\frac{1}{2}})+F(x_{n-\frac{1}{2}}) \big)	\Big),
		\end{aligned}
	\end{equation}
	where $h$ is the chosen stepsize, $x_0=x(0)$ and $x_{n \pm 1/2}:=(x_n + x_{n\pm 1})/2$.
	
An approximation for the velocity $v(t):=\dot{x}(t)$ is formulated as
\begin{equation}\label{AV1}
\frac{v_{n+1}+v_n}{2}
=
\Phi\frac{x_{n+1}-x_n}{h}.
\end{equation}
We shall refer to this algorithm as FVI.
\end{algodef}

For the implementation of FVI, it is noticed that the starting
value $x_1$ is needed. To get this result, letting $n=0$ in \eqref{onesx} yields 
\begin{equation}\label{start_x}
	\begin{aligned}
		x_{1}
		=x_0+h\Psi\Big(p_0
		+\frac{1}{2}A'\big(\frac{x_0+ x_{1}}{2}\big)^{\intercal}(x_{1}-x_{0})
		-A\big(\frac{x_0+ x_{1}}{2}\big)	+\frac{h}{2}F\big(\frac{x_0+ x_{1}}{2}\big)\Big),
	\end{aligned}
\end{equation}
where $p_0$ is obtained by considering $p(t)=v(t)+A(x(t))$ at $t=0$. 

Since the two-step scheme \eqref{Method} is implicit, an iteration is required to compute $x_{n+1}$.
For practical implementation,  the $\varepsilon^{-1}$-dependent terms involving $x_{n+1}$ are collected on the left-hand side of \eqref{Method}.  The resulting equation for $x_{n+1}$ is then solved by a fixed-point iteration, which reduces the influence of high oscillations in the numerical implementation.  Based on the implicit two-step formulation \eqref{Method}, the complete implementation of the filtered variational integrator is summarized in the flowchart below.
\begin{algorithm}[h!]
		\refstepcounter{algorithm}%
		\label{alg:FVI}%
		\text{Flowchart of the Filtered Variational Integrator (FVI) for Charged-Particle Dynamics}\par
		\vspace{2pt}\hrule\vspace{4pt}
		\begin{algorithmic}[1]
		\setlength{\baselineskip}{0.9\baselineskip}
			\Require Initial values $x_0$, $v_0$;  step size $h$; parameter $\varepsilon$; end time $T$.
			\Ensure Trajectory $\{x_n\}_{n=1}^{T/h}$ and velocity $\{v_n\}_{n=1}^{T/h}$.
			\State $\Psi=I+\big(1-\mathrm{tanc}\big(\frac{h}{2 \varepsilon}\big)\big){\tilde{B}_0}^2,\ \ \
			\Phi=I+\big(1-\mathrm{sinc}\big(\frac{h}{2 \varepsilon}\big)^{-1}\big){\tilde{B}_0}^2.$
			\Statex \textbf{Implicit solve for $x_1$:}
			\State  $p_0=v_0+A(x_0)$, $M=\big(I+(h/(2\varepsilon))\Psi \tilde{B}_0\big)^{-1}$.
			\State  $x_{1}^{(0)}=M\big(x_0+h\Psi\big(p_0-A_1(x_0)+hF(x_0)/2\big)\big).$
			\State $x_{1}^{(1)}=M\big(x_0+h\Psi\big(p_0+A_1'((x_{1}^{(0)}+x_0)/2)(x_{1}^{(0)}
			-x_0)/2-A_1((x_{1}^{(0)}+x_0)/2)+hF((x_{1}^{(0)}+x_0)/2)/2\big)\big).$
			\MyWhile{$|x_{1}^{(0)}-x_{1}^{(1)}|>10^{-16}$ \textbf{and} $m<50$}
			\State $x_{1}^{(2)}
				=M\big(x_0+h\Psi\big(p_0+A_1'((x_{1}^{(1)}+x_0)/2)(x_{1}^{(1)}-x_0)/2-A_1((x_{1}^{(1)}+x_0)/2)$
			\Statex $\qquad \qquad+hF((x_{1}^{(1)}+x_0)/2)/2\big)\big).$
			\State $x_{1}^{(0)}=x_{1}^{(1)}$; $x_{1}^{(1)}=x_{1}^{(2)}$.
			\State $m=m+1$.
			\MyEndWhile
			\State	$x_1=x_{1}^{(2)}$.
			\Statex \textbf{Implicit solve for $x_{n+1}$:}
			\MyFor{$n = 1$ \textbf{to} $T/h-1$}
			\State	$
			x_{n+1}^{(0)}=M\big(2x_{n}-x_{n-1}+h\Psi\big(\tilde{B}_0x_{n-1}/(2\varepsilon)+A_1'((x_{n}+x_{n-1})/2)(x_{n}-x_{n-1})/2-A_1(x_{n})$
			\Statex $\qquad\qquad\ \ +A_1((x_{n}+x_{n-1})/2)+hF(x_{n})/2+hF((x_{n}+x_{n-1})/2)/2\big)\big).$
			\State $x_{n+1}^{(1)}=M\big(2x_{n}-x_{n-1}+h\Psi\big(\tilde{B}_0x_{n-1}/(2\varepsilon)+A_1'((x_{n+1}^{(0)}+x_{n})/2)(x_{n+1}^{(0)}-x_{n})/2
			$
			\Statex $\qquad\qquad \ \ \  +A_1'((x_{n}+x_{n-1})/2)(x_{n}-x_{n-1})/2-A_1((x_{n+1}^{(0)}+x_{n})/2)+A_1((x_{n}+x_{n-1})/2))$
			\Statex $\qquad\qquad \ \ \ +hF((x_{n+1}^{(0)}+x_{n})/2)/2+hF((x_{n}+x_{n-1})/2)/2\big)\big).$
			\State $m=0.$
			\MyWhile{$|	x_{n+1}^{(0)}-x_{n+1}^{(1)}|>10^{-16}$ \textbf{and} $m<50$}
			\State $x_{n+1}^{(2)}=M\big(2x_{n}-x_{n-1}+h\Psi\big(\tilde{B}_0x_{n-1}/(2\varepsilon)+A_1'((x_{n+1}^{(1)}+x_{n})/2)(x_{n+1}^{(1)}-x_{n})/2$
			\Statex $\qquad \qquad\quad \ \ \ \  +A_1'((x_{n}+x_{n-1})/2)(x_{n}-x_{n-1})/2-A_1((x_{n+1}^{(1)}+x_{n})/2)+A_1((x_{n}+x_{n-1})/2)$
			\Statex $\qquad \qquad\quad \ \ \ \ +hF((x_{n+1}^{(1)}+x_{n})/2)/2 +hF((x_{n}+x_{n-1})/2)/2\big)\big).$
			\State $x_{n+1}^{(0)}=x_{n+1}^{(1)}$; $x_{n+1}^{(1)}=x_{n+1}^{(2)}$.
			\State $m=m+1.$
			\MyEndWhile
			\State	$x_{n+1}=x_{n+1}^{(2)}.$
			\State  $v_{n+1}=2(\Phi(x_{n+1}-x_{n})/h)-v_{n}$.
			\State $x_{n-1}=x_{n}$; $x_{n}=x_{n+1}$;$v_{n}=v_{n+1}$.
			\MyEndFor
	\end{algorithmic}
\end{algorithm}

\section{Main results  and numerical experiments} \label{sec3}
This section presents the main results on error bounds and long-time near-conservation of the filtered two-step variational integrator  \eqref{Method}--\eqref{AV1}. The first part gives the results for solving the system in a moderate magnetic field  $\varepsilon = 1$,  and the second part focuses on the case of a strong magnetic field   $0<\varepsilon \ll 1$. Two numerical experiments are performed to clarify these results in each part of this section.  
\subsection{Theoretical and numerical results in a moderate magnetic field ($\varepsilon=1$)} 
In what follows, we first provide the theoretical results of the filtered two-step variational integrator \eqref{Method}--\eqref{AV1}  applied to  the CPD \eqref{charged-particle}, including  error bounds and long-time near-conservation properties. Subsequently, numerical experiments are presented to illustrate and support these theoretical findings in a moderate magnetic field regime.
\subsubsection{Theoretical results in a moderate magnetic field} \label{sec3.1}
In this subsection, we shall study the error bounds and  long-time energy and momentum near-conservation of FVI in a moderate magnetic field. The error bounds are given firstly in the following theorem.
\begin{mytheo}\textbf{(Error bounds)}\label{error bound 1}
It is assumed that \eqref{charged-particle} has sufficiently smooth solutions, and the functions $A(x)$ and $F(x)$ are sufficiently differentiable. Moreover, we assume that $A(x)$ and $F(x)$ are locally Lipschitz continuous with Lipschitz constants $L$. There exists a  constant $h_0>0$, such that  if the stepsize $h$ satisfies $ h \leq h_0$ and 
$\abs{\cos(h/2)}\ge c>0$, the global errors can be estimated as
	\begin{equation*}
		\abs{x_{n+1}-x(t_{n+1})} \leq Ch^2, \ \quad	\abs{v_{n+1}-v(t_{n+1})} \leq Ch^2 \quad \text{for} \quad nh\leq T,
	\end{equation*}	
where
$C>0$ is a generic constant independent of $h$  or $n$  but depends on  $L$ and  $T$.
\end{mytheo}
The good long-time  behavior of energy  along the numerical solution of the filtered two-step variational integrator \eqref{Method}--\eqref{AV1} is shown by the following theorem.
\begin{mytheo}\textbf{(Energy near-conservation)}\label{conservation-H1}
Assume that the numerical solution \eqref{Method}--\eqref{AV1}
belongs to a compact set that is independent of $h$. Then the energy is nearly preserved along the filtered two-step variational integrator as follows:
	\begin{equation*}
		\abs{H(x_{n+1/2},v_{n+1/2})-H(x_{1/2},v_{1/2})} \leq Ch^2 \quad \  \textmd{for} \quad nh \leq ch^{-N+2},
	\end{equation*}	
where $N \geq 3$ is the truncation number, constant $c$  is independent of $h$ and the  $C$ is a constant that is independent of $n$ and $h$ as long as $nh \leq ch^{-N+2}$.
\end{mytheo}

In addition to energy, the CPD \eqref{charged-particle} also possesses a conserved momentum when the scalar and vector potentials satisfy certain invariance conditions. Specifically, if
\begin{equation}\label{U-A}
	U(e^{\tau S}x) = U(x) \quad \textmd{and} \quad e^{-\tau S}A(e^{\tau S}x) = A(x) \quad \textmd{for all real} \ \tau,
\end{equation}
where $S$ is a skew-symmetric matrix, then the momentum is exactly conserved along the solution of the system \eqref{charged-particle}. 
The following theorem shows that the filtered two-step variational integrator \eqref{Method}--\eqref{AV1} exhibits near-conservation of the momentum over long times.
\begin{mytheo}(\textbf{Momentum near-conservation})\label{conservation-M1}
Suppose that the numerical solution \eqref{Method}--\eqref{AV1}  remains in a compact set independent of $h$. Under the  invariance conditions \eqref{U-A}, we further assume that the skew-symmetric matrix $S$ satisfies $S\tilde{B}_{0}^2=\tilde{B}_{0}^{2}S$. Then the numerical solution satisfies a near-conservation of the momentum:
\begin{equation*}
	\abs{M(x_{n+1/2},v_{n+1/2})-M(x_{1/2},v_{1/2})} \leq Ch^2 \quad \ \textmd{for} \quad nh \leq ch^{-N+2},
\end{equation*}	
where $N \geq 3$ is a truncation, $c$ is a constant independent of $h$ and $C$ is a constant independent of $n$ and $h$ with $nh \leq ch^{-N+2}$.
\end{mytheo}

\subsubsection{Numerical experiments in a moderate magnetic field}  \label{sec3.1.1}
In this section, we carry out two numerical experiments to demonstrate the advantages of the proposed method.
The schemes used for comparison include the Boris method (BORIS) \cite{Boris}, the second-order two-step symmetric method (TSM) \cite{wangwu2020}, the filtered variational method (FVARM) \cite{Lubich2022}, and the filtered two-step variational integrator (FVI) introduced in Section~\ref{sec2}.

To test the performance of all the methods, we compute the global errors:
$error_{x}:=\lvert x_{n}-x(t_n) \rvert /\lvert x(t_n)\rvert $, $error_{v}:=\lvert v^{n}-v(t_n)\rvert/\lvert v(t_n)\rvert$, $error_{v_{\parallel}}:= \lvert v_{\parallel}^{n}-v_{\parallel}(t_n)\rvert/\lvert v_{\parallel}(t_n)\rvert$, $error_{v_{\perp}}:=\lvert v_{\perp}^{n}-v_{\perp}(t_n)\rvert/\lvert v_{\perp}(t_n)\rvert$,
the energy  error
$e_{H}:=H(x_{n+1/2},v_{n+1/2})-H(x_{1/2},v_{1/2})$,
the momentum error $e_{M}:=M(x_{n+1/2},v_{n+1/2})-M(x_{1/2},v_{1/2})$
and the magnetic moment  error
$e_{I}:=I(x_{n+1/2},v_{n+1/2})-I(x_{1/2},v_{1/2})$
in the numerical experiment. The reference solution is obtained by using ``ode45" of MATLAB. For implicit methods, we choose fixed-point iteration and set $ 10^{-16} $ as the
error tolerance  and 50  as the maximum number of  each iteration.

\noindent\vskip3mm \noindent{Problem 1. \textbf{(Moderate magnetic field with invariance  conditions \eqref{U-A})}}
 For the charged-particle dynamics \eqref{charged-particle},  we consider the magnetic field ${B}(x)=\nabla_x \times A(x)=\big(
0,0,(x_{1}^2+x_{2}^2)^{\frac{1}{2}}\big)^{\intercal}$, where 
$A(x)=\big(-x_{2}(x_{1}^2+x_{2}^2)^{\frac{1}{2}}/3,x_{1}(x_{1}^2+x_{2}^2)^{\frac{1}{2}}/3,0\big)^{\intercal}$. The scalar potential is given by $U(x)=(x_{1}^2+x_{2}^2)^{-\frac{1}{2}}/100$ and the momentum reads $M(x,v)=\big(v_{1}-x_2(x_{1}^2+x_{2}^2)^{\frac{1}{2}}/3\big)x_2-\big(v_{2}+x_1(x_{1}^2+x_{2}^2)^{\frac{1}{2}}/3\big)x_1.$
We choose the initial values $x(0)=(0,1,0.1)^{\intercal}$
 and $v(0)=(0.09,0.05,0.2)^{\intercal}$.
The problem is solved on $[0, 1]$ with stepsizes
$h=1/2^{k}$ for $k=1,\ldots,8$ to compute the global errors shown in Figure~\ref{fig:problem11}.
To illustrate the near-conservation of the energy $H$ and the momentum $M$, we further integrate the system on $[0,10000]$ with different stepsizes;
the results are displayed in Figure~\ref{fig:problem12}.

\begin{figure}[H]
	\centering\tabcolsep=0.5mm
	\begin{tabular}
		[c]{cc}
		\includegraphics[width=3.8cm,height=3.2cm]{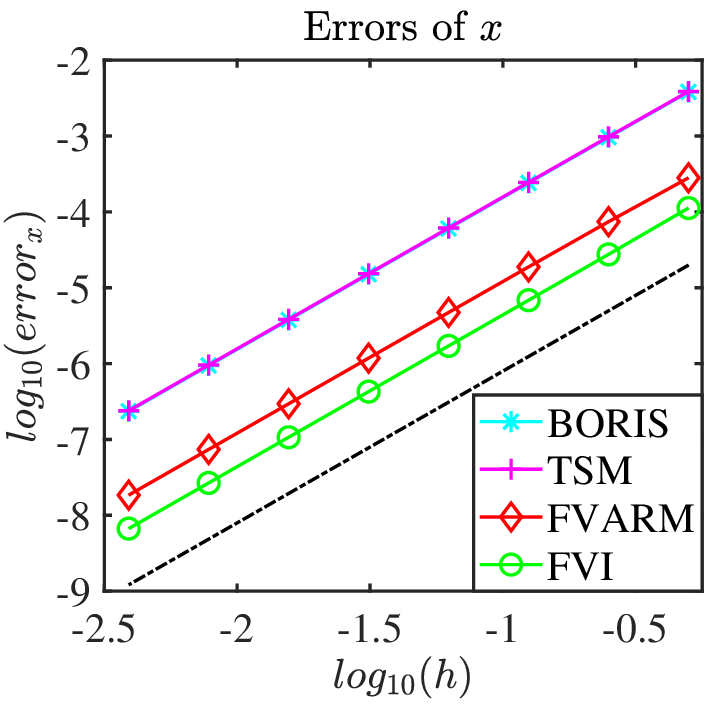}  \qquad
		\includegraphics[width=3.8cm,height=3.2cm]{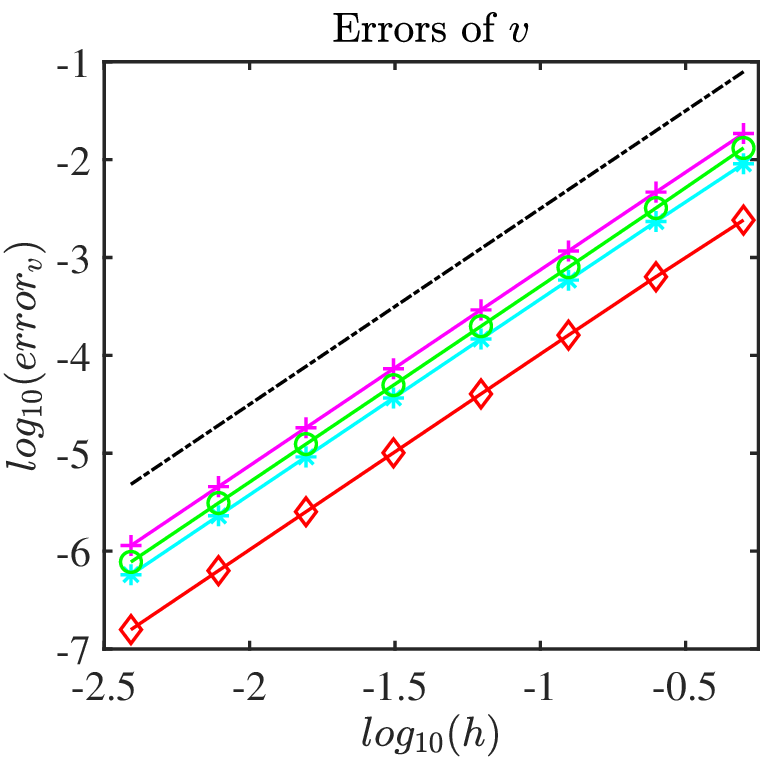}
	\end{tabular}
\caption{Problem 1.
The global errors $error_x$ and $error_v$ with $t=1$ and $h=1/2^{k}$ for $k=1,\ldots,8$ (the dash-dot line is slope two).}
	\label{fig:problem11}
\end{figure}

\begin{figure}[H]
	\centering\tabcolsep=0.5mm
	\begin{tabular}
		[c]{cc}
		\includegraphics[width=5.8cm,height=2.75cm]{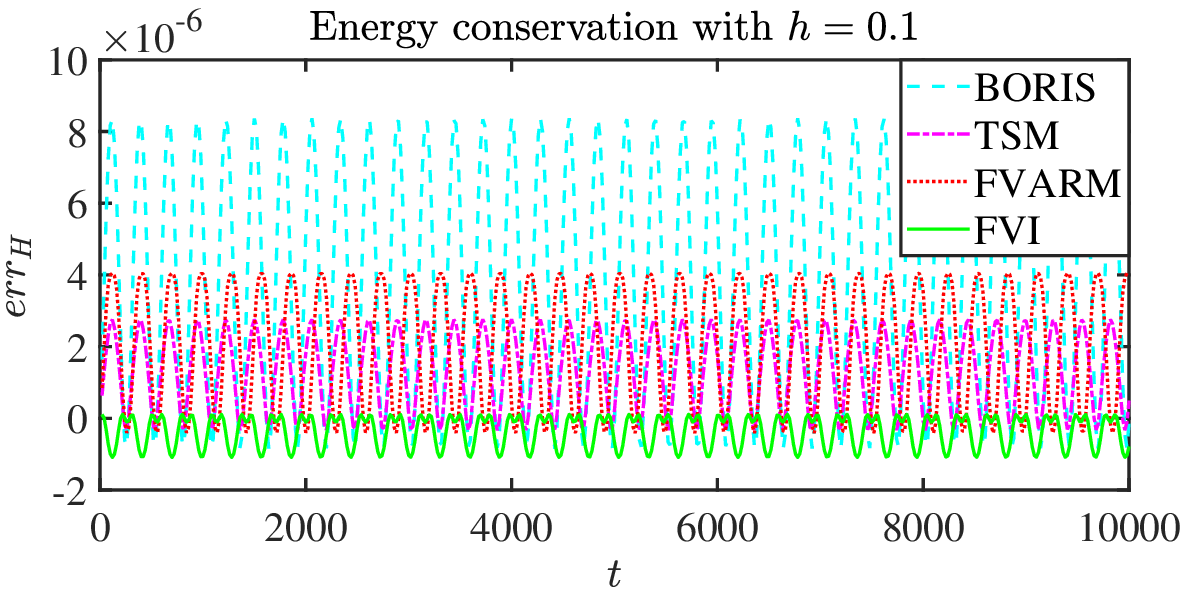}\ \
		\includegraphics[width=5.8cm,height=2.75cm]{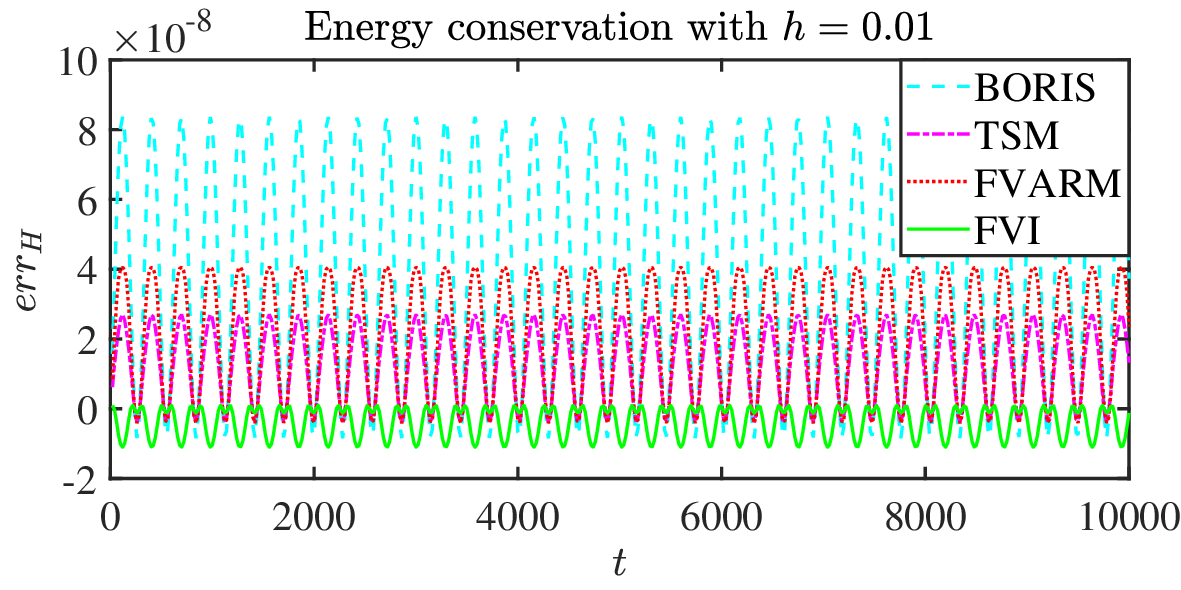}\\
        \includegraphics[width=5.8cm,height=2.75cm]{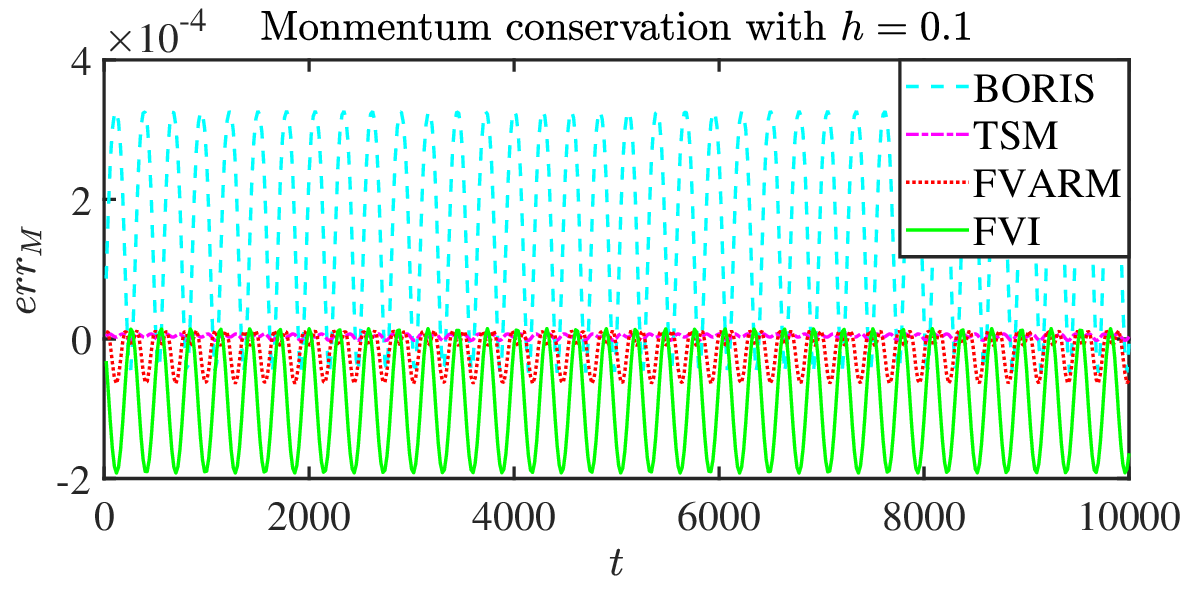}\ \
\includegraphics[width=5.8cm,height=2.75cm]{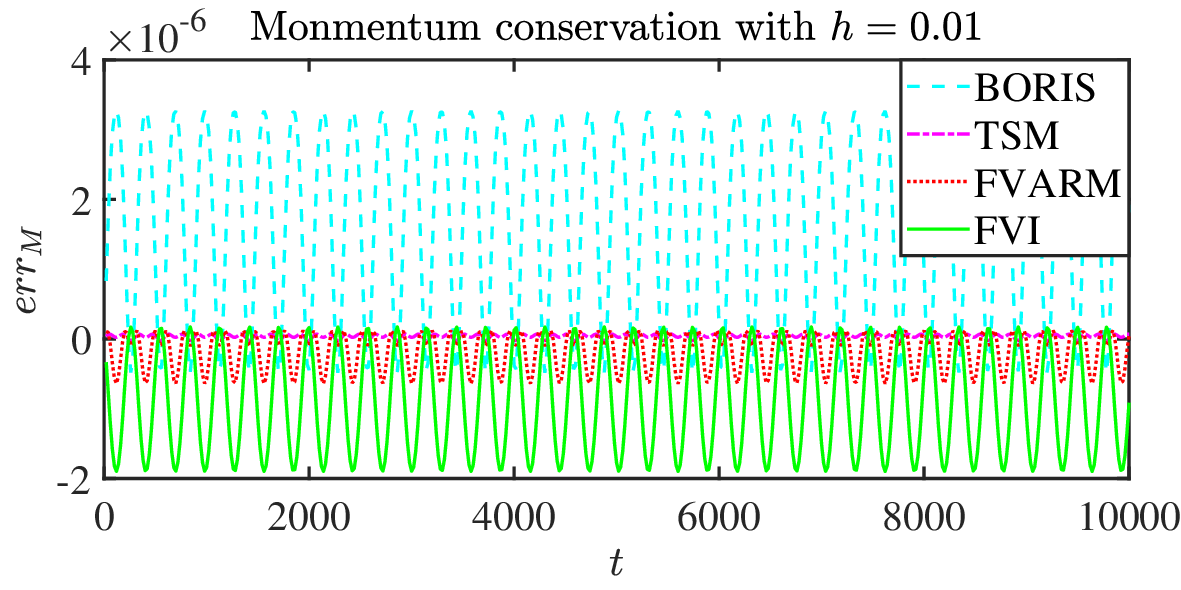}
	\end{tabular}
	\caption{Problem 1. Evolution of energy error $e_H$  and  momentum error $e_{M}$ with different step sizes $h$.}
	\label{fig:problem12}
\end{figure}

\noindent\vskip3mm \noindent{Problem 2. \textbf{(Moderate magnetic field without invariance  conditions \eqref{U-A})}} This problem concerns the motion of a charged particle under the magnetic field
$B(x)=\nabla_x \times \big(x_{3}^2- x_{2}^2-x_2,x_{3}^2- x_{1}^2+x_1,x_{2}^2- x_{1}^2\big)^{\intercal}/2=(x_{2}- x_{3},x_{1}+x_3,x_{2}- x_{1}+1)^{\intercal}.$
The scalar potential and the momentum are given by $U(x)=x_{1}^2+2x_{2}^2+3x_{3}^2-x_{1}$ and $M(x,v)=(v_{1}+(x_{3}^2- x_{2}^2-x_2)/2)x_2-(v_{2}+(x_{3}^2- x_{1}^2+x_1)/2)x_1$. 
The initial values are  $x(0)=(0,0.1,0.5)^{\intercal}$
and $v(0)=(0.02,0.1,0.7)^{\intercal}$. The global errors and the near-conservation of energy and momentum are presented in Figures \ref{fig:problem21} and \ref{fig:problem22}, respectively. Note that the scalar and vector potentials do not satisfy the invariance conditions \eqref{U-A}. This case is designed to examine how the proposed integrator captures the long-time momentum behavior in the absence of these conditions.
\begin{figure}[h!]
	\centering\tabcolsep=0.5mm
	\begin{tabular}
		[c]{cc}
		\includegraphics[width=3.8cm,height=3.2cm]{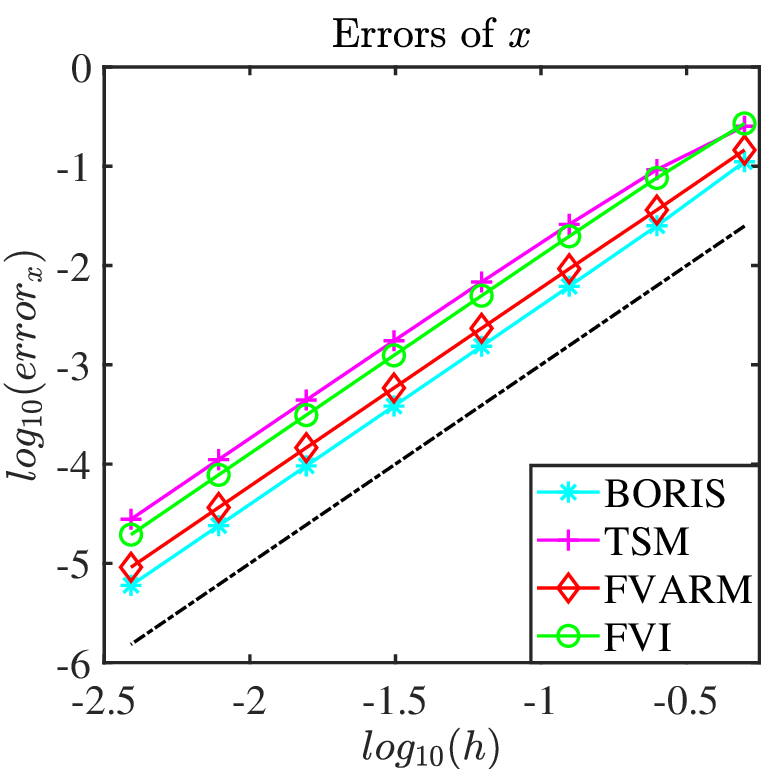}  \qquad
		\includegraphics[width=3.8cm,height=3.2cm]{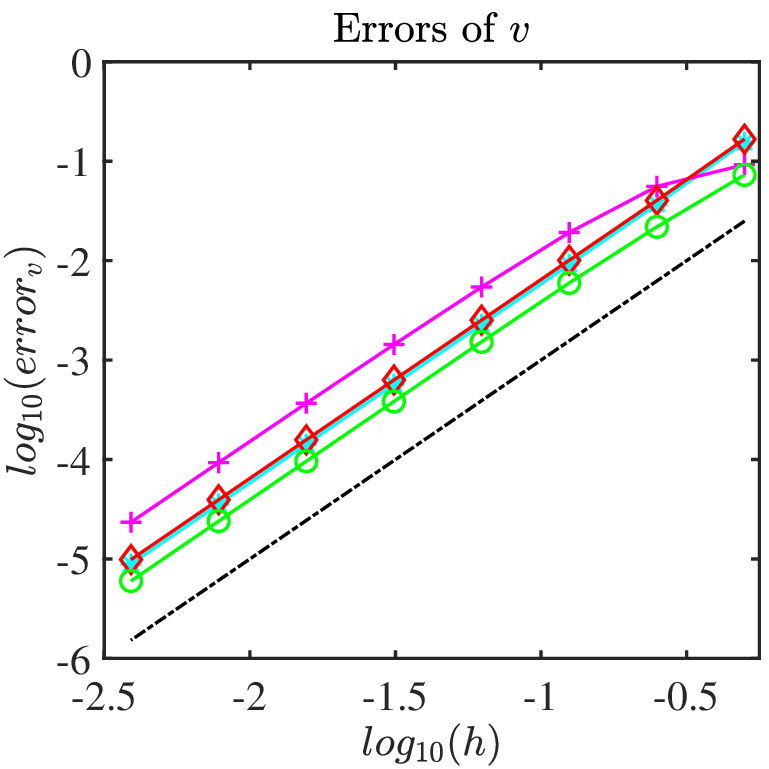}
	\end{tabular}
	\caption{ Problem 2.
	The global errors $error_x$ and $error_v$ with $t=1$ and $h=1/2^{k}$ for $k=1,\ldots,8$ (the dash-dot line is slope two).}
	\label{fig:problem21}
\end{figure}

\begin{figure}[h!]
	\centering\tabcolsep=0.5mm
	\begin{tabular}
		[c]{cc}
		\includegraphics[width=5.8cm,height=2.8cm]{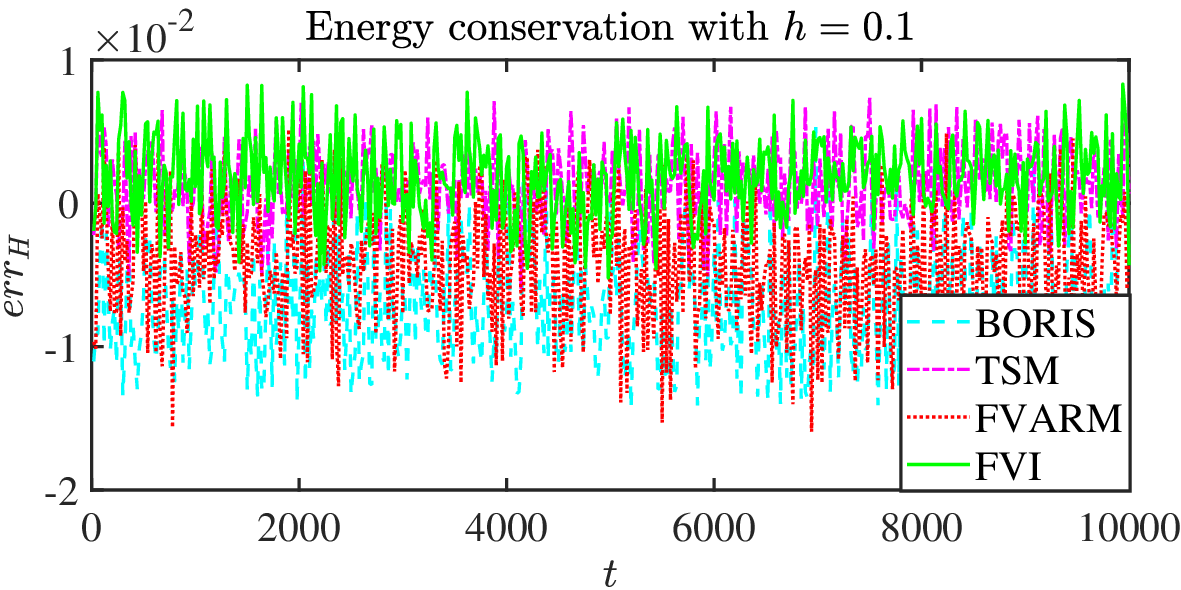}\ \
		\includegraphics[width=5.8cm,height=2.8cm]{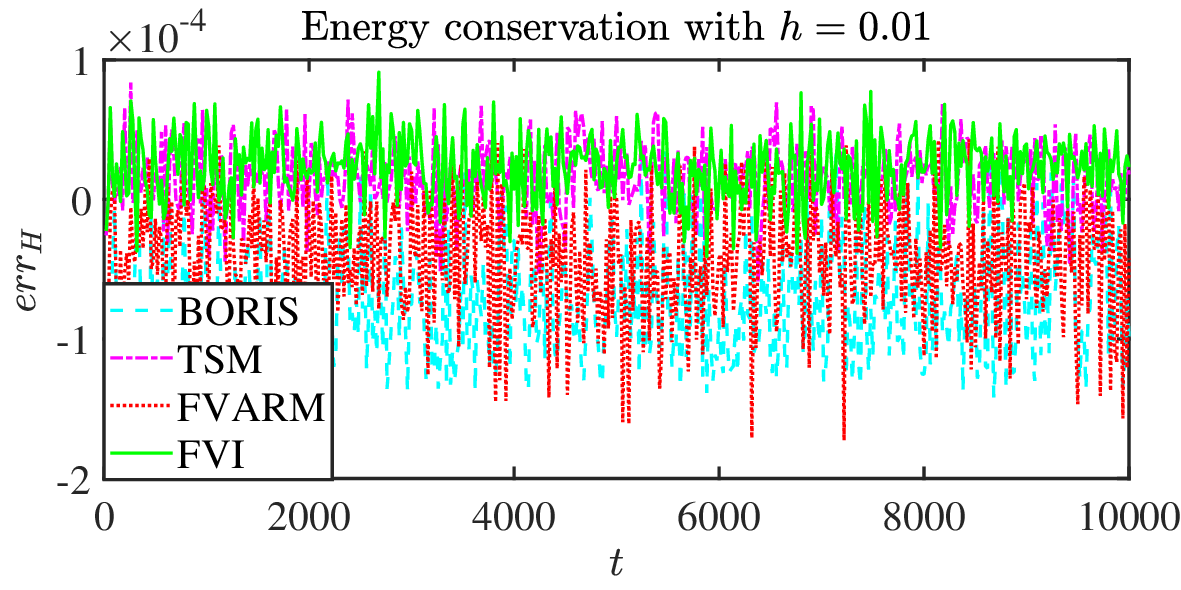}\\
        \includegraphics[width=5.8cm,height=2.8cm]{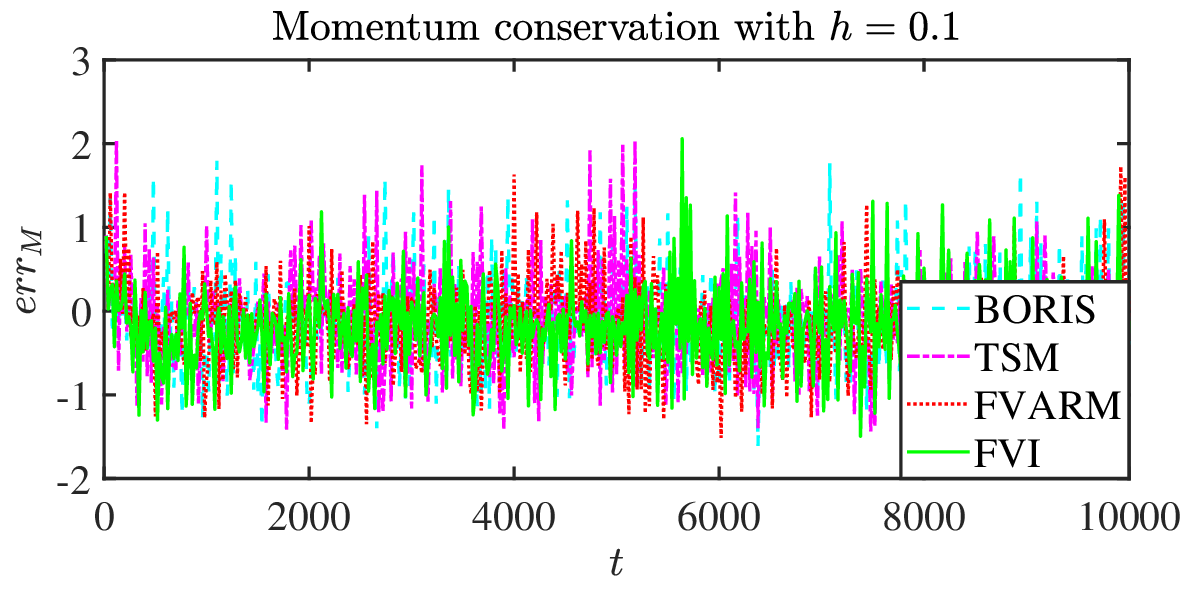}  \ \
\includegraphics[width=5.8cm,height=2.8cm]{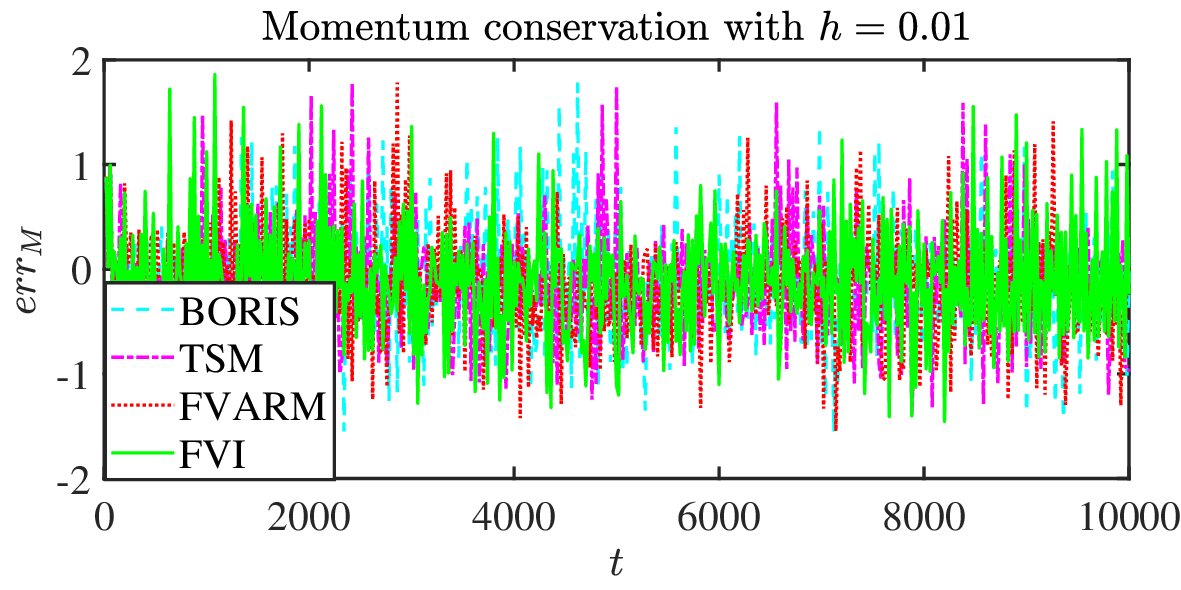}
	\end{tabular}
	\caption{Problem 2. 
   Evolution of energy error $e_H$ and  momentum error $e_{M}$ with different step sizes $h$.}
	\label{fig:problem22}
\end{figure}

The numerical results shown in Figures \ref{fig:problem11}--\ref{fig:problem22} lead to the following observations:

\textbf{Order behavior.}
Figures \ref{fig:problem11} and \ref{fig:problem21} show the global errors and it can be seen from the results
that the global error lines  are nearly parallel to the line with slope 2, which clarifies that the filtered two-step variational integrator has second-order accuracy in the position and velocity. 

\textbf{Energy and momentum behavior.} Figures \ref{fig:problem12} (top row) and \ref{fig:problem22} (top row) demonstrate the favorable long term behavior of energy $H$ along the numerical solution obtained by FVI. The results in Figure \ref{fig:problem12} (bottom row) illustrate that the momentum $M$ is nearly preserved by FVI over long times when the invariance  conditions \eqref{U-A} are satisfied. However, as shown in Figure \ref{fig:problem22} (bottom row), the momentum is no longer conserved when the scalar and vector potentials do not satisfy the invariance  conditions \eqref{U-A}. 

\subsection{Theoretical and numerical results in a strong magnetic field ($0<\varepsilon\ll 1$)} 
To investigate the performance of the proposed method in the strong magnetic field regime, the analysis is organized into two parts: the first establishes the theoretical properties, while the second presents numerical experiments that verify the theoretical analysis.
\subsubsection{Theoretical results in a strong magnetic field} \label{sec3.2}
This subsection gives the results of the filtered two-step variational integrator \eqref{Method}--\eqref{AV1}  for solving the CPD \eqref{charged-particle} in a strong magnetic field. Before establishing the error bounds of the filtered two-step variational integrator, we present a priori bounds for the exact solution of system \eqref{charged-particle}. For the exact velocity $v(t)=\dot x(t)$, we define 
$$
	v_\parallel(t)
	=
	\frac{B_0}{|B_0|}
	\left(
	\frac{B_0}{|B_0|}\cdot v(t)
	\right):=P_{\parallel}v(t),
	\quad \
	v_\perp(t)=v(t)-v_\parallel(t).
$$
With this notation, the following result holds. 
\begin{mytheo}
Assume that the electric field $F(x)$ 
is globally Lipschitz continuous.
If the initial value satisfy
$ 
|x(0)| + |\dot{x}(0)| \le C,
$ 
then the position $x(t)$ and the parallel velocity $v_{\parallel}(t)$, 
together with their time derivatives, are uniformly bounded
with respect to $\varepsilon$, i.e.,
\begin{equation}\label{eaxct-bound1}
	\abs{x(t)}+\abs{\dot{x}(t)} \leq C, \qquad \quad  \text{for all} \quad t\in (0,T],
\end{equation}	
and 
\begin{equation}\label{eaxctvd}
\abs{{v}_{\parallel}(t)}+\abs{\dot{v}_{\parallel}(t)} \leq C, \qquad \text{for all} \quad t\in (0,T],
\end{equation}	
where $C$ is  independent of $\varepsilon$.
\end{mytheo}

\begin{proof}
We first prove the estimate \eqref{eaxct-bound1}, following arguments similar to those in~\cite{VP1}. The system \eqref{charged-particle} can be rewritten as
\begin{equation}\label{charged-particle-2}
\dot{x}(t)=v(t), \quad \ \dot{v}(t)=v(t) \times  B(x(t)) +F(x(t)).
\end{equation}	
Taking the inner product of both sides of \eqref{charged-particle-2}
with $x(t)$ and $v(t)$, respectively, and applying the Cauchy--Schwarz
inequality, we obtain
	\begin{equation*}
\frac{\mathrm{d}}{\mathrm{d}t} \abs{x(t)}^2 \leq 2\abs{x(t)}\abs{v(t)}, \quad \ \frac{\mathrm{d}}{\mathrm{d}t} \abs{v(t)}^2 \leq 2\abs{v(t)}\abs{F(x(t))}.
	\end{equation*}	
The above estimates imply
	\begin{equation*}
	\frac{\mathrm{d}}{\mathrm{d}t} \abs{x(t)} \leq 2\abs{v(t)}, \quad \ \frac{\mathrm{d}}{\mathrm{d}t} \abs{v(t)} \leq 2\abs{F(x(t))}\leq 2\abs{F(x(0))}+2C\abs{x(t)}+2C\abs{x(0)}.
\end{equation*}	
Together with the initial value $ |x(0)| + |\dot{x}(0)| \le C$,  an application of Gronwall's inequality to the above estimates yields \eqref{eaxct-bound1}.

We now turn to the proof of \eqref{eaxctvd}. We only prove $|v_{\parallel}(t)|\le C$, as the estimate for $|\dot v_{\parallel}(t)|$ is analogous. To this end, we introduce
	\begin{equation*}
	\begin{aligned}
&\check{v}_\parallel(t)
=
\frac{B_0+\varepsilon B_{1}(x(t))}{|B_0|}
\Big(
\frac{B_0+\varepsilon B_{1}(x(t))}{|B_0|}\cdot v(t)
\Big):=\check{P}_{\parallel}(x(t))v(t).
\end{aligned}
\end{equation*}	
It can be verified that $\dot{\check{P}}_{\parallel}(x(t))=\mathcal{O}(\varepsilon)$ and $\ddot{\check{P}}_{\parallel}(x(t))=\mathcal{O}(1)$, where we use the notation $\dot{\check{P}}_{\parallel}(x(t))=\frac{\mathrm{d}}{\mathrm{d}t}{\check{P}}_{\parallel}(x(t))$ and similar for $\ddot{\check{P}}_{\parallel}(x(t))$.
Differentiating $\check v_\parallel(t)$ yields
\begin{equation}\label{oridothat_v}
 \dot{\check{v}}_\parallel(t)
		=\dot{\check{P}}_{\parallel}(x(t))v(t)+\check{P}_{\parallel}(x(t))\dot{v}(t).
\end{equation}	
Substituting the expression of $\dot v(t)$ into~\eqref{oridothat_v}, we obtain
\begin{equation*}\label{dothat_v}
	\begin{aligned}
\dot{\check{v}}_\parallel(t)
&=\dot{\check{P}}_{\parallel}(x(t))v(t)+\check{P}_{\parallel}(x(t))\big(v(t) \times  B(x(t)) +F(x(t))\big)
=\dot{\check{P}}_{\parallel}(x(t))v(t)+\check{P}_{\parallel}(x(t))F(x(t)\\
&=\dot{\check{P}}_{\parallel}(x(t))\check{v}_\parallel(t)+\dot{\check{P}}_{\parallel}(x(t))\check{v}_\perp(t)+\check{P}_{\parallel}(x(t))F(x(t)),
\end{aligned}
\end{equation*}	
where we have used the identity $\check{P}_{\parallel}(x(t))(z\times (B_0+\varepsilon B_{1}(x(t))))=0$ for  all vectors $z$. Taking the inner product of both sides of the above formula
with ${\check{v}}_\parallel(t)$, and applying the Cauchy--Schwarz inequality, we obtain
\begin{equation*}\label{vpar-dot-bound}
	\frac{\mathrm{d}}{\mathrm{d}t}	\bigl|\check{v}_\parallel(t)\bigr|
	\leq
		2\big(\bigl|\dot{\check P}_{\parallel}(x(t))\bigr|\,|\check{v}_\parallel(t)|
		+\bigl|\dot{\check P}_{\parallel}(x(t))\bigr|\,|\check{v}_\perp(t)|+
		\bigl|\check P_{\parallel}(x(t))\bigr|\,|F(x(t))|\big).
\end{equation*}
Combining  the initial bound  $|\check{v}_\parallel(0)|=|\check{P}_{\parallel}(x(0))v(0)|\leq |v(0)| \leq C$ with $\dot{\check{P}}_{\parallel}(x(t))=\mathcal{O}(\varepsilon)$, an application of Gronwall's inequality yields
$\abs{\check{v}_{\parallel}(t)}\leq C.$
By the definitions of $\check{P}_{\parallel}(x(t))$ and ${P}_{\parallel}$, it follows that ${P}_{\parallel}=\check{P}_{\parallel}(x(t))+\mathcal{O}(\varepsilon)$, and hence
$\abs{{v}_{\parallel}(t)}\leq C.$
This completes the proof of \eqref{eaxctvd}.
\end{proof}

\begin{rem}
In contrast to the parallel component, the perpendicular velocity
$v_\perp(t)$ and its time derivative are not uniform for $\varepsilon$. Note that
\begin{equation*}
\begin{aligned}
\dot{v}_\perp(t)&=\dot{v}(t)-\dot{v}_\parallel(t)=v(t) \times  B(x(t)) +F(x(t))-\dot{v}_\parallel(t)=({v}_\perp(t)+{v}_\parallel(t))\times  B(x(t)) +F(x(t))-\dot{v}_\parallel(t)\\
&={v}_\perp(t)\times  B(x(t))+{v}_\parallel(t)\times  B_1(x(t)) +F(x(t))-\dot{v}_\parallel(t).
\end{aligned}
\end{equation*}	
Taking the inner product of the above equation with $v_\perp(t)$
and applying the Cauchy--Schwarz inequality, we obtain
\begin{equation*}
	\begin{aligned}
\frac{\mathrm{d}}{\mathrm{d}t}\bigl|{v}_\perp(t)\bigr|
\leq 2\bigl|{v}_\perp(t)\times  B(x(t))\bigr|+2\bigl|{v}_\parallel(t)\times  B_1(x(t))\bigr| +2\bigl|F(x(t))\bigr|+2\bigl|\dot{v}_\parallel(t)\bigr|.
\end{aligned}
\end{equation*}	
Recalling that $B(x)=\varepsilon^{-1}B_0+B_1(x)$ and applying
Gronwall's inequality, we get $\bigl|{v}_\perp(t)\bigr|\leq Ce^{CT/\varepsilon}$, which shows that $v_\perp(t)$ is not uniformly bounded with respect to
$\varepsilon$. By a similar argument, the same conclusion holds for
$\dot v_\perp(t)$.
\end{rem} 
 
Based on these properties of the exact solution, we now investigate the error of the filtered two-step variational integrator in the strong magnetic field regime.
The following theorem establishes error estimates for the numerical solution
in the position and the parallel velocity.

\begin{mytheo}\textbf{(Error bounds)}\label{error bound 2}
	For  a strong magnetic field with $0<\varepsilon\ll 1$, suppose  that \\
	(a)  The initial velocity of \eqref{charged-particle} is bounded independent of $\varepsilon$ and $h$;\\
	(b) The exact solution $x(t)$
	stays in an $\varepsilon$-independent compact set $K$ for $0 \leq t \leq T$;\\
	(c) For some $N \geq 1$, the following nonresonance conditions are assumed,
	\begin{equation}\label{non-resonance conditions}
	\abs{\sin\Big(\frac{k h}{2 \varepsilon}\Big)}\geq c>0, \quad 
	\abs{\cos\Big(\frac{k h}{2 \varepsilon}\Big)}\geq c>0,\quad  k=1,2,\ldots,N, \quad  
\end{equation}	
where $c$ is a positive constant.   When the filtered two-step variational integrator 
\eqref{Method}--\eqref{AV1} is applied to \eqref{charged-particle}, 
there exist positive constants $C^{*}$ and $c^{*}$ 
such that the global errors in the position $x$ 
and the parallel velocity $v_{\parallel}=P_{\parallel}v$ 
at time $t_{n+1/2}=(n+1/2)h \leq T$ satisfy
\begin{equation*}
	\begin{aligned}
		&\abs{x_{n+1/2}-x(t_{n+1/2})} \leq Ch^2, \qquad \quad \ \abs{v_{\parallel}^{n+1/2}-v_{\parallel}(t_{n+1/2})} \leq Ch^2, \quad \textmd{provided that} \ h^2 > C^{*}\varepsilon,\\
		&\abs{x_{n+1/2}-x(t_{n+1/2})} \leq C\varepsilon,\qquad \quad \ \ \
			\abs{v_{\parallel}^{n+1/2}-v_{\parallel}(t_{n+1/2})} \leq C \varepsilon,  \quad\  \textmd{provided that} \  c^{*}\varepsilon^2 \leq h^2 \leq C^{*}\varepsilon.
	\end{aligned}
\end{equation*}	
Moreover, the estimates also hold at time $t_{n+1}=(n+1)h \leq T$ are given by	
\begin{equation*}
	\begin{aligned}
		&\abs{x_{n}-x(t_{n})} \leq Ch^2, \qquad\qquad \	\abs{v_{\parallel}^{n}-v_{\parallel}(t_{n})} \leq Ch^2, \quad \textmd{provided that} \ h^2 > C^{*}\varepsilon,\\
			&\abs{x_{n}-x(t_{n})} \leq C\varepsilon,\qquad \quad \quad \ \ \
			\abs{v_{\parallel}^{n}-v_{\parallel}(t_{n})} \leq C \varepsilon,  \quad\  \textmd{provided that} \  c^{*}\varepsilon^2 \leq h^2 \leq C^{*}\varepsilon,
	\end{aligned}
\end{equation*}	
where the constant $C$ is independent of $\varepsilon$, $h$, and $n$,  but depends on $T$, the velocity bound, the constants $c$, $c^{*}$, $C^{*}$,  and the bounds of the derivatives of $B_1$ and $F$ on the compact set $K$.
\end{mytheo}

\begin{rem}
The error estimates are  derived from a comparison between the modulated Fourier expansions of the numerical solution (see Section~\ref{MFE-method}) and the exact solution (see Section~\ref{MEF-exact}). In addition, for the perpendicular velocity $v_{\perp}=P_{\pm1}v$, the modulated Fourier expansion analysis shows that the error remains of order $\mathcal{O}(1)$ for both step-size regimes.
\end{rem}

\begin{rem}
In the small-step regime $h^2 < c^{*}\varepsilon^2$, the position and parallel velocity errors are observed to be of order $h^2/\varepsilon$, while the perpendicular velocity error is of order $h^2/\varepsilon^2$, as illustrated by the numerical results in Problem~3 (see Figure~\ref{fig:problem33}).
Such refined estimates are not accessible by the present modulated Fourier expansion analysis. 
Their derivation would require improved error techniques as in 
\cite{WZ,WJ23,Bao2024,FengSchratz2024,Yin2025}, 
and is beyond the scope of this paper. 
This regime will be investigated in a subsequent work.
\end{rem}

In the follows, it is shown that the  filtered two-step variational integrator \eqref{Method}--\eqref{AV1}  has good conservation properties for a strong magnetic field. We first show the result on the  long-time  near-conservation of energy.
\begin{mytheo}\textbf{(Energy near-conservation)}\label{conservation-H2}
Under the  assumptions given in Theorem \ref{error bound 2}, it is further supposed  that the numerical solution $x_n$ stays in a compact set $K$ which is independent of $\varepsilon$ and $h$. Then there exists a positive constant $C^{*}$, independent of $\varepsilon$ and $h$, such that the integrator \eqref{Method}--\eqref{AV1} conserves the following energy over long times:
	\begin{equation*}
    \begin{aligned}
		&\abs{H(x_{n+1/2}, v_{n+1/2})-H(x_{1/2}, v_{1/2})} \leq Ch \quad  \textmd{for}\ \  0\leq nh \leq  c\min(h^{-M},\varepsilon^{-N}), \ \  \textmd{provided that} \ h^2 > C^{*}\varepsilon,\\ 
& \abs{H(x_{n+1/2}, v_{n+1/2}) - H(x_{1/2}, v_{1/2})} \leq C\min\{h^2/\varepsilon^2,\varepsilon\} \quad  \textmd{for}\ \  0\leq nh \leq    c\varepsilon^{-N}, \ \ \textmd{provided that} \ h^2 \leq C^{*}\varepsilon,\\
\end{aligned}
	\end{equation*}	
where  $M>N$  are arbitrary positive integers and $c$ is independent of $\varepsilon$ and $h$. The constant $C$ is independent of $\varepsilon$, $n$ and $h$, but depends on $M$, $N$ and $K$,  on  bounds of derivatives of $B_1$ and $F$ and on the initial velocity bound.
\end{mytheo}
It is well known that the magnetic moment is an adiabatic invariant, the long time near-conservation of magnetic moment  can also be obtained for the filtered two-step variational integrator \eqref{Method}--\eqref{AV1} by the following theorem.
\begin{mytheo}\textbf{(Magnetic moment near-conservation)}\label{conservation-M2}
Under the conditions of Theorem~\ref{conservation-H2}, the numerical solution satisfies the following near-conservation of the magnetic moment:
	\begin{equation*}
    \begin{aligned}
		&\abs{I(x_{n+1/2}, v_{n+1/2})-I(x_{1/2}, v_{1/2})} \leq Ch \quad \textmd{for}\  0\leq nh \leq  c\min(h^{-M},\varepsilon^{-N}), \ \ \textmd{provided that}\  h^2 >C^{*}\varepsilon,\\
          & \abs{I(x_{n+1/2}, v_{n+1/2}) - I(x_{1/2}, v_{1/2})} \leq C\varepsilon  \quad \textmd{for}\  0\leq nh \leq   c\varepsilon^{-N}, \ \ \textmd{provided that} \  h^2 \leq C^{*}\varepsilon,
\end{aligned}
	\end{equation*}	
where  $M>N$  are arbitrary positive integers and $c$ is independent of $\varepsilon$ and $h$. The constant $C$ is independent of $\varepsilon$, $n$ and $h$, but depends on $M$, $N$ and $K$,  on  bounds of derivatives of $B_1$ and $F$ and on the initial velocity bound.
\end{mytheo}
\begin{rem}
Under the condition $h^2 \le C^{*}\varepsilon$, the near-conservation result for the magnetic moment
in Theorem~\ref{conservation-M2} differs from the  energy estimate in Theorem~\ref{conservation-H2}. This difference arises from the adiabatic invariance of the magnetic moment, which is conserved up to $\mathcal{O}(\varepsilon)$ over long times $t \le \varepsilon^{-N}$ for any $N>1$; see, e.g., \cite{Northrop,Benettin1994,Hairer2018}. Moreover, Theorem~\ref{conservation-M2} shows that the filtered two-step variational integrator exhibits this behavior by preserving the magnetic moment with $\mathcal{O}(\varepsilon)$ accuracy
over times $0\leq nh \leq   c\varepsilon^{-N}$, as established by the modulated Fourier expansion analysis.
This is consistent with the adiabatic invariance of the exact solution.\end{rem}

\subsubsection{ Numerical experiments in a strong magnetic field} \label{sec3.2.1}
In what follows, we present two numerical experiments to illustrate the behavior of the proposed method in a strong magnetic field. For comparison, we consider the BORIS, TSM, and FVARM schemes introduced in Section~\ref{sec3.1.1}.
	
\noindent\vskip3mm \noindent{Problem 3. \textbf{(Strong magnetic field)}}
Solve the charged-particle dynamics \eqref{charged-particle} under a strong
 magnetic field 
$B(x)=\nabla_x \times A(x)=\big(x_2(x_3-x_2),x_1(x_1-x_3),x_3(x_2-x_1)+1/\varepsilon\big)^{\intercal}$, where the vector potential is defined as 
$A(x)=-(x_2,-x_1,0)^{\intercal}/(2\varepsilon)+x_{1}x_{2}x_{3}(1,1,1)^{\intercal}$. The scalar potential is given by $U(x)=\abs{x}^2/2$.
We choose the initial values $x(0)=(0.3,0.2,-1.4)^{\intercal}$, $v(0)=(-0.7,0.08,0.2)^{\intercal}$ and consider it on $[0, \pi/2]$.  To illustrate the performance of the method under the three step-size regimes  $h^2 > C^{*}\varepsilon$, 
$c^{*}\varepsilon^2 \le h^2 \le C^{*}\varepsilon$, 
and $h^2 < c^{*}\varepsilon^2$,  we present both global errors and long-time behavior. Figures \ref{fig:problem31}--\ref{fig:problem33} display the global errors in $x$, $v_{\parallel}$, and $v_{\perp}$ for varying values of $\varepsilon$ and stepsizes $h$. Figures \ref{fig:problem34} and \ref{fig:problem35} show the energy and magnetic moment errors over the interval $[0,10000]$ for different values of $\varepsilon$ and $h$.

 \begin{figure}[H]
 	\centering\tabcolsep=0.5mm
 	\begin{tabular}
 		[l]{ll}
 		\includegraphics[width=3.8cm,height=3.2cm]{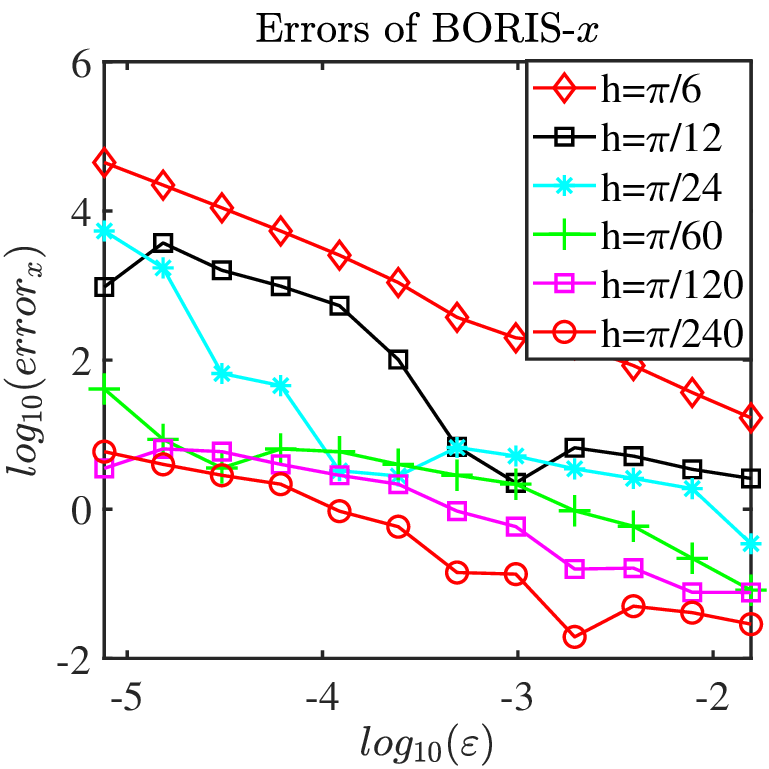}\quad \ \
 		\includegraphics[width=3.8cm,height=3.2cm]{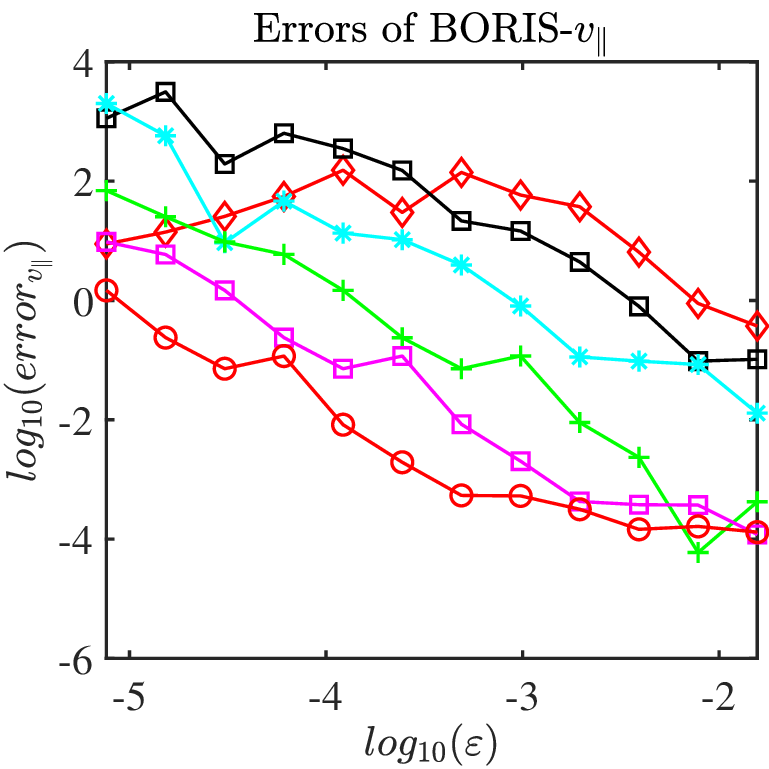}\quad \ \
 		\includegraphics[width=3.8cm,height=3.2cm]{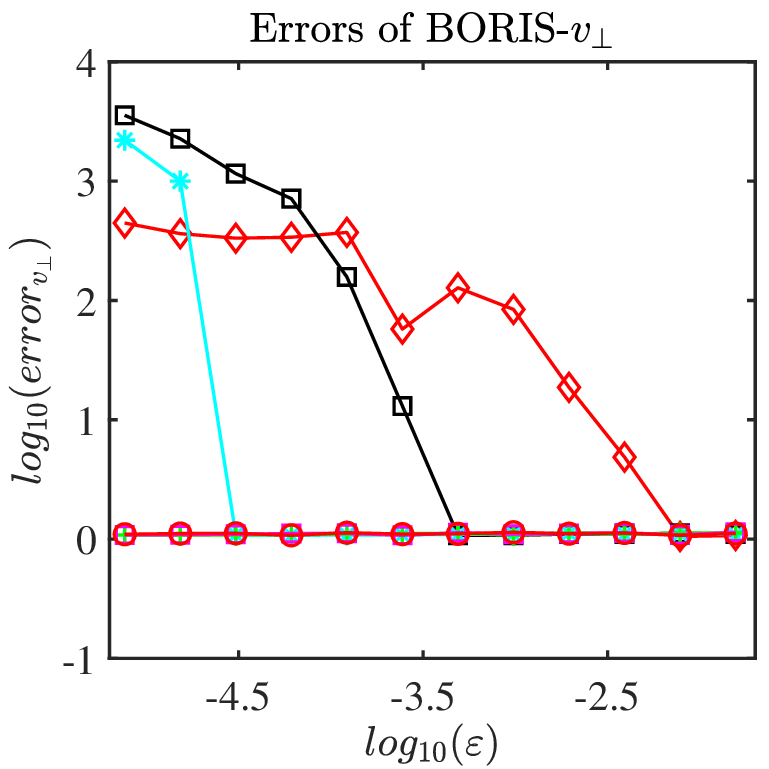} \\
 			\includegraphics[width=3.8cm,height=3.2cm]{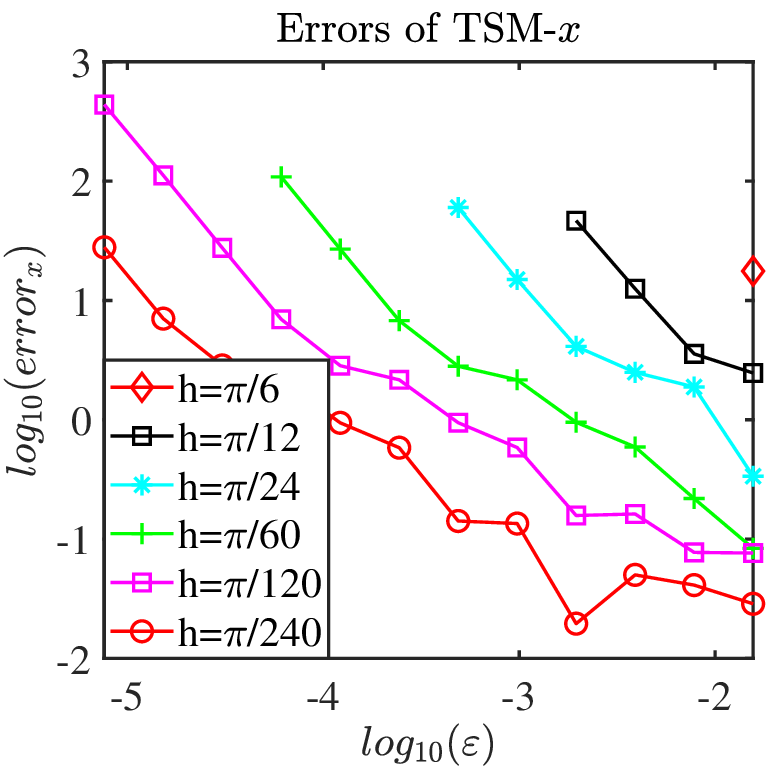}\quad \ \
 			\includegraphics[width=3.8cm,height=3.2cm]{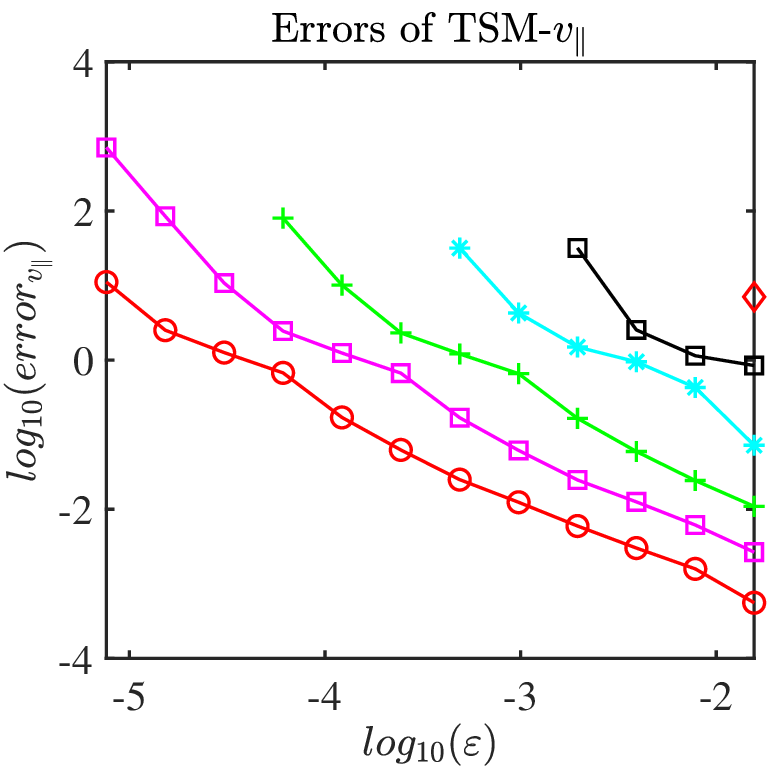}\quad \ \
 			\includegraphics[width=3.8cm,height=3.2cm]{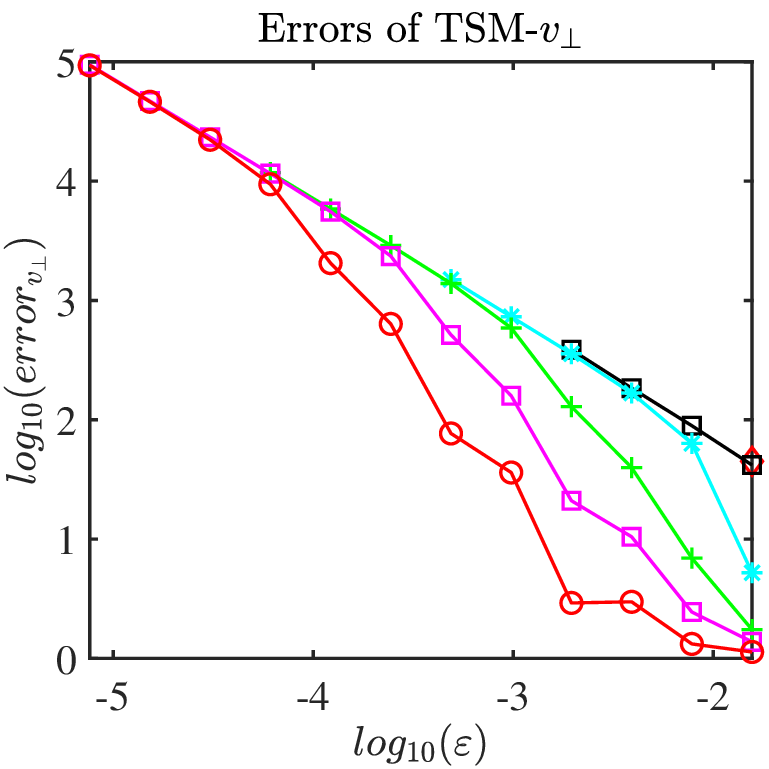}\\
 			\includegraphics[width=3.8cm,height=3.2cm]{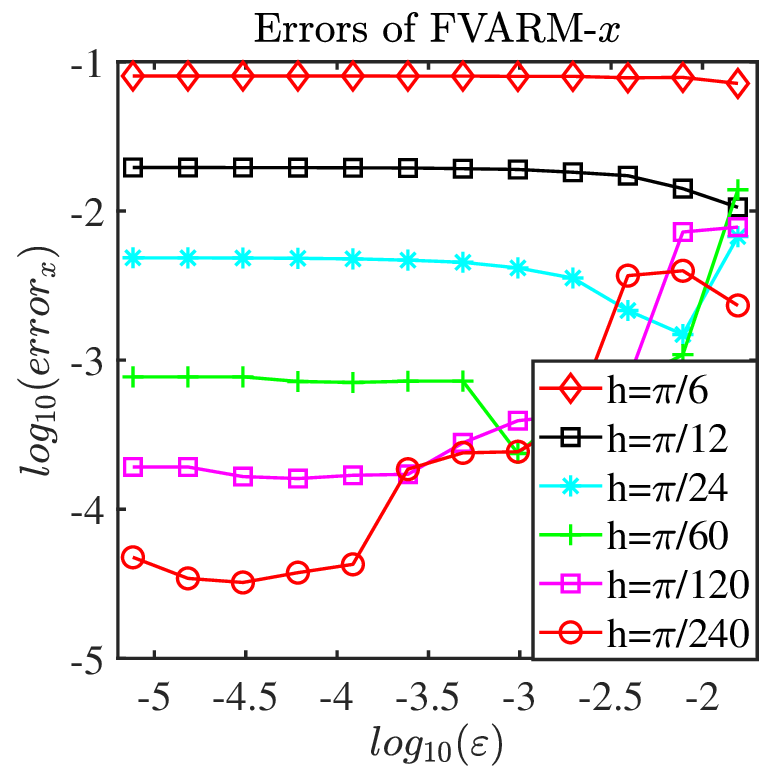}\quad \ \
 			\includegraphics[width=3.8cm,height=3.2cm]{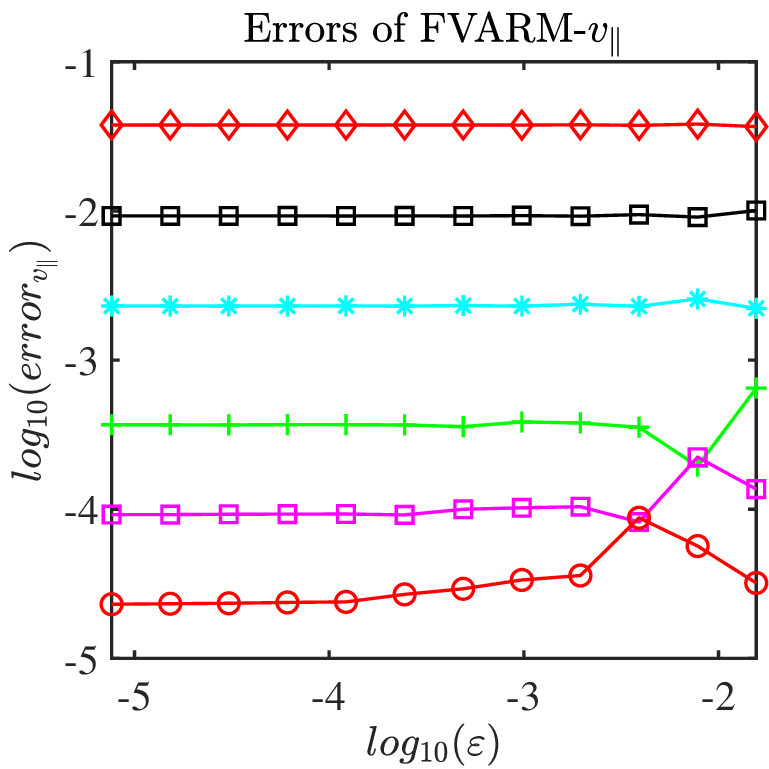}\quad \ \
 			\includegraphics[width=3.8cm,height=3.2cm]{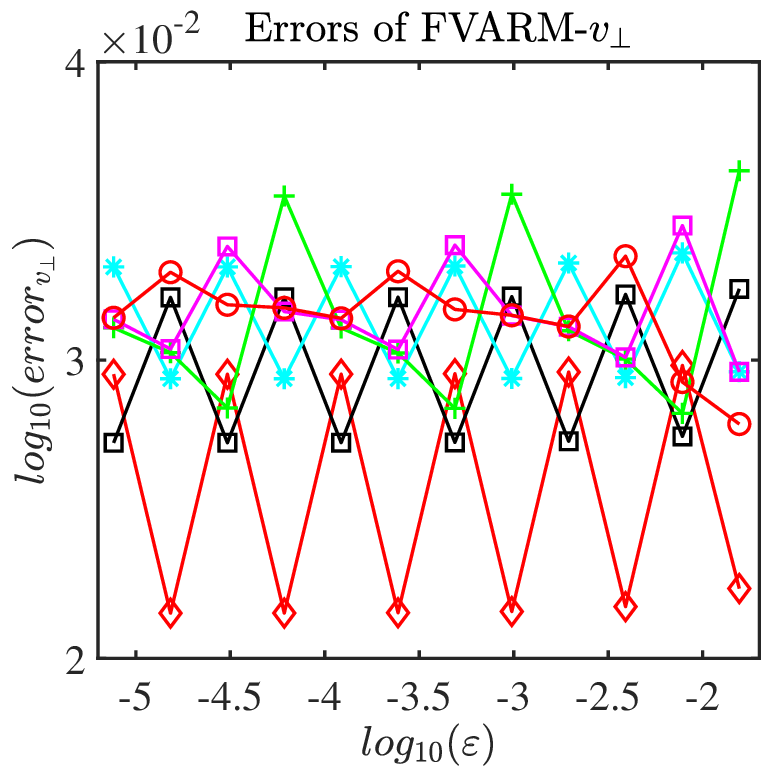}\\
 			\includegraphics[width=3.8cm,height=3.2cm]{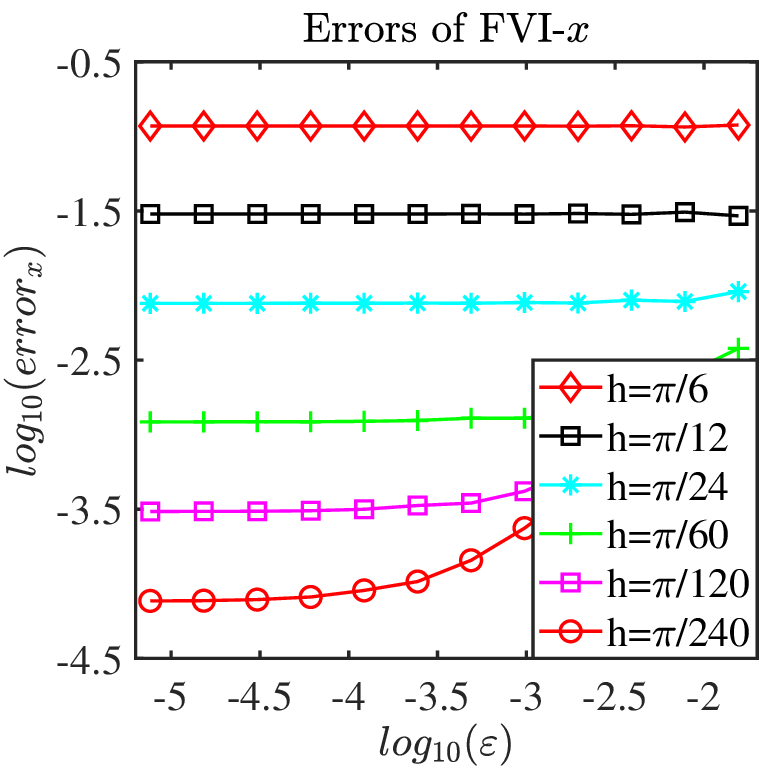}\quad \ \
 			\includegraphics[width=3.8cm,height=3.2cm]{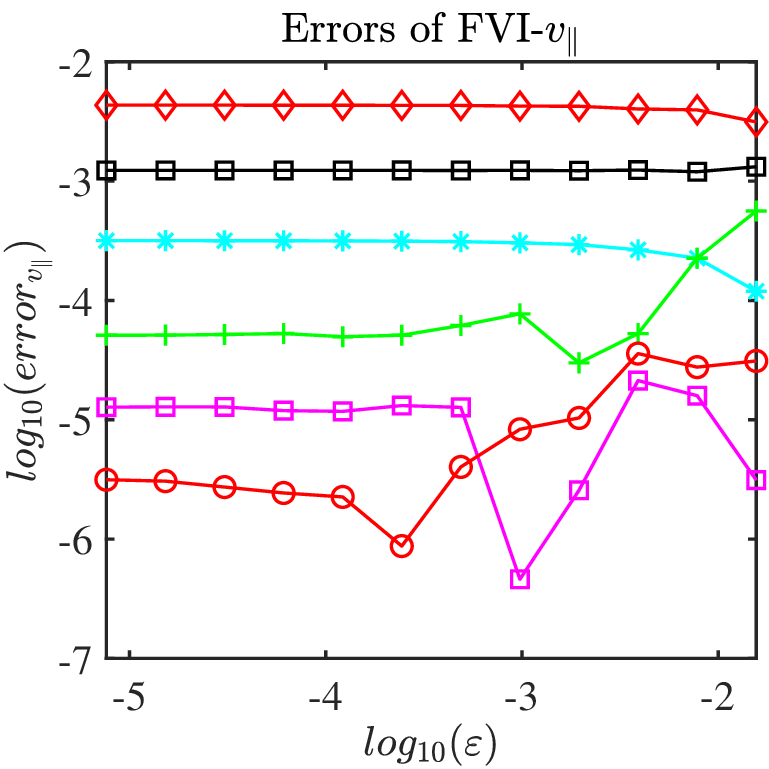}\quad \ \
 		\includegraphics[width=3.8cm,height=3.2cm]{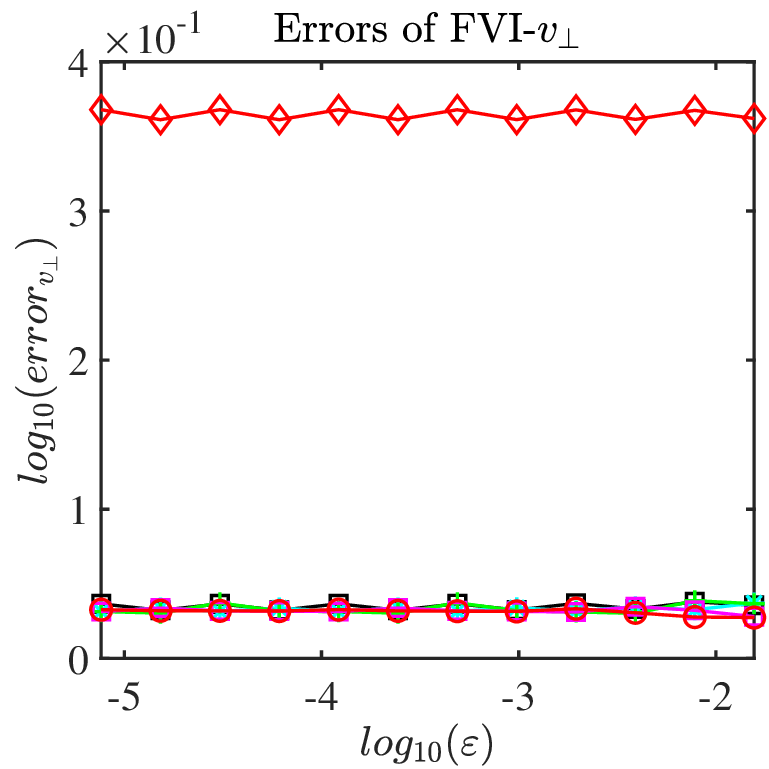}	
 	\end{tabular}
 	\caption{ Problem 3. The global errors in $x$, $v_{\parallel}$  and $v_{\perp}$ at time $t=\pi/2$ vs. $\varepsilon$ ($\varepsilon=1/2^{k}, k=6,\ldots,17 $) with different $h$. }
 	\label{fig:problem31}
 \end{figure}

 \begin{figure}[H]
 	\centering\tabcolsep=0.5mm
 	\begin{tabular}
 		[c]{ccc}
 		\includegraphics[width=3.8cm,height=3.2cm]{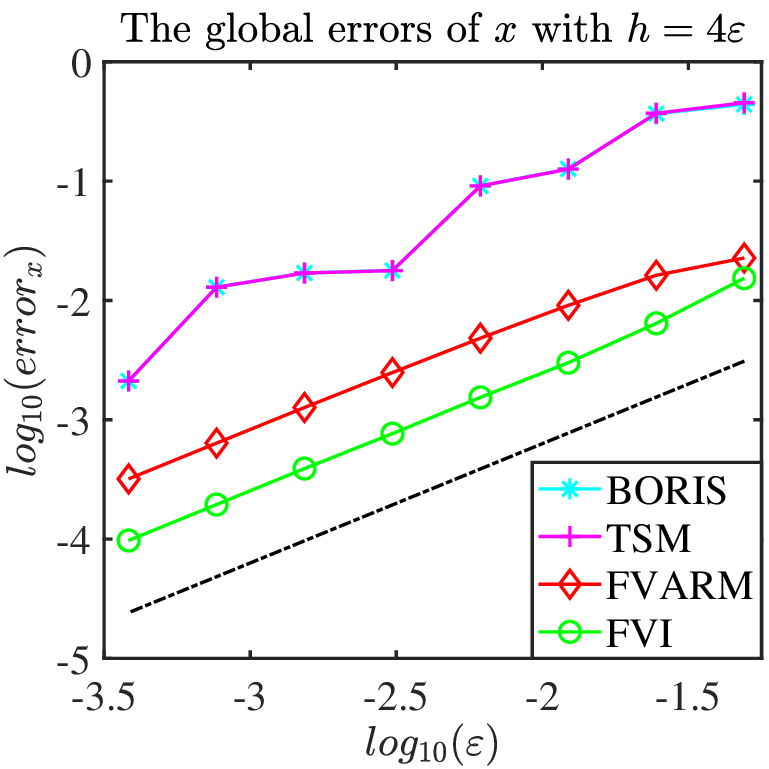}\quad \ \
 		\includegraphics[width=3.8cm,height=3.2cm]{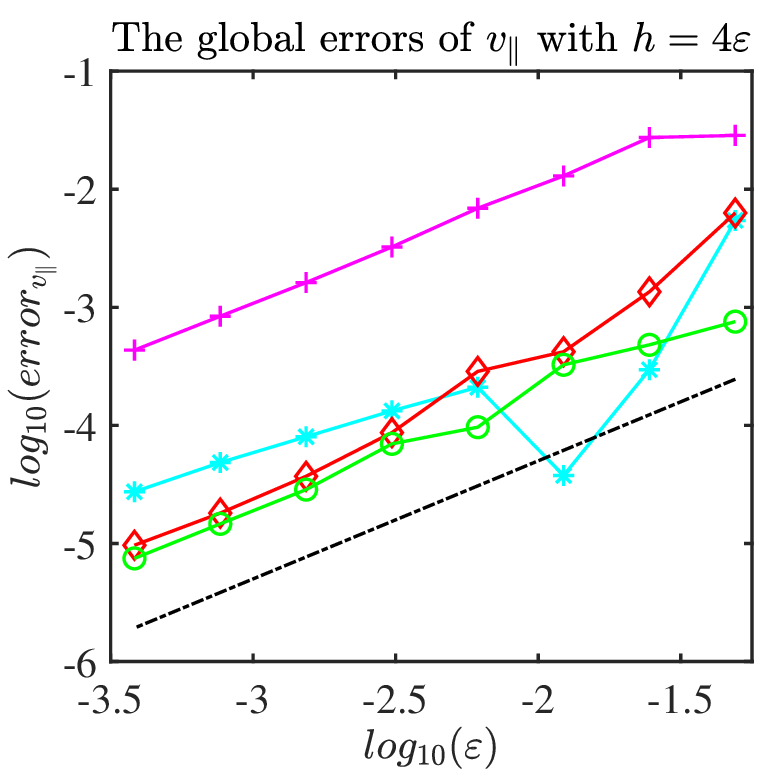}\quad \ \
 		\includegraphics[width=3.8cm,height=3.2cm]{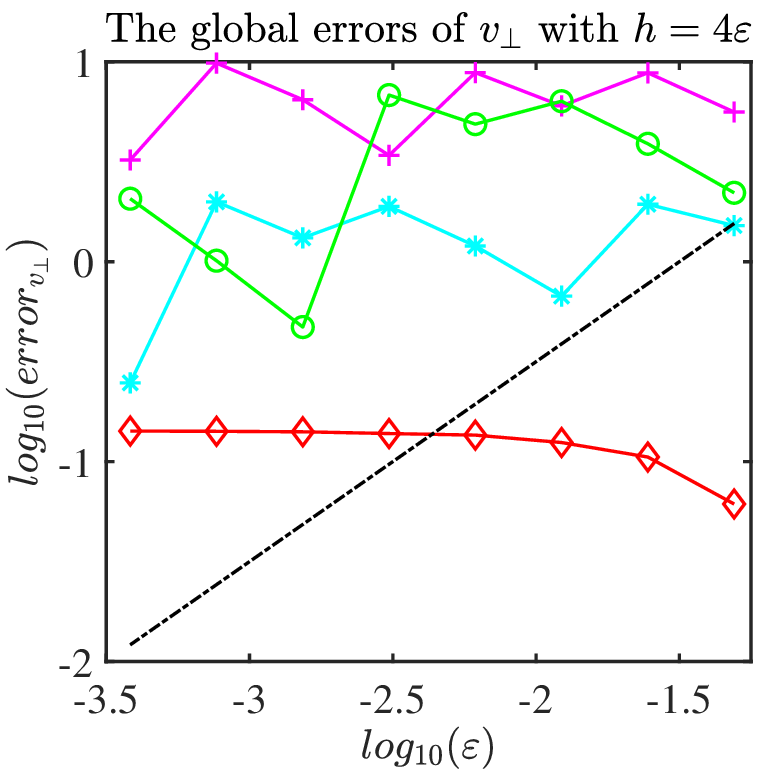}
 	\end{tabular}
 	\caption{Problem 3. The global errors in $x$, $v_{\parallel}$  and $v_{\perp}$ at time $t=\pi/2$ vs. $\varepsilon$ ($\varepsilon=\pi/2^k$, $k=6,\ldots,13$ ) with different $h$ (the dash-dot line is slope one).}
 	\label{fig:problem32}
 \end{figure}

 \begin{figure}[H]
 	\centering\tabcolsep=0.5mm
 	\begin{tabular}
 		[c]{ccc}
   		\includegraphics[width=3.8cm,height=3.2cm]{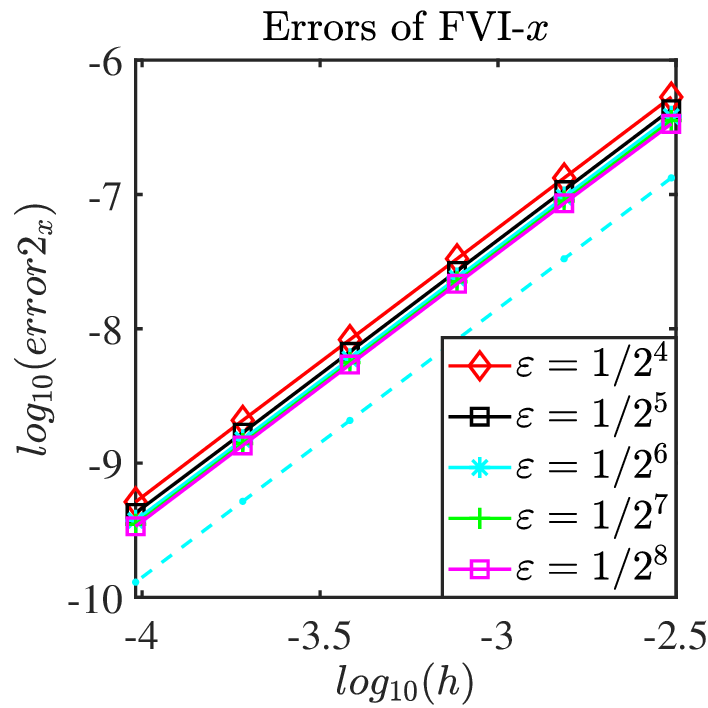}\quad \ \
 		\includegraphics[width=3.8cm,height=3.2cm]{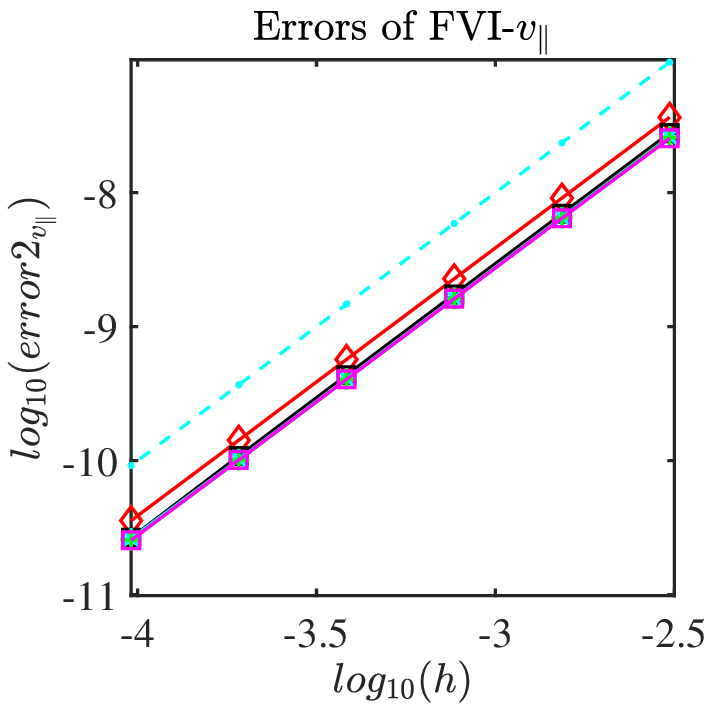}\quad \ \
 		\includegraphics[width=3.8cm,height=3.2cm]{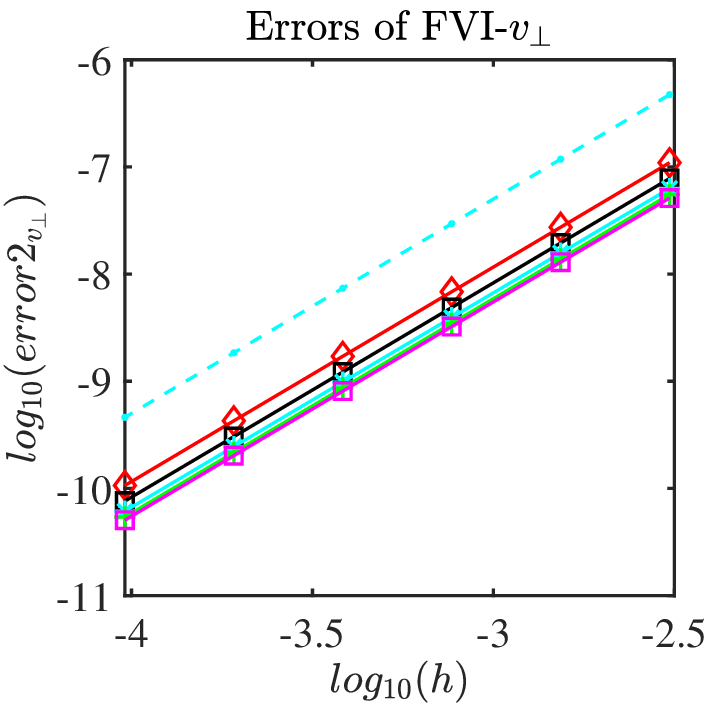}
 	\end{tabular}
 	\caption{Problem 3. The global errors in $x$, $v_{\parallel}$  and $v_{\perp}$ at time $t=\pi/2$ vs. $\varepsilon$ ($h=\pi/2^k$, $k=10,\ldots,15$ ) with different $h$ (the dash-dot line is slope two).}
 	\label{fig:problem33}
 \end{figure}

\begin{figure}[H]
	\centering\tabcolsep=0.5mm
	\begin{tabular}
		[c]{cc}
		\includegraphics[width=5.8cm,height=2.8cm]{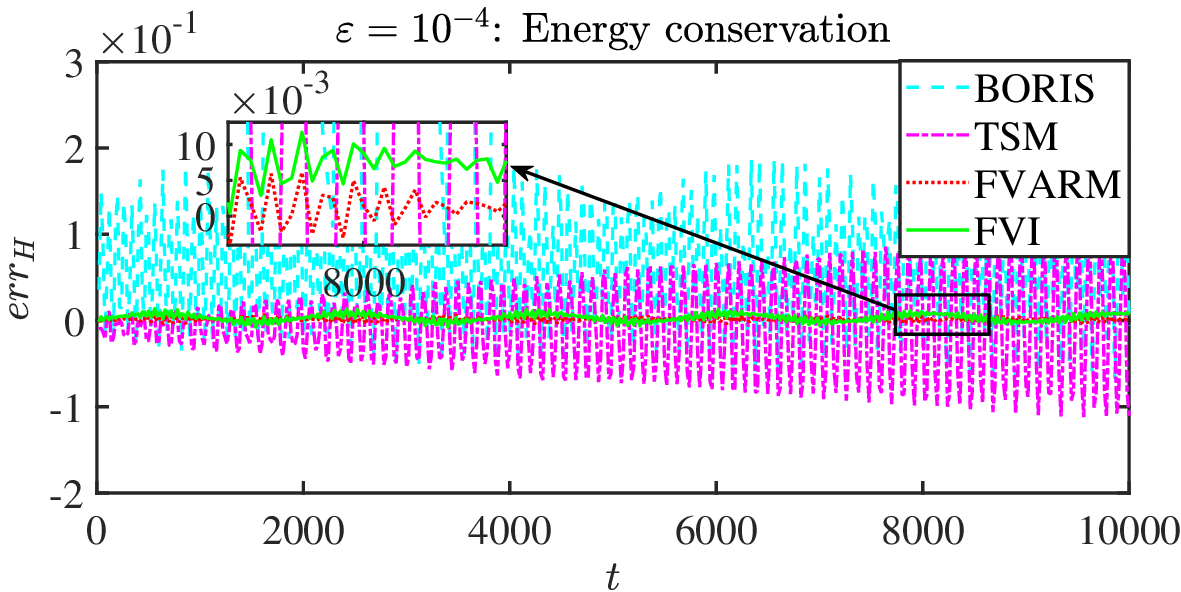}\qquad
		\includegraphics[width=5.8cm,height=2.8cm]{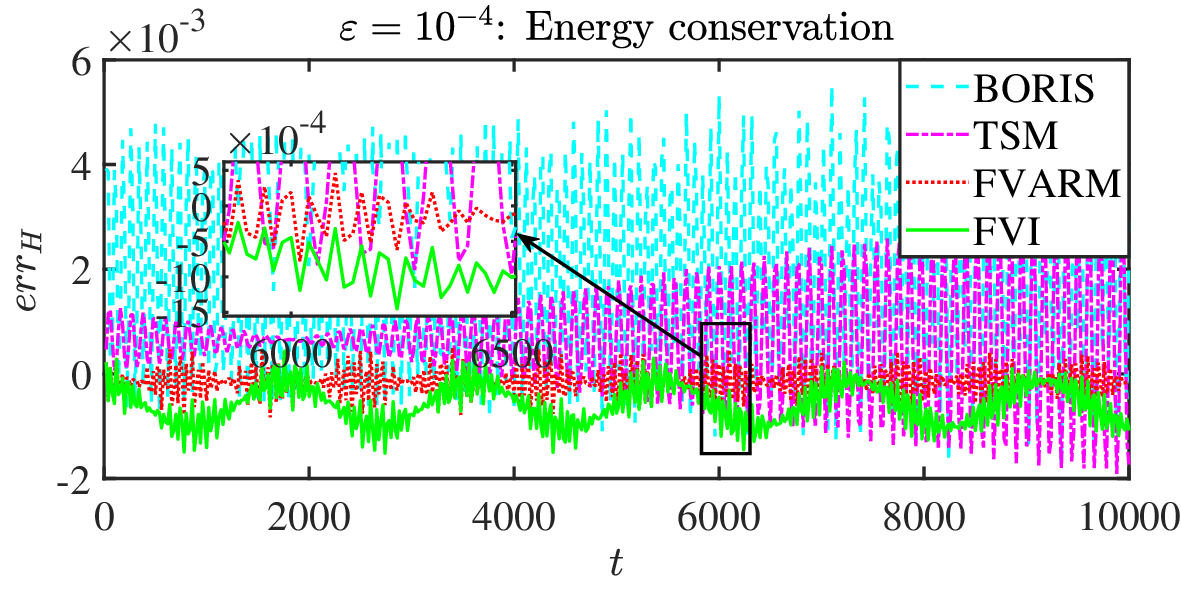}\\
		\includegraphics[width=5.8cm,height=2.8cm]{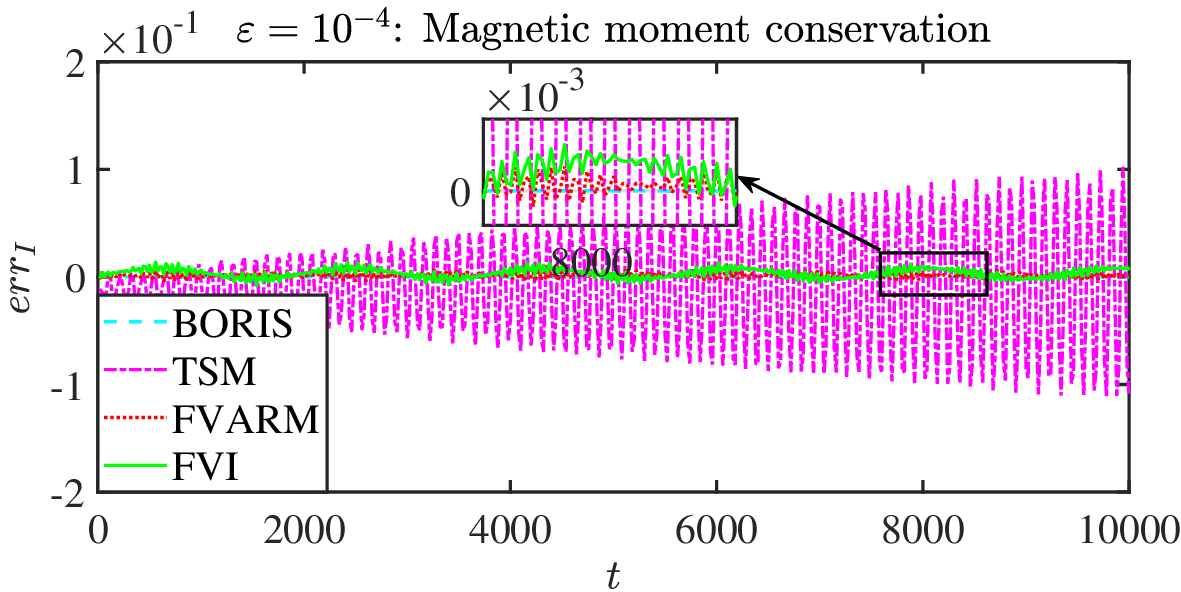}\qquad
		\includegraphics[width=5.8cm,height=2.8cm]{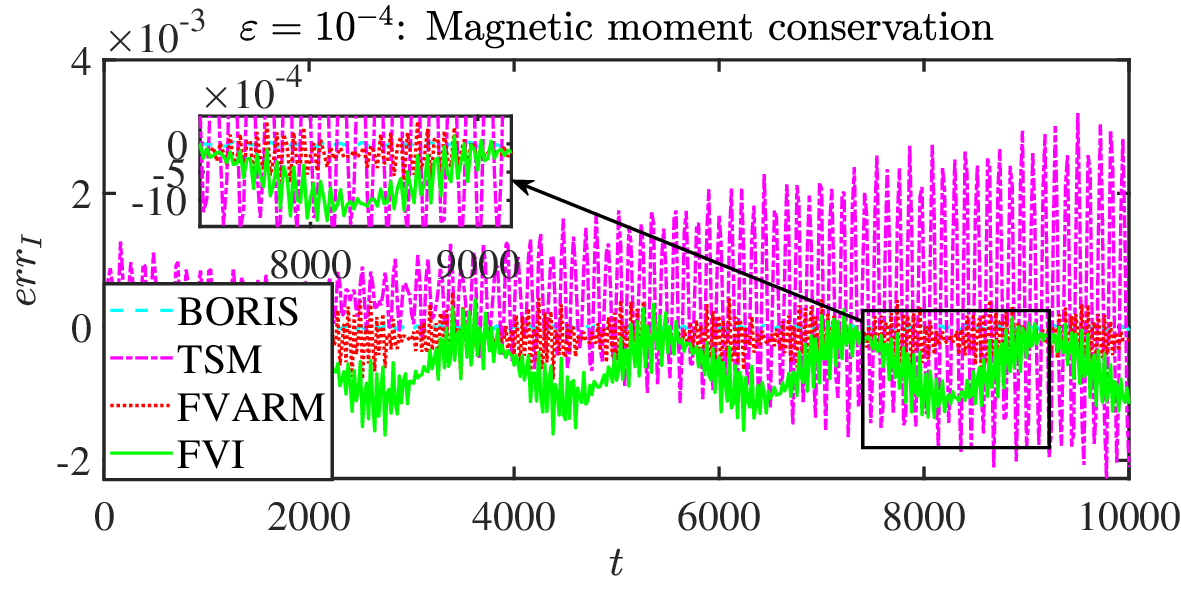}  
	\end{tabular}
	\caption{Problem 3. Evolution of energy error $e_H$ and magnetic moment error $e_{I}$ with step sizes $h = 0.01$ (left) and $h = 0.2\varepsilon$ (right).}
	\label{fig:problem34}
\end{figure}

\begin{figure}[H]
	\centering\tabcolsep=0.5mm
	\begin{tabular}
[c]{cc}
\includegraphics[width=5.8cm,height=2.8cm]{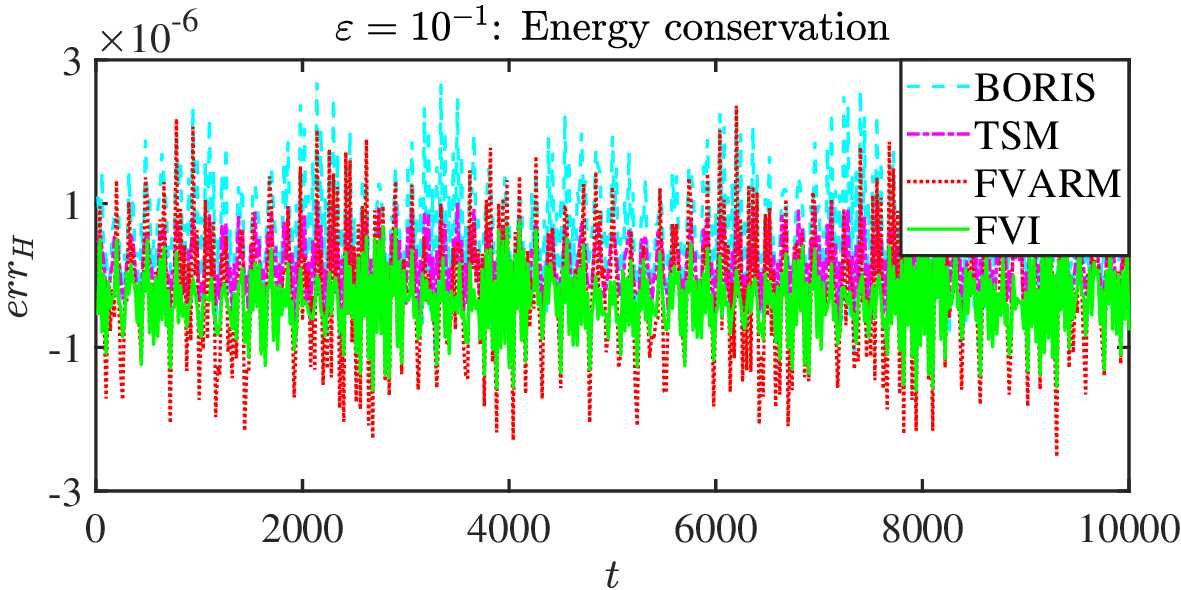}\qquad
\includegraphics[width=5.8cm,height=2.8cm]{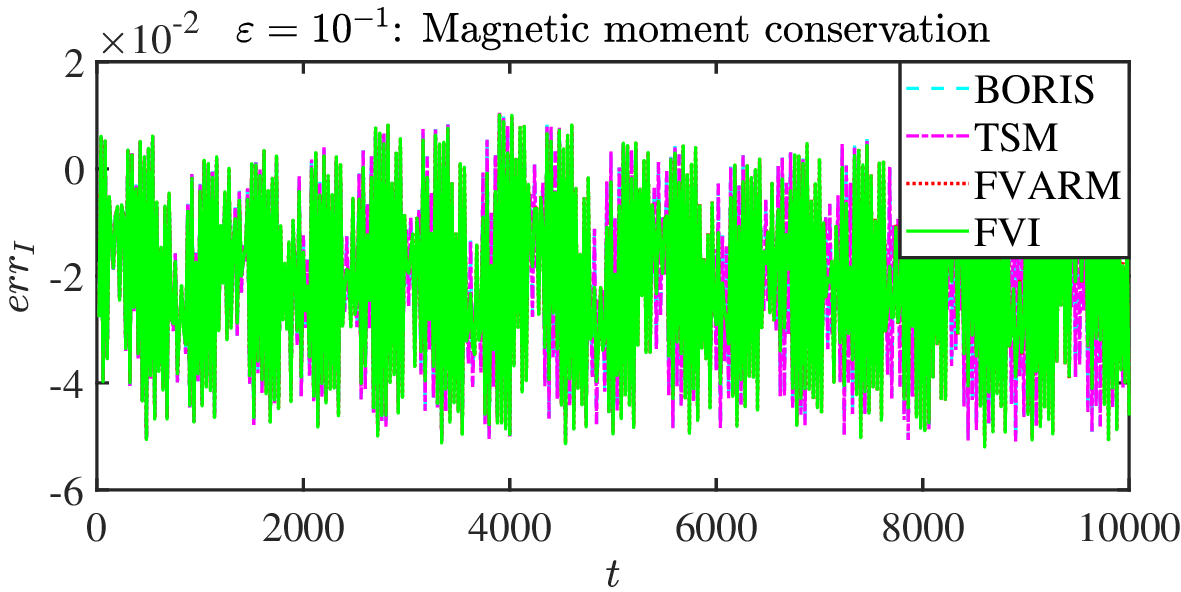} 
	\end{tabular}
	\caption{Problem 3. Evolution of energy error $e_H$ and magnetic moment error $e_{I}$ with step sizes $h = 0.001$.}
	\label{fig:problem35}
\end{figure}

\noindent\vskip3mm \noindent{Problem 4. \textbf{(Maximal Ordering Scaling})}
This problem investigates a strong
magnetic field as ${B}(x)=\varepsilon^{-1 }\big(1+\varepsilon x_2/(1+\varepsilon^2(x_1^2+x_2^2+x_3^2))^{1/2} ,1-\varepsilon x_1/(1+\varepsilon^2(x_1^2+x_2^2+x_3^2))^{1/2},0\big)^{\intercal}$ and scalar potential $U(x)=x_1^2+2x_2^2+3x_3^2-x_1$  with  $F(x)=-\nabla_{x}U(x)$.
The initial conditions are chosen as  $x(0)=(0.1,0.03,-0.04)^{\intercal}$  and $v(0)=(-0.2,0.01,0.7)^{\intercal}$, and the simulation is performed over the interval $[0,  \pi/2]$. Figures \ref{fig:problem41} and \ref{fig:problem42} show the global errors and Figure \ref{fig:problem43}  presents the near-conservation of energy and magnetic moment.

\begin{figure}[H]
	\centering\tabcolsep=0.5mm
	\begin{tabular}
		[c]{ccc}
		\includegraphics[width=3.8cm,height=3.2cm]{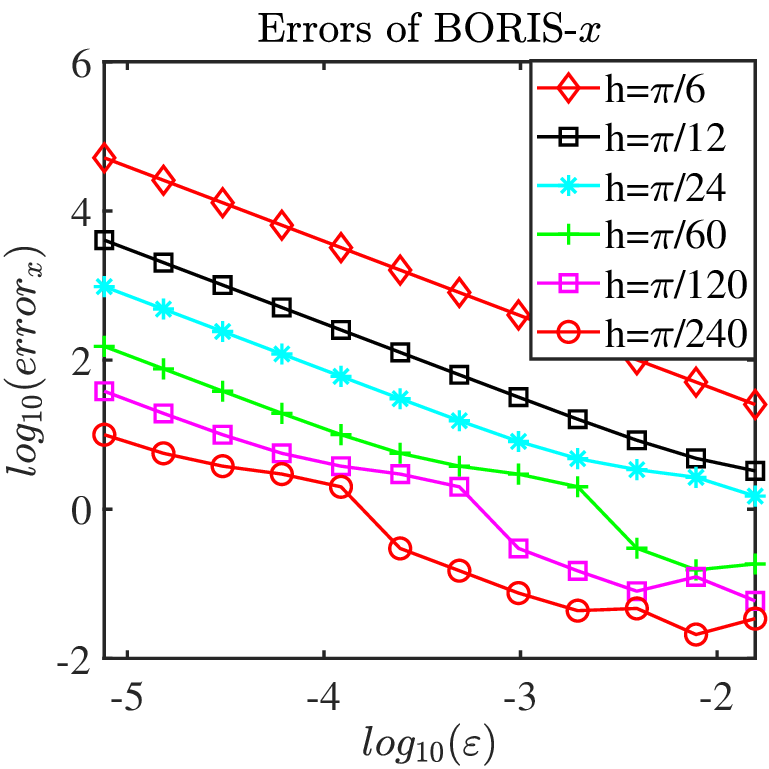}\quad \ \
		\includegraphics[width=3.8cm,height=3.2cm]{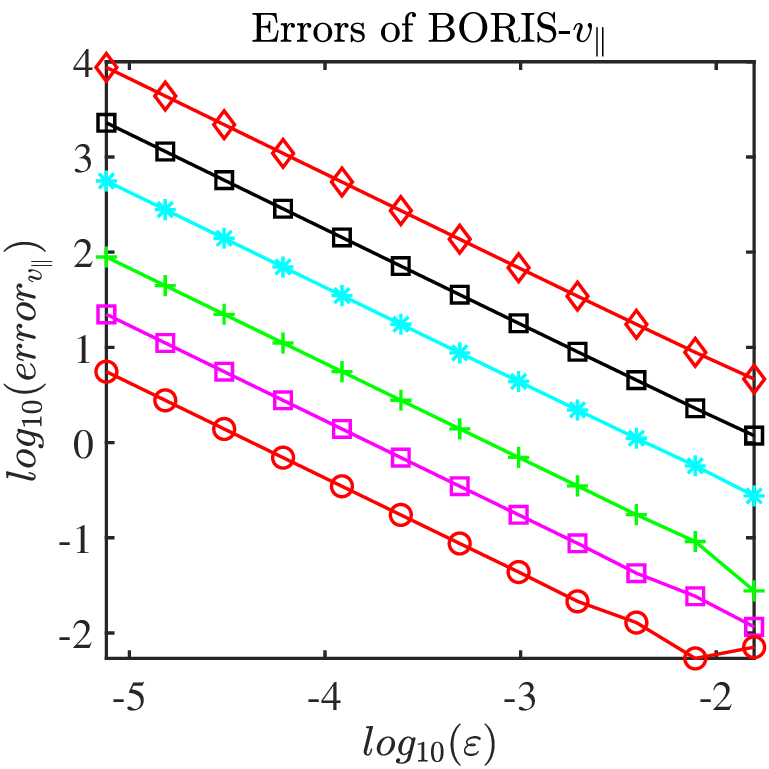}\quad \ \
		\includegraphics[width=3.8cm,height=3.2cm]{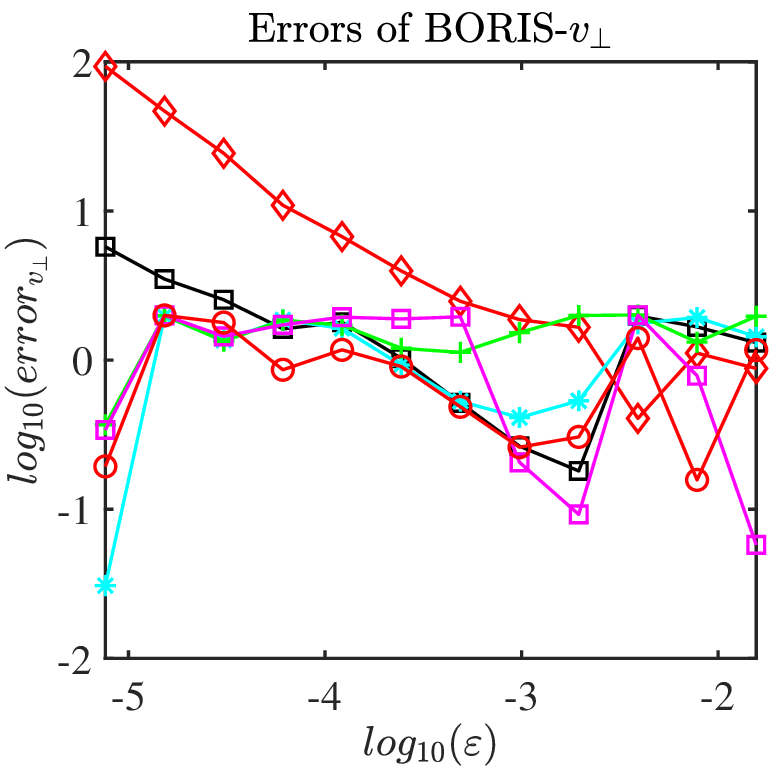} \\
		\includegraphics[width=3.8cm,height=3.2cm]{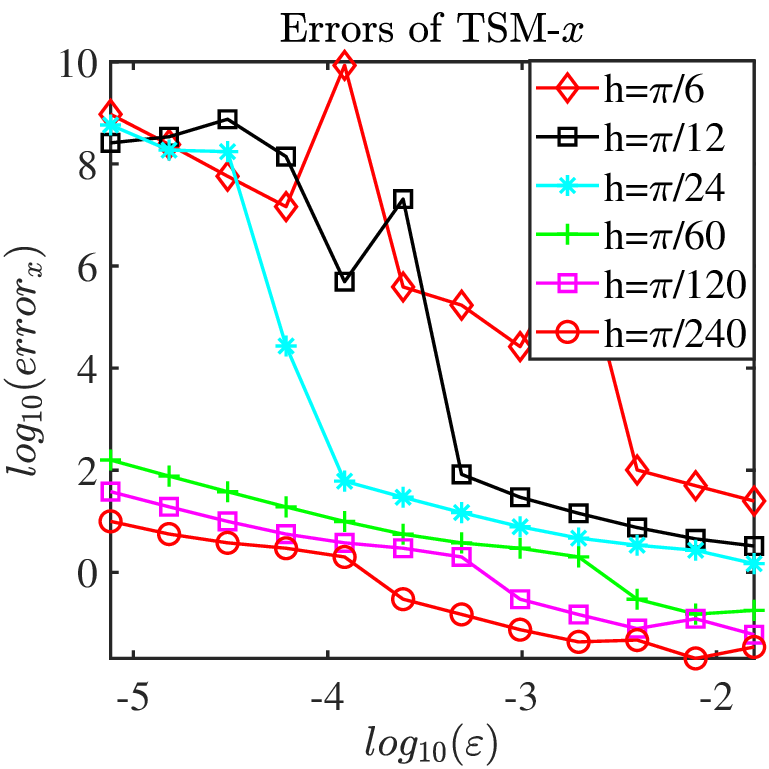}\quad \ \
		\includegraphics[width=3.8cm,height=3.2cm]{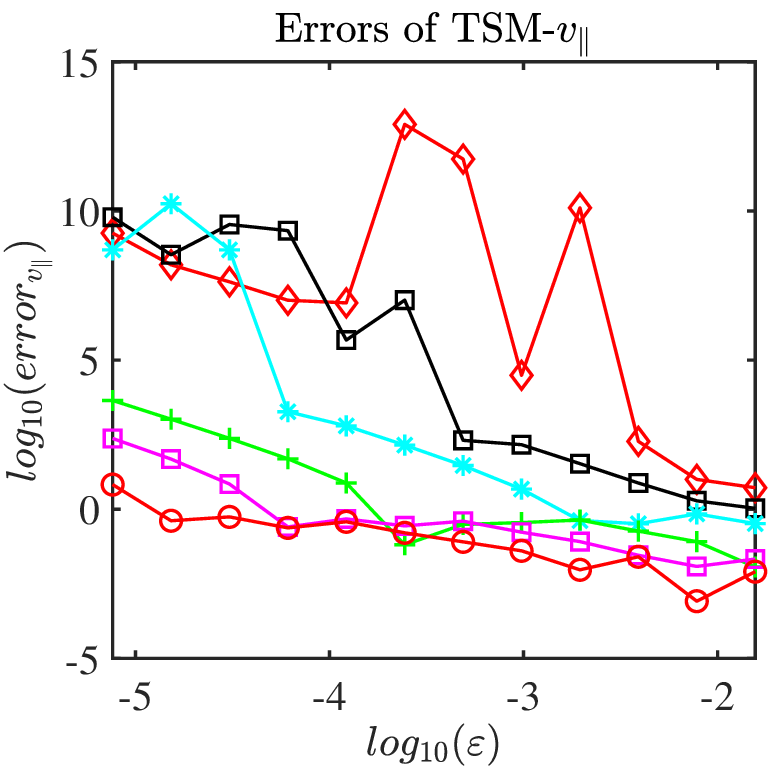}\quad \ \
		\includegraphics[width=3.8cm,height=3.2cm]{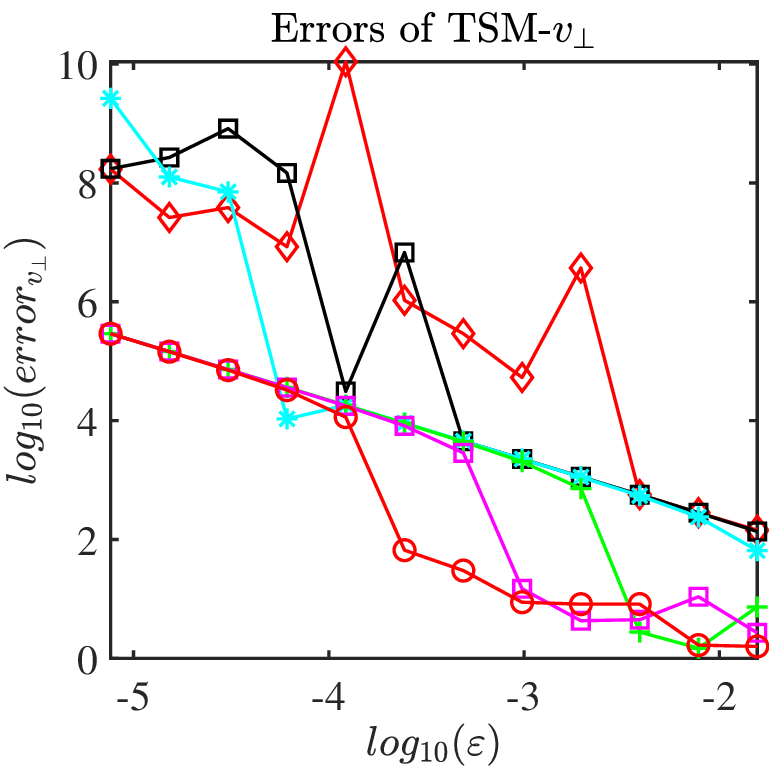}\\
		\includegraphics[width=3.8cm,height=3.2cm]{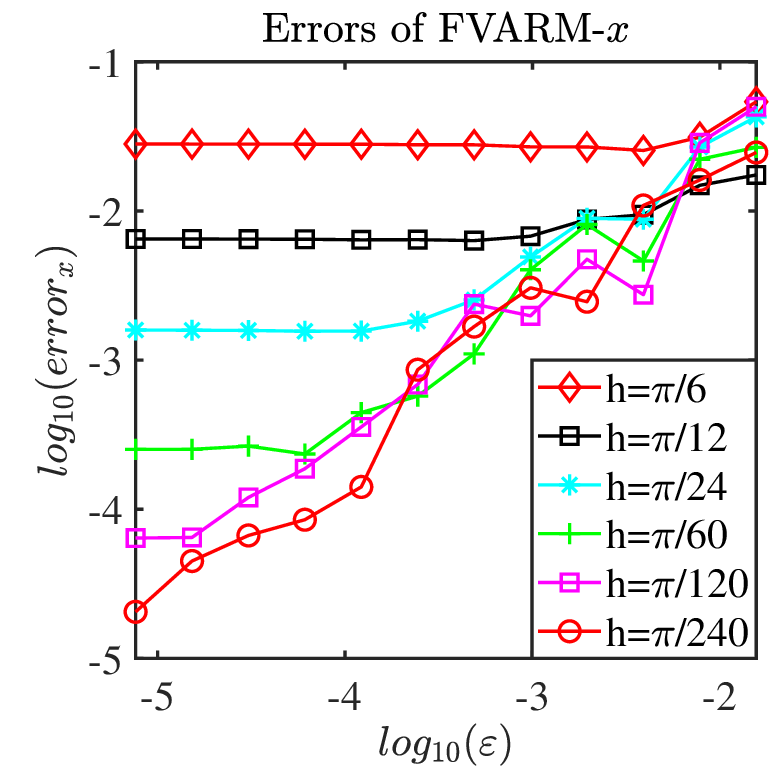}\quad \ \
		\includegraphics[width=3.8cm,height=3.2cm]{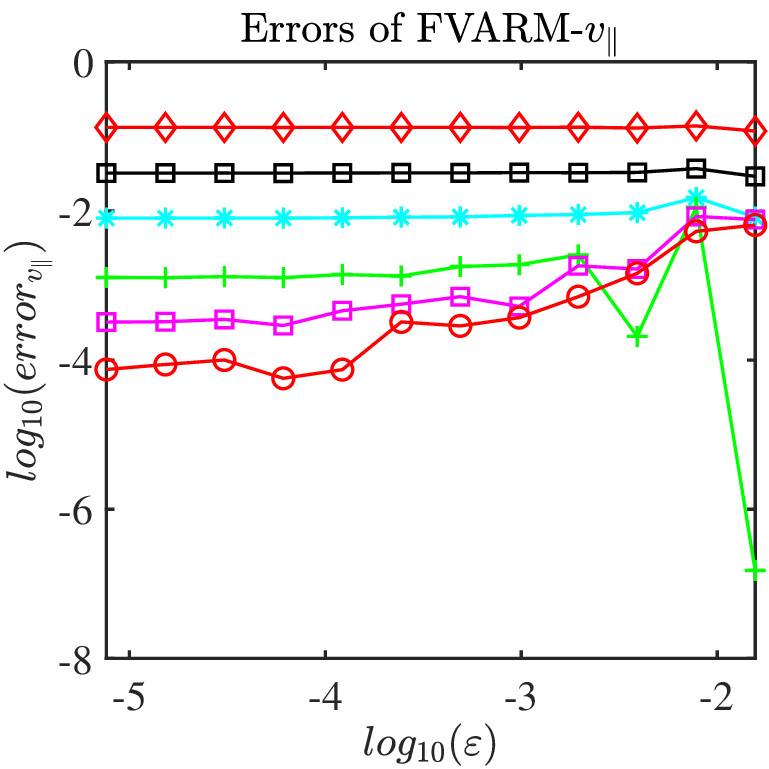}\quad \ \
		\includegraphics[width=3.8cm,height=3.2cm]{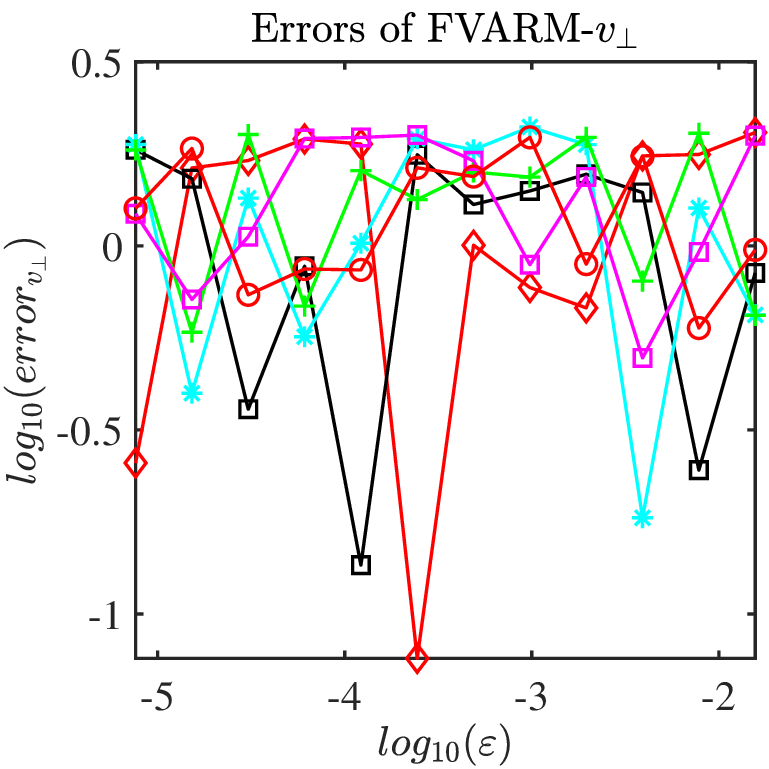}\\
		\includegraphics[width=3.8cm,height=3.2cm]{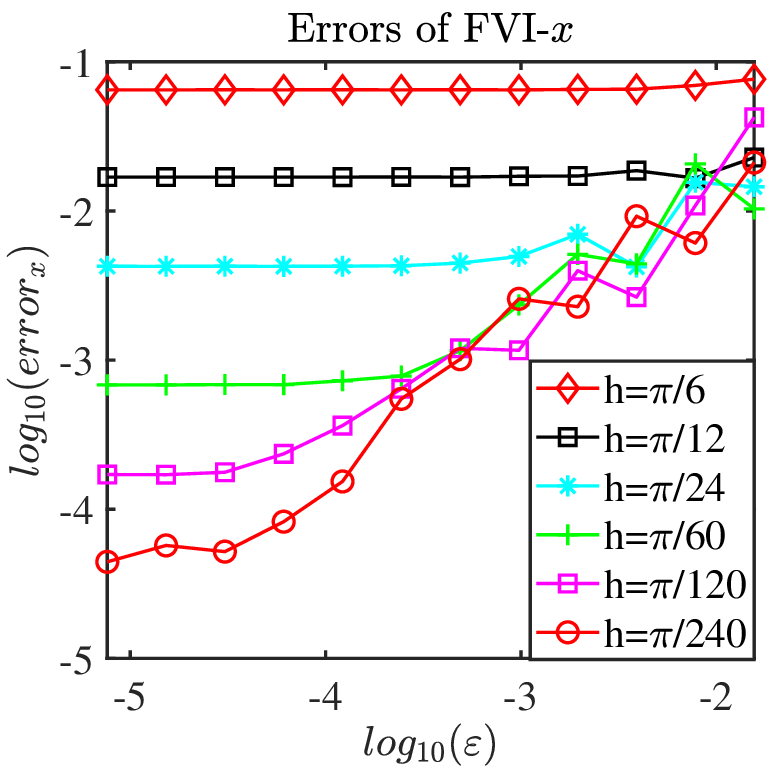}\quad \ \
		\includegraphics[width=3.8cm,height=3.2cm]{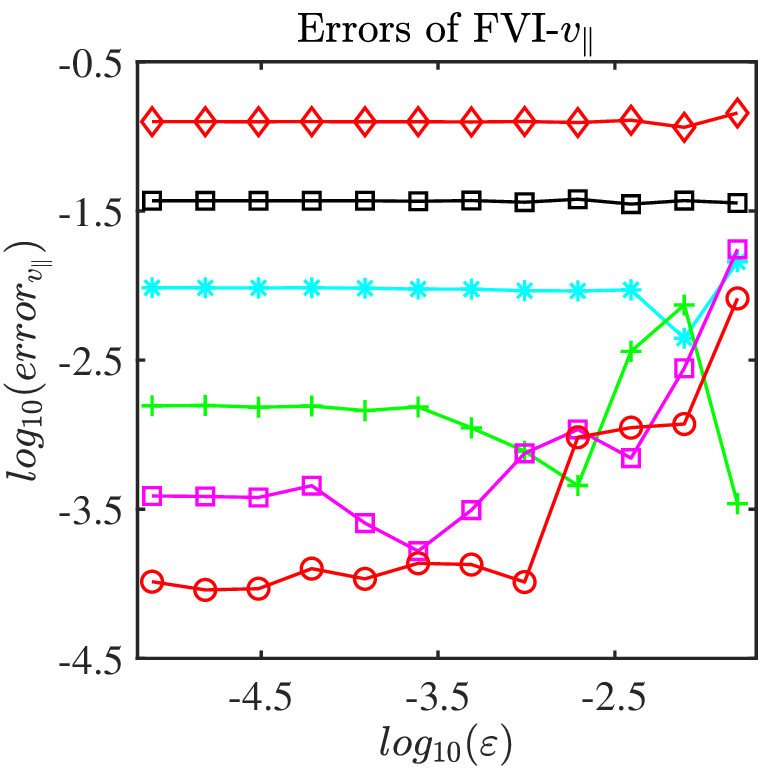}\quad \ \
		\includegraphics[width=3.8cm,height=3.2cm]{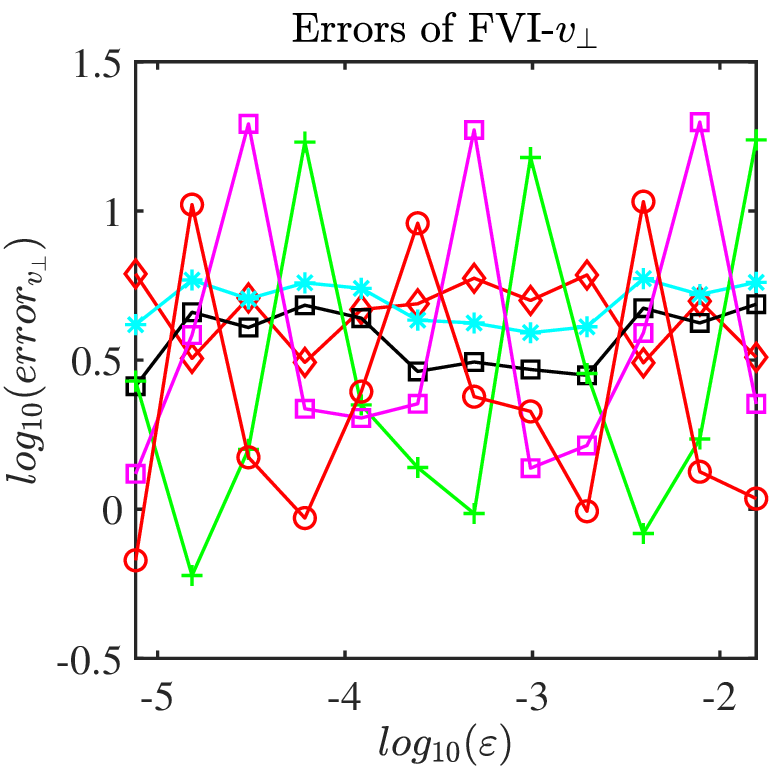}	
	\end{tabular}
	\caption{Problem 4. The global errors in $x$, $v_{\parallel}$  and $v_{\perp}$ at time $t=\pi/2$ vs. $\varepsilon$ ($\varepsilon=1/2^{k}, k=6,\ldots,17 $) with different $h$.}
	\label{fig:problem41}
\end{figure}

\begin{figure}[H]
	\centering\tabcolsep=0.5mm
	\begin{tabular}
		[c]{ccc}
		\includegraphics[width=3.8cm,height=3.2cm]{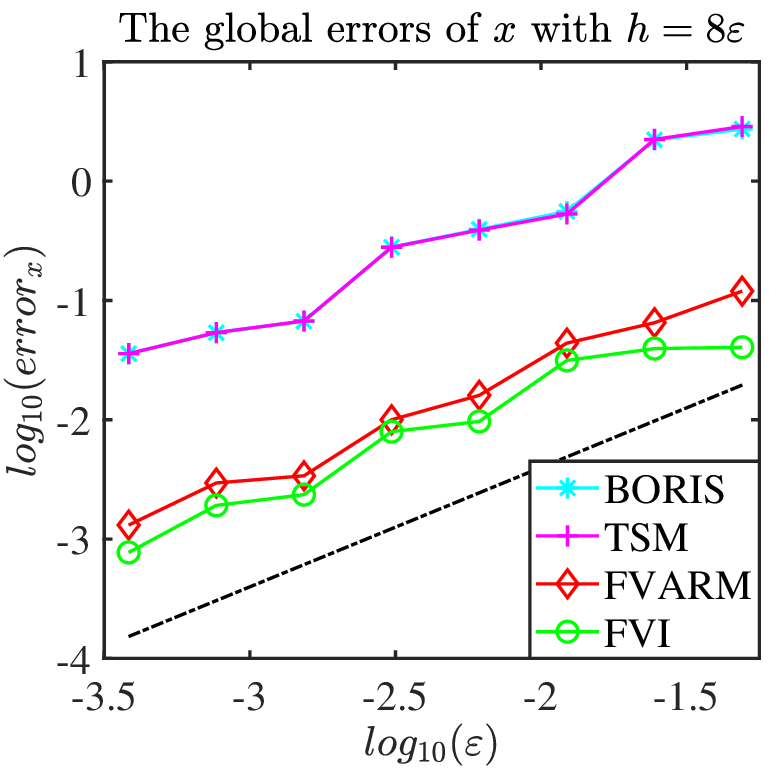}\quad \ \
		\includegraphics[width=3.8cm,height=3.2cm]{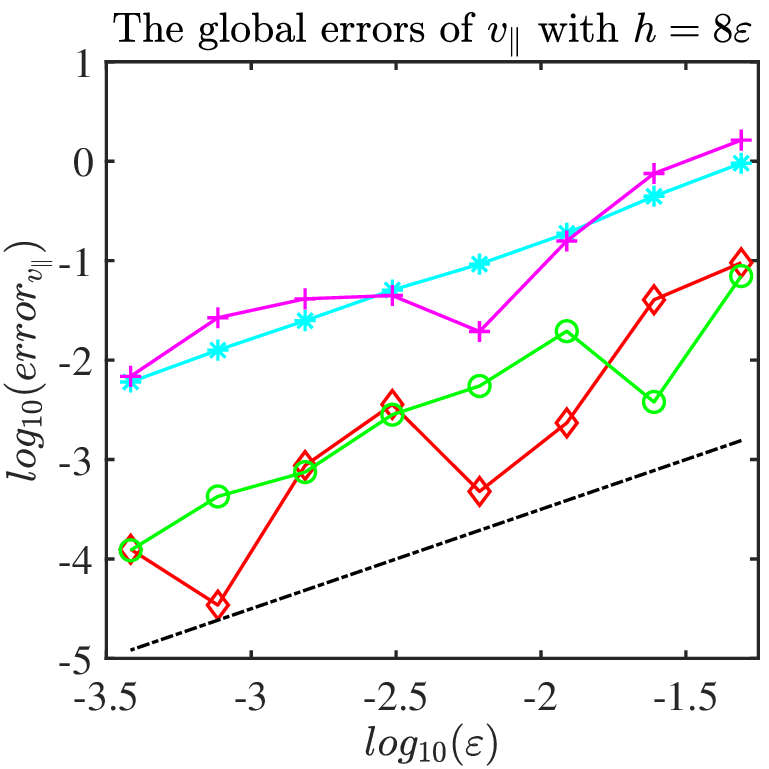}\quad \ \
		\includegraphics[width=3.8cm,height=3.2cm]{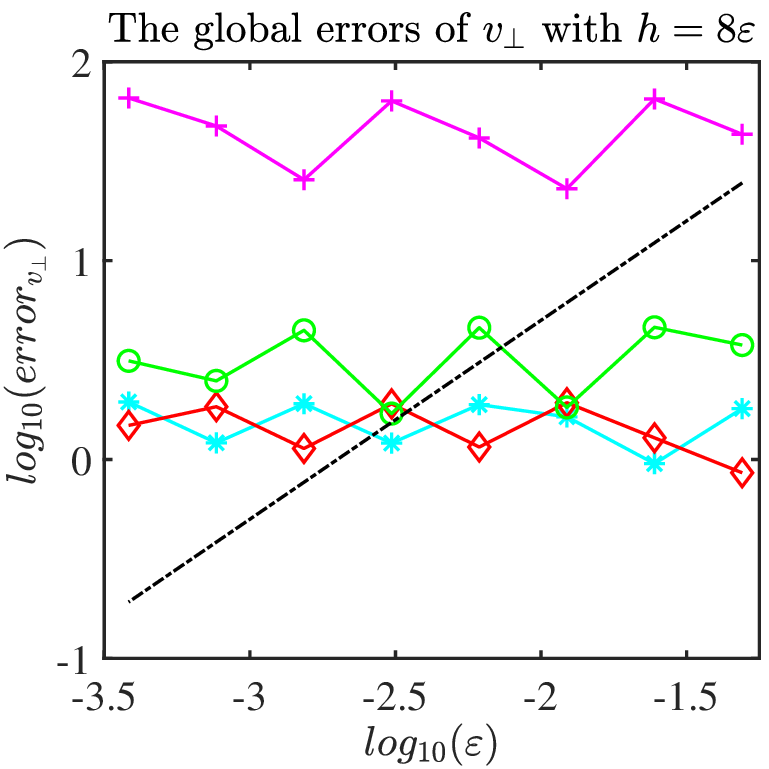} 
	\end{tabular}
	\caption{Problem 4. The global errors in $x$, $v_{\parallel}$  and $v_{\perp}$ at time $t=\pi/2$ vs. $\varepsilon$ ($\varepsilon=\pi/2^k$, $k=6,\ldots,13$ ) with different $h$ (the dash-dot line is slope one).}
	\label{fig:problem42}
\end{figure}

 \begin{figure}[H]
	\centering\tabcolsep=0.5mm
	\begin{tabular}
		[c]{cc}
		\includegraphics[width=5.8cm,height=2.8cm]{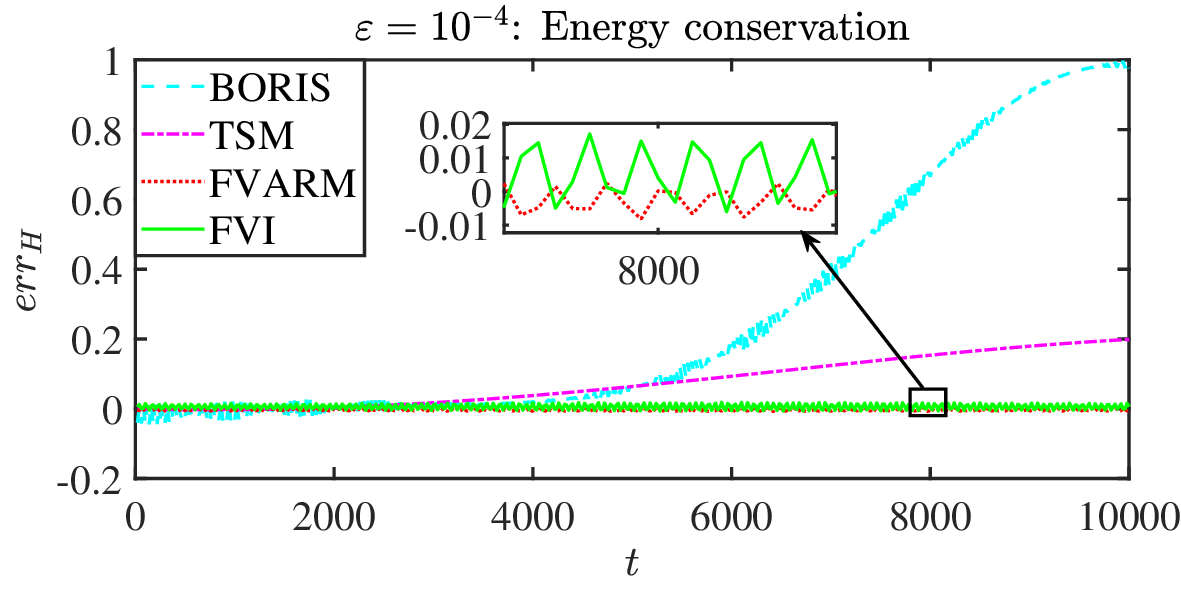}\qquad
		\includegraphics[width=5.8cm,height=2.8cm]{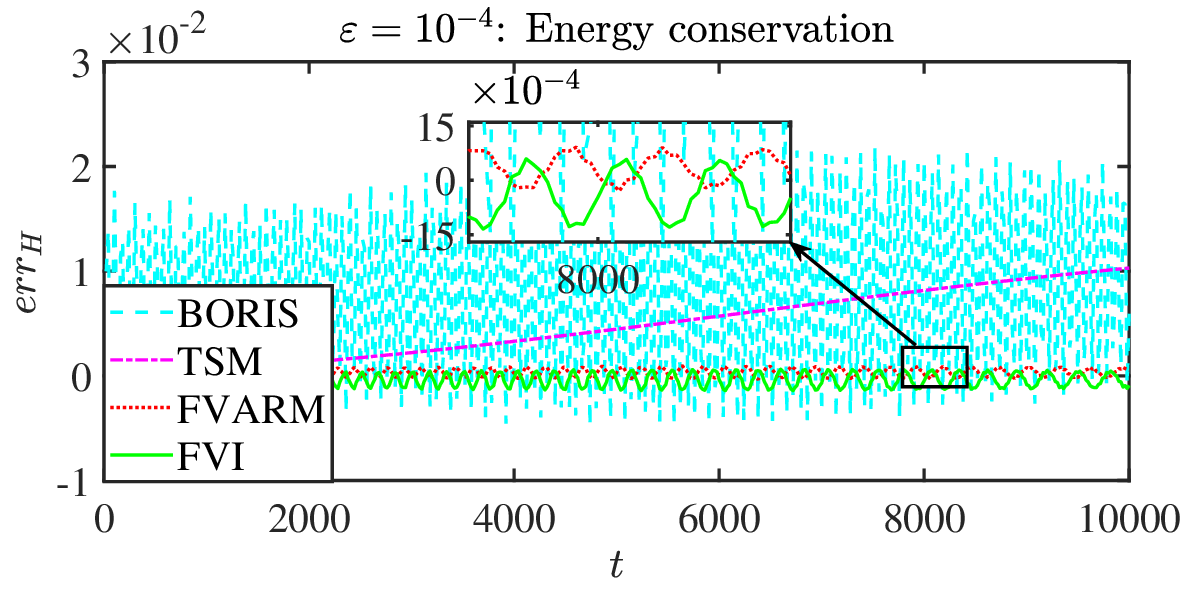}\\
		\includegraphics[width=5.8cm,height=2.8cm]{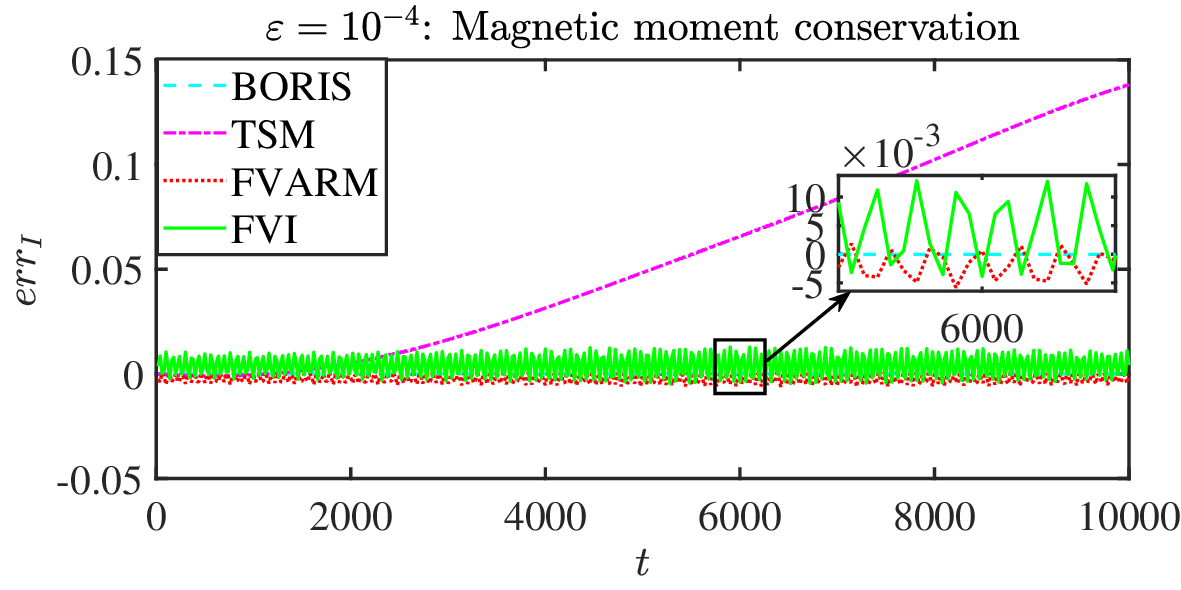}\qquad
		\includegraphics[width=5.8cm,height=2.8cm]{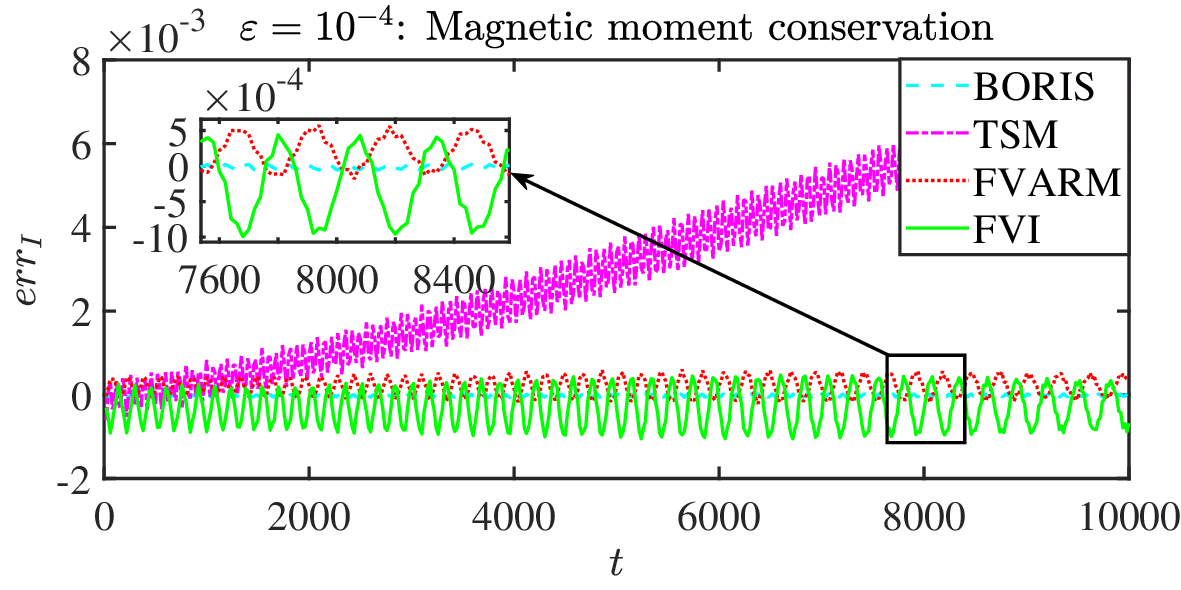}   
	\end{tabular}
	\caption{Problem 4. Evolution of energy error $e_H$ and magnetic moment error $e_{I}$ with step sizes $h = 0.01$ (left) and $h = 0.2\varepsilon$ (right).}
	\label{fig:problem43}
\end{figure}

The following key observations can be drawn from Figures \ref{fig:problem31}--\ref{fig:problem43}:

\textbf{Order behavior.}
Figures \ref{fig:problem31} and \ref{fig:problem41} show second-order convergence for the position $x$ and the parallel velocity $v_{\parallel}$ when $h^2 > C^{*}\varepsilon$. In the regime $c^{*}\varepsilon^2 \le h^2 \le C^{*}\varepsilon$, the global error curves are nearly parallel to a line of slope one, as shown in Figures \ref{fig:problem32} and \ref{fig:problem42}, confirming the $\mathcal{O}(\varepsilon)$ accuracy.
For $h^2 < c^{*}\varepsilon^2$, Figure~\ref{fig:problem33} shows the $\varepsilon$-dependence of the scaled errors defined by $error2_{x}:=\varepsilon\lvert x_{n}-x(t_n) \rvert /\lvert x(t_n)\rvert $, $error2_{v_{\parallel}}:= \varepsilon\lvert v_{\parallel}^{n}-v_{\parallel}(t_n)\rvert/\lvert v_{\parallel}(t_n)\rvert$ and $error2_{v_{\perp}}:=\varepsilon^2\lvert v_{\perp}^{n}-v_{\perp}(t_n)\rvert/\lvert v_{\perp}(t_n)\rvert$. The curves corresponding to different $\varepsilon$ nearly coincide, indicating that the errors in $x$ and $v_{\parallel}$ are of size $\mathcal{O}(h^2/\varepsilon)$, while the error in $v_{\perp}$ is of size $\mathcal{O}(h^2/\varepsilon^2)$. 

\textbf{Energy and magnetic moment behavior.}
Figures \ref{fig:problem34} (top row) and \ref{fig:problem43} (top row) illustrate that FVI and FVARM perform better in preserving energy than the other two methods. For $h^2 > C^{*}\varepsilon$, the energy error remains at the level $\mathcal{O}(h)$ over long times. In the regime $c^{*}\varepsilon^2 \le h^2 \le C^{*}\varepsilon$, the deviation further improves to $\mathcal{O}(\varepsilon)$. For $h^2 < c^{*}\varepsilon^2$, all four methods exhibit stable long-time energy behavior, and the energy error remains bounded, as illustrated in Figure~\ref{fig:problem35}, consistent with Theorem~\ref{conservation-H2}.
A similar behavior is observed for the magnetic moment errors. As shown in Figures \ref{fig:problem34} (bottom row) and \ref{fig:problem43} (bottom row), FVI, BORIS, and FVARM exhibit improved near-conservation compared with TSM in the first two step-size regimes. For $h^2 < c^{*}\varepsilon^2$, Figure~\ref{fig:problem35} shows that all four methods display comparable long-time behavior, with magnetic moment errors of order $\mathcal{O}(\varepsilon)$.

\section{Error bounds and long-term analysis in a moderate magnetic field (Proof of Theorems \ref{error bound 1}--\ref{conservation-M1})} \label{sec4}
In this section, the proof of error bounds and long-term analysis is given for  the filtered two-step variational integrator  \eqref{Method}--\eqref{AV1} applied to the CPD \eqref{charged-particle} in a moderate magnetic field $\varepsilon=1$. To demonstrate the long-time near-conservation of energy and momentum, the approach based on backward error analysis (see Chap. IX of \cite{hairer2006}) will be used. We first present the proof of the error bounds.
\subsection{Error bounds (Proof of Theorem \ref{error bound 1})} \label{sec4.1}
The error analysis consists of two parts: the local truncation error analysis of the two-step formulation \eqref{Method} and the corresponding global error analysis. 

 $\bullet$  \textbf{Proof of the local truncation error.} We begin with the exact solution to the charged particle dynamics \eqref{charged-particle}, which can be formulated via the variation-of-constants formula as follows:
 \begin{equation}\label{exact-solution}
 	\begin{aligned}
 		x(t_n+\tau h)&=x(t_n) +\tau hv(t_n)+h^2
 		\int_{0}^{\tau}(\tau-z)\tilde F(x(t_n+hz),v(t_n+hz))
 		dz,\\
 		v(t_n+\tau h)&=v(t_n)+h
 		\int_{0}^{\tau} \tilde F(x(t_n+hz),v(t_n+hz))  dz,
 	\end{aligned}
 \end{equation}
 where $\tilde{F}(x(t),v(t))=-\tilde{B}(x(t))v(t)+F(x(t))$. Letting $\tau=1$ and $\tau=-1$, the sum of these two equations yields
 \begin{equation}\label{exact-two step}
 	x(t_n+h)-2x(t_n)+x(t_n-h)=h^2\int_{0}^{1}(1-z)\big(\tilde F(x(t_n+hz),v(t_n+hz))+\tilde F(x(t_n-hz),v(t_n-hz))\big)dz.
 \end{equation}  
Inserting the exact solution \eqref{exact-solution} into the numerical scheme \eqref{Method}, we have
\begin{equation}\label{defect-two step}
	\begin{aligned}
		&	x(t_{n+1})-2x(t_{n})+x(t_{n-1})=\Psi\Big(\frac{h}{2}A'(x(t_{n+1/2}))^{\intercal}(x(t_{n+1})-x(t_{n}))
		+\frac{h}{2}A'(x(t_{n-1/2}))^{\intercal}(x(t_{n})-x(t_{n-1}))\\
		&\qquad \qquad \qquad \qquad \qquad \quad  -h\big(A(x(t_{n+1/2}))-A(x(t_{n-1/2}))\big)
 +\frac{h^2}{2}\big(F(x(t_{n+1/2}))+F(x(t_{n-1/2}))\big)	\Big)+\Delta,
	\end{aligned}
\end{equation}
where $x(t_{n\pm 1/2})=\big(x(t_{n})+x(t_{n\pm 1})\big)/2$ and $\Delta$  denotes the local truncation error.
Combining \eqref{exact-two step} with a Taylor expansion yields
\begin{equation*}
	\begin{aligned}
		\Delta&=h^2\int_{0}^{1}(1-z)\big(\tilde F(x(t_n+hz),v(t_n+hz))+\tilde F(x(t_n-hz),v(t_n-hz))\big)dz-\Psi\Big(h\big(A(x(t_{n-1/2})\\
		&\ -A(x(t_{n+1/2})\big)+\frac{h}{2}A'(x(t_{n+1/2}))^{\intercal}\big(x(t_{n+1})-x(t_{n})\big)
		+\frac{h}{2}A'(x(t_{n-1/2}))^{\intercal}\big(x(t_{n})-x(t_{n-1})\big) \\
	&\
		+\frac{h^2}{2}\big(F(x(t_{n+1/2}))+F(x(t_{n-1/2}))\big)
		\Big)\\
		&=h^2\int_{0}^{1}2(1-z)\Big(\tilde F(x(t_n),v(t_n))+\frac{h^2z^2}{2!}\tilde F^{(2)}(x(t_n),v(t_n))+\frac{h^4z^4}{4!}\tilde F^{(4)}(x(t_n),v(t_n))+\mathcal{O}(h^6) \Big)dz\\
		&\
		-\Psi\bigg(-h\bigg(A\Big(x(t_{n})+\frac{h}{2}\dot{x}(t_n)+\frac{h^2}{4}\ddot{x}(t_n)+\mathcal{O}(h^3)\Big)
		-A\Big(x(t_{n})-\frac{h}{2}\dot{x}(t_n)+\frac{h^2}{4}\ddot{x}(t_n)+\mathcal{O}(h^3)\Big)\bigg)\\
		&\ +\frac{h^2}{2}\bigg(F\Big(x(t_{n})+\frac{h}{2}\dot{x}(t_n)+\frac{h^2}{4}\ddot{x}(t_n)+\mathcal{O}(h^3)\Big)
		+F\Big(x(t_{n})-\frac{h}{2}\dot{x}(t_n)+\frac{h^2}{4}\ddot{x}(t_n)+\mathcal{O}(h^3)\Big)\bigg)\\
		&\ +\frac{h}{2}A'\Big(x(t_{n})+\frac{h}{2}\dot{x}(t_n)+\mathcal{O}(h^2)\Big)^{\intercal}\big(h\dot{x}(t_n)+\mathcal{O}(h^2)\big)
		+\frac{h}{2}A'\Big(x(t_{n})-\frac{h}{2}\dot{x}(t_n)+\mathcal{O}(h^2)\Big)^{\intercal}\big(h\dot{x}(t_n)+\mathcal{O}(h^2)\big)	\bigg)\\
		&=h^2\big(-\tilde{B}(x(t_n))v(t_n)+F(x(t_n))\big)-
		h^2\Psi\Big(\big (A'(x(t_n))^{\intercal}-A'(x(t_n))\big)\dot{x}(t_n)+F(x(t_n))\Big)+\mathcal{O}(h^4)=\mathcal{O}(h^4),
	\end{aligned}
\end{equation*}
where $(\cdot)^{(m)}$ denotes the $m$th derivative with respect to the variable. Here, we use $v \times B(x)=\big(A'(x)^{\intercal}-A'(x)\big)v$  and the definition of $\Psi$\label{MVF}, along with the Taylor expansion
$\mathrm{tanc}(h/2)=1+h^2/12+\mathcal{O}(h^4)$.

$\bullet$ \textbf{Proof of the global error.} The global error of the two-step formulation \eqref{Method} is denoted by $e_n^x = x(t_n) - x_n.$
Subtracting the formula \eqref{Method} from \eqref{defect-two step} yields
\begin{equation*}
	\begin{aligned}
		&e_{n+1}^{x}-2e_{n}^{x}+e_{n-1}^{x}	=\Psi\Big(\frac{h}{2}A'(x(t_{n+1/2}))^{\intercal}\big(x(t_{n+1})-x(t_{n})\big)
		+\frac{h}{2}A'(x(t_{n-1/2}))^{\intercal}\big(x(t_{n})-x(t_{n-1})\big)\\
		& \ \ -h\big(A(x(t_{n+1/2}))-A(x(t_{n-1/2}))\big)
		+\frac{h^2}{2}\big(F(x(t_{n+1/2}))+F(x(t_{n-1/2}))\big)-\frac{h}{2}A'(x_{n+1/2})^{\intercal}(x_{n+1}-x_{n})\\
		& \ \
		-\frac{h}{2}A'(x_{n-1/2})^{\intercal}(x_{n}-x_{n-1})
		+h\big(A(x_{n+1/2})-A(x_{n-1/2})\big)
		-\frac{h^2}{2}\big(F(x_{n+1/2})+F(x_{n-1/2})\big)	\Big)+\Delta.
	\end{aligned}
\end{equation*}
 By performing some calculations, it can be deduced that
\begin{equation*}
	\begin{aligned}
		&e_{n+1}^{x}-2e_{n}^{x}+e_{n-1}^{x}	=\Psi\bigg(\frac{h}{2}\Big(A'(x(t_{n+1/2}))^{\intercal}\big(x(t_{n+1})-x_{n+1}\big)+A'(x(t_{n+1/2}))^{\intercal}x_{n+1}-A'(x_{n+1/2})^{\intercal}x_{n+1}\\
		&\ \
		-A'(x(t_{n+1/2}))^{\intercal}(x(t_{n})-x_{n})-A'(x(t_{n+1/2}))^{\intercal}x_{n}+A'(x_{n+1/2})^{\intercal}x_{n}\Big)+\frac{h}{2}\Big(A'(x(t_{n+1/2}))^{\intercal}\big(x(t_{n})-x_{n}\big)\\
		&\ \
		+A'(x(t_{n+1/2}))^{\intercal}x_{n}-A'(x_{n+1/2})^{\intercal}x_{n}-A'(x(t_{n+1/2}))^{\intercal}\big(x(t_{n-1})-x_{n-1}\big)-A'(x(t_{n+1/2}))^{\intercal}x_{n-1}\\
		&\ \ +A'(x_{n+1/2})^{\intercal}x_{n-1}\Big)
		 -h\Big(A(x(t_{n+1/2}))-A(x_{n+1/2})-\big(A(x(t_{n-1/2}))-A(x_{n-1/2})\big)\Big)\\
		& \ \
		+\frac{h^2}{2}\Big(\big(F(x(t_{n+1/2}))-F(x_{n+1/2})\big)
		+\big(F(x(t_{n-1/2}))-F(x_{n-1/2})\big)\Big) 	\bigg)+\Delta.
	\end{aligned}
\end{equation*}
Then we obtain 
$
\abs{e_{n+1}^{x}-2e_{n}^{x}+e_{n-1}^{x}} \leq Ch\big(\abs{e_{n+1}^{x}}+\abs{e_{n}^{x}}+\abs{e_{n-1}^{x}}\big)+Ch^4.
$
With $\delta_n := e_n^{x}-e_{n-1}^{x}$, this estimate becomes
\begin{equation}\label{D1}
	\abs{\delta_{n+1}-\delta_n}
	\le
	C h\big(\abs{e_{n+1}^{x}}+\abs{e_{n}^{x}}+\abs{e_{n-1}^{x}}\big)
	+ C h^4.
\end{equation}
Summing \eqref{D1} from $m=1$ to $n$ yields
$
	\abs{\delta_{n+1}}
	\le
	\abs{\delta_1}
	+ C h \sum_{m=1}^{n}
	\big(\abs{e_{m+1}^{x}}+\abs{e_m^{x}}+\abs{e_{m-1}^{x}}\big)
	+ C n h^4.
$
Since $n h \le T$, the last term satisfies $C n h^4 \le C T h^3$, and hence
\begin{equation}\label{D3}
	\abs{\delta_{n+1}}
	\le
	\abs{\delta_1}
	+ C h \sum_{m=0}^{n+1} \abs{e_m^{x}}
	+ C T h^3.
\end{equation}
Since $e_0^{x}=0$, we have
$
e_n^{x} = \sum_{j=1}^{n} \delta_j$ and 
$
\abs{e_n^{x}}
\le
\sum_{j=1}^{n} \abs{\delta_j}.
$
A further summation of \eqref{D3} over $j=1,\dots,n$ gives
$
  \sum_{j=1}^{n} \abs{\delta_j}
\le
n \abs{\delta_1}
+ C h \sum_{j=1}^{n} \sum_{m=0}^{j} \abs{e_m^{x}}
+ C n T h^3.
$
Noting that
$
\sum_{j=1}^{n} \sum_{m=0}^{j} \abs{e_m^{x}}
=
\sum_{m=1}^{n} (n-m+1) \abs{e_m^{x}}
\le
(n+1) \sum_{m=1}^{n} \abs{e_m^{x}},
$
and using $n \le T/h$, we obtain
$
	\abs{e_n^{x}}
	\leq
	T/h \abs{\delta_1}
	+ C T \sum_{m=1}^{n} \abs{e_m^{x}}
	+ C T^2 h^2.
$
By the discrete Gronwall inequality, if $\delta_1:=e_{1}^{x}=\mathcal{O}(h^3)$, then $e_n^{x}=\mathcal{O}(h^2)$.

Following a similar approach as in the previous analysis, we now derive the value of $e_{1}^{x}$. Referring to the formulation process in Section \ref{sec2}, the starting value $x_1$ is expressed as
\begin{equation}\label{start_x1}
	\begin{aligned}
		x_{1}
		=x_0+h\Psi\big(p_0
		+A'\big((x_0+ x_{1})/2\big)^{\intercal}(x_{1}-x_{0})/2
		-A\big((x_0+ x_{1})/2\big)	+hF\big((x_0+ x_{1})/2\big)/2\big),
	\end{aligned}
\end{equation}
where $p_0$ is obtained by evaluating
 $p(t)=v(t)+A(x(t))$ at $t=0$.
Then the variation-of-constants formula \eqref{exact-solution} gives
\begin{equation}\label{eaxct_x1}
		x(t_0+ h)=x(t_0) +hv(t_0)+h^2
		\int_{0}^{1}(1-z)\tilde F(x(t_0+hz),v(t_0+hz))
		dz.
\end{equation}
Inserting the exact solution \eqref{eaxct_x1} into \eqref{start_x1}, we have
\begin{equation}\label{local_x1}
	\begin{aligned}
		x(t_1)
		=&x(t_0)+h\Psi\Big(p(t_0)
		+A'\big((x(t_0)+ x(t_1))/2\big)^{\intercal}\big(x(t_1)-x(t_0)\big)/2
		-A\big((x(t_0)+ x(t_1))/2\big)\\	&+hF\big((x(t_0)+ x(t_1))/2\big)/2\Big)
		+	\Delta_1,
	\end{aligned}
\end{equation}
where $\Delta_1$ denotes the local truncation error at the first step. From the above two formulae and the definition of $\Psi$, it follows that
\begin{equation*}
	\begin{aligned}
	\Delta_1
		&=hv(t_0)+h^2\int_{0}^{1}(1-z)\tilde F(x(t_0+hz),v(t_0+hz))
		dz -h\Psi\Big(A'\big((x(t_0)+ x(t_1))/2\big)^{\intercal}\big(x(t_1)-x(t_0)\big)/2\\
		&\quad  +p(t_0)
		-A\big((x(t_0)+ x(t_1))/2\big)	+hF\big((x(t_0)+ x(t_1))/2\big)/2\Big)\\
		&=hv(t_0)+h^2\tilde F(x(t_0),v(t_0))/2
		-h\Big(v(t_0)+A(x(t_0))
		+hA'(x(t_0))^{\intercal}\dot{x}(t_0)/2-A(x(t_0))\\
		&\quad 
		-hA'(x(t_0))\dot{x}(t_0)/2	+hF(x(t_0))/2\Big)+\mathcal{O}(h^3).
	\end{aligned}
\end{equation*}
By combining the expression $\tilde{F}(x(t),v(t))=-\tilde{B}(x(t))v(t)+F(x(t))$ with the fact $v \times B(x)=\big(A'(x)^{\intercal}-A'(x)\big)v$, we obtain  $\Delta_1=\mathcal{O}(h^3)$. Then, comparing  \eqref{local_x1} with \eqref{start_x1} yields
\begin{equation*}
	\begin{aligned}
	e_{1}^{x}
	&=x(t_1)-x_1
	=x(t_0)-x_0+h\Psi\bigg(p(t_0)
	+\frac{1}{2}A'\Big(\frac{x(t_0)+ x(t_1)}{2}\Big)^{\intercal}\big(x(t_1)-x(t_0)\big)
	-A\Big(\frac{x(t_0)+ x(t_1)}{2}\Big)\\	
	&\quad +\frac{h}{2}F\Big(\frac{x(t_0)+ x(t_1)}{2}\Big)\bigg)-h\Psi\bigg(p_0
	+\frac{1}{2}A'\Big(\frac{x_0+ x_{1}}{2}\Big)^{\intercal}(x_{1}-x_{0})
	-A\Big(\frac{x_0+ x_{1}}{2}\Big)	+\frac{h}{2}F\Big(\frac{x_0+ x_{1}}{2}\Big)\bigg)+\Delta_1\\
	&=\frac{h}{2}A'\Big(\frac{x(t_0)+ x(t_1)}{2}\Big)^{\intercal}(x(t_1)-x_1)+\frac{h}{2}A'\Big(\frac{x(t_0)+ x(t_1)}{2}\Big)^{\intercal}x_1-\frac{h}{2}A'\Big(\frac{x_0+ x_{1}}{2}\Big)^{\intercal}x_{1}\\
	&\quad -\frac{h}{2}A'\Big(\frac{x(t_0)+ x(t_1)}{2}\Big)^{\intercal}(x(t_0)-x_0)-\frac{h}{2}A'\Big(\frac{x(t_0)+ x(t_1)}{2}\Big)^{\intercal}x_0+\frac{h}{2}A'\Big(\frac{x_0+ x_{1}}{2}\Big)^{\intercal}x_{0}\\
	&\quad -h\Big(A\Big(\frac{x(t_0)+ x(t_1)}{2}\Big)-A\Big(\frac{x_0+ x_{1}}{2}\Big)\Big)+\frac{h^2}{2}\Big(F\Big(\frac{x(t_0)+ x(t_1)}{2}\Big)-F\Big(\frac{x_0+ x_{1}}{2}\Big)\Big)+\mathcal{O}(h^3)+\Delta_1.
	\end{aligned}
\end{equation*}
This gives 
\begin{equation*}
		\abs{e_{1}^{x}} \leq C(h+h^2)\abs{e_{1}^{x}}+Ch^3,
\end{equation*}	
where we use the fact that $e_{0}^{x}=0$. For sufficiently small $h$, this implies $e_{1}^{x}=\mathcal{O}(h^3)$. Thus, the global error in position of Theorem \ref{error bound 1} is obtained. Similarly, the global error in velocity can also be derived.

\subsection{Conservation properties (Proof of Theorems \ref{conservation-H1}--\ref{conservation-M1})} 
This section is devoted to the analysis of the long-time near-conservation of energy and momentum based on backward error analysis, which is an important tool for understanding
the long-time behavior of numerical methods. We first establish the long-time near-conservation of energy below.

\subsubsection{Long-time  near-conservation of energy (Proof of Theorem \ref{conservation-H1})} \label{sec4.2}
The proof of the long-time near-conservation of energy proceeds in three steps. First, we search for a modified differential equation whose solution is formally equivalent to the numerical solution. Next, we show that this modified equation admits a formal first integral that remains close to the total energy $H$.
Finally, the long-time near-conservation of energy is obtained
by extending the analysis from short to long time intervals.

$\bullet$  \textbf{Modified differential equation.} We seek a modified differential equation in the form of a formal power series in $h$ whose solution $p(t)$ formally satisfies $p(nh)=x_n$, where $x_n$ denotes the numerical solution generated by the filtered two-step variational integrator. Such a function has to satisfy  
\begin{equation}\label{modif}
	\begin{aligned}
		&p(t+h)-2p(t)+p(t-h)
		=\Psi\big(-h\big(A((p(t+h)+p(t))/2)-A((p(t)+p(t-h))/2)\big)\\
		&\ +hA'((p(t+h)+p(t))/2)^{\intercal}(p(t+h)-p(t))/2 +hA'((p(t)+p(t-h))/2)^{\intercal}(p(t)-p(t-h))/2
		\\
		&\ +h^2\big( F((p(t+h)+p(t))/2)+F((p(t)+p(t-h))/2)\big)/2	\big).
	\end{aligned}
\end{equation}
Define the operators
	\begin{equation}\label{oper}
	L_1(e^{hD})=e^{hD}-1,\quad L_2(e^{hD})=(e^{hD}+1)/2,
\end{equation}
where $D$ denotes time differentiation and $e^{hD}$
is the corresponding shift operator. Using the Taylor expansion, we have
\begin{equation}\label{operators-expasion}
	\begin{aligned}
		L_{1}^{2}L_{2}^{-2}(e^{hD})=h^2D^2-\frac{1}{6}h^4D^4+\frac{17}{720}h^6D^6+\cdots, \ \
		L_{1}L_{2}^{-1}(e^{hD})=hD-\frac{1}{12}h^3D^3+\frac{1}{120}h^5D^5+\cdots.
	\end{aligned}
\end{equation}
Then letting $$\tilde{y}(t):=(p(t)+p(t+h))/2=L_2(e^{hD})p(t),$$ 
the formula \eqref{modif} becomes
\begin{equation}\label{modif2}
\begin{aligned}
&\frac{1}{h^2}(e^{hD}-2+e^{-hD})L_2^{-1}(e^{hD})\tilde{y}(t)=\Psi\Big(-\frac{1}{h}\big(A(\tilde{y}(t))-A(\tilde{y}(t-h))\big)+\frac{1}{2h}A'(\tilde{y}(t))^{\intercal}\big(L_1 L_2^{-1}(e^{hD})\tilde{y}(t)\big)\\
&\quad +\frac{1}{2h}A'(\tilde{y}(t-h))^{\intercal}\big(L_1 L_2^{-1}(e^{hD})\tilde{y}(t-h)\big)+\frac{1}{2}\big(F(\tilde{y}(t))+F(\tilde{y}(t-h))\big)\Big)\\
&\ =L_2(e^{-hD})\Big(\Psi\Big(\frac{1}{h}A'(\tilde{y}(t))^{ \intercal}(L_1 L_2^{-1}(e^{hD})\tilde{y}(t))
-\frac{1}{h}L_1 L_2^{-1}(e^{hD})A(\tilde{y}(t))
+F(\tilde{y}(t))\Big)\Big).
\end{aligned}
\end{equation}	
Here, we have used the definition for the derivatives of $f: \mathbb{R}^3 \to \mathbb{R}^3$  with respect to $\tilde{y}$ and its time derivatives. For example,  the first-order derivative is given by $Df(\tilde{y})=f'(\tilde{y})\dot{\tilde{y}}$  and the second-order derivative is expressed as $D^2f(\tilde{y})=f''(\dot{\tilde{y}},\dot{\tilde{y}})+f'(\tilde{y})\ddot{\tilde{y}}$ (see \cite{hairer2004}).  Moreover, we note that $e^{hD}\tilde{y}(t)=\tilde{y}(t+h)$ and the operator $L_{1}L_{2}^{-1}(e^{hD})$
is an odd operator. Based on the definitions of $L_1$ and $L_2$, it follows that
\begin{equation*}\label{}
e^{hD}-2+e^{-hD}
=L_2(e^{-hD})L_1^2L_2^{-1}(e^{hD}).
\end{equation*}	
Substituting this identity into \eqref{modif2}, the expression can be reformulated as
\begin{equation}\label{modified-eq}
\begin{aligned}
\frac{1}{h^2}L_1^2 L_2^{-2}(e^{hD})\tilde{y}
=\Psi  \Big(\frac{1}{h}A'(\tilde{y})^{ \intercal}(L_1 L_2^{-1}(e^{hD})\tilde{y})
-\frac{1}{h}L_1 L_2^{-1}(e^{hD})A(\tilde{y})
+F(\tilde{y})\Big).
\end{aligned}
\end{equation}	
Inserting the expansions from \eqref{operators-expasion} into \eqref{modified-eq}, one arrives at
\begin{equation*}
	\begin{aligned}
	\ddot{\tilde{y}}-\frac{1}{6}h^2\tilde{y}^{(4)}+\frac{17}{720}h^4\tilde{y}^{(6)}+\cdots
	=
	&\Psi  \Big(A'(\tilde{y})^{ \intercal}\dot{\tilde{y}}-\frac{1}{12}h^{2}A'(\tilde{y})^{ \intercal}\tilde{y}^{(3)}+\frac{1}{120}
	h^{4}A'(\tilde{y})^{ \intercal}\tilde{y}^{(5)}\\
	&\quad \ -A'(\tilde{y})\dot{\tilde{y}}+\frac{1}{12}h^2(A(\tilde{y}))^{(3)}-\frac{1}{120}h^4(A(\tilde{y}))^{(5)}+F(\tilde{y})+\cdots\Big).
		\end{aligned}
\end{equation*}
 For $h = 0$, the equation reduces to \eqref{charged-particle}, which uses the fact that $\dot{\tilde{y}}\times B(\tilde{y})=\big(A'(\tilde{y})^{\intercal}-A'(\tilde{y})\big)\dot{\tilde{y}}$. Higher-order derivatives (of third order and above) can be recursively eliminated by successive differentiation of the equation followed by setting $h=0$ at each step.
As a result, we obtain a second-order modified differential equation whose right-hand side is a formal series in even powers of $h$: 
\begin{equation}\label{modified}
\ddot{\tilde{y}}=\Psi  \big(\dot{\tilde{y}}\times B(\tilde{y})+F(\tilde{y})+h^2g_2(\tilde{y},\dot{\tilde{y}})+h^4g_4(\tilde{y},\dot{\tilde{y}})+\cdots\big).
\end{equation}
The coefficient functions $g_{2i}(\tilde{y},\dot{\tilde{y}})$ ($i=1,2,\ldots$) are independent of $h$ and uniquely determined.

$\bullet$  \textbf{A formal first integral of the energy $H$.} We show that the modified differential equation \eqref{modified} admits a formal first integral that remains close to the total energy $H$. Rewriting \eqref{modified-eq} and multiplying it with
 $\dot{\tilde{y}}^{\intercal}$
yields
	\begin{equation}\label{modified-eq2}
	\begin{aligned}
		h^{-2}\dot{\tilde{y}}^{\intercal}\Psi^{-1}(L_1^2 L_2^{-2}(e^{hD})\tilde{y})-\dot{\tilde{y}}^{\intercal}F(\tilde{y})
		=h^{-1}\dot{\tilde{y}}^{\intercal}  \big(A'(\tilde{y})^{ \intercal}(L_1 L_2^{-1}(e^{hD})\tilde{y})
		-L_1 L_2^{-1}(e^{hD})A(\tilde{y})
		\big).
	\end{aligned}
\end{equation}	
Following the idea of backward error analysis, we shall show that both sides of \eqref{modified-eq2} are total derivatives. Moreover, there exist functions $H_h(\tilde{y},\dot{\tilde{y}})$ and $H_{2j}(\tilde{y},\dot{\tilde{y}})$ such that $$H_{h}(\tilde{y},\dot{\tilde{y}})=H(\tilde{y},\dot{\tilde{y}})+\sum\limits_{j=1}^{\infty}h^{2j}H_{2j}(\tilde{y},\dot{\tilde{y}}).$$
The argument is demonstrated in the following two steps.

\textsl{(i)} By the definition of $\Psi$, we have $\Psi^{-1}
=I+\big(1-\mathrm{tanc}(h/2))^{-1}\big){\tilde{B}_0}^2.$ Then the left-hand side of  \eqref{modified-eq2}  can be rewritten  as
	\begin{equation*}
	\begin{aligned}
		\dot{\tilde{y}}^{\intercal}\Big(\ddot{\tilde{y}}-\frac{1}{6}h^2\tilde{y}^{(4)}+\frac{17}{720}h^4\tilde{y}^{(6)}+\cdots \Big)
		+\dot{\tilde{y}}^{\intercal}\Big(1-\mathrm{tanc}\Big(\frac{h}{2 }\Big)^{-1}\Big){\tilde{B}_0}^2\Big(\ddot{\tilde{y}}-\frac{1}{6}h^2\tilde{y}^{(4)}+\frac{17}{720}h^4\tilde{y}^{(6)}+\cdots \Big)-\dot{\tilde{y}}^{\intercal}F(\tilde{y}).
	\end{aligned}
\end{equation*}	
For the first term of the above formula, we note that $\dot{\tilde{y}}^{\intercal}\ddot{\tilde{y}}=\frac{1}{2}\frac{\mathrm{d}}{\mathrm{d}t}(\dot{\tilde{y}}^{\intercal}\dot{\tilde{y}})$  and $\dot{\tilde{y}}^{\intercal}\tilde{y}^{(4)}=\frac{\mathrm{d}}{\mathrm{d}t}\big(\dot{\tilde{y}}^{\intercal}\tilde{y}^{(3)}-\frac{1}{2}\ddot{\tilde{y}}^{\intercal}\ddot{\tilde{y}}\big)$. Furthermore, according to \cite{hairer2006}, the higher derivatives satisfy
\begin{equation*}
\dot{\tilde{y}}^{\intercal}\tilde{y}^{(2l)}=\frac{\mathrm{d}}{\mathrm{d}t}\Big(\dot{\tilde{y}}^{\intercal}\tilde{y}^{(2l-1)}-\ddot{\tilde{y}}^{\intercal}\tilde{y}^{(2l-2)}+\ldots \pm \frac{1}{2}\tilde{y}^{(l)\intercal}\tilde{y}^{(l)}\Big).
\end{equation*}	
Hence the left-hand side of the formula \eqref{modified-eq2} can be transformed into
\begin{equation*}\label{left-modified-eq2}
	\frac{\mathrm{d}}{\mathrm{d}t}\Big(\underbrace{\frac{1}{2}\dot{\tilde{y}}^{\intercal}\dot{\tilde{y}}+U(\tilde{y})}_{H(\tilde{y},\dot{\tilde{y}})}+h^2\tilde{g}_2(\tilde{y},\dot{\tilde{y}})+h^4\tilde{g}_4(\tilde{y},\dot{\tilde{y}})+\cdots\Big),
	\end{equation*}	
where we have used the fact that  $F(\tilde{y})=-\nabla_{\tilde{y}}U(\tilde{y})$, the antisymmetric property of $\tilde{B}_{0}$ and the Taylor expansion
$\mathrm{tanc}(h/2)^{-1}=1-h^{2}/12-h^{4}/720+\mathcal{O}(h^6).$

\textsl{(ii)} In what follows, we shall show the right-hand side of formula \eqref{modified-eq2} can also be expressed as the time derivative of an expression of $(\tilde{y},\dot{\tilde{y}})$. Thanks to the Euclidean inner product $	\langle a, b\rangle=a^{\intercal}b $, we have
\begin{equation*}
	\begin{aligned}
	&h^{-1}\dot{\tilde{y}}^{\intercal}  \big(A'(\tilde{y})^{ \intercal}(L_1 L_2^{-1}(e^{hD})\tilde{y})
		-L_1 L_2^{-1}(e^{hD})A(\tilde{y})
		\big)\\
		&=h^{-1} \big(	\langle \dot{\tilde{y}} ,\ A'(\tilde{y})^{ \intercal}(L_1 L_2^{-1}(e^{hD})\tilde{y}) \rangle-\langle L_1 L_2^{-1}(e^{hD})A(\tilde{y}),\ D\tilde{y} \rangle\big)\\
		&=h^{-1} \big(	\langle A'(\tilde{y})\dot{\tilde{y}},\ L_1 L_2^{-1}(e^{hD}\tilde{y}) \rangle-\langle L_1 L_2^{-1}(e^{hD})A(\tilde{y}),\ D\tilde{y} \rangle\big)\\
		&=h^{-1} \big(	\langle DA(\tilde{y}),\ L_1 L_2^{-1}(e^{hD}\tilde{y}) \rangle-\langle L_1 L_2^{-1}(e^{hD})A(\tilde{y}),\ D\tilde{y} \rangle\big).
	\end{aligned}
\end{equation*}	
It follows from \cite{Hairer2017-2} that whenever $f$ is analytic at 0 and $u$ and $v$ are smooth functions, $\langle f(hD)u,v \rangle-\langle u,f(-hD)v \rangle$ is a total time derivative up to $\mathcal{O}(h^N)$ for arbitrary $N$. By setting $f(hD)=L_1 L_2^{-1}(e^{hD})/(hD)$, $u=DA(z)$ and $v=Dz$, we obtain  $\langle DA(\tilde{y}),\ L_1 L_2^{-1}(e^{hD}\tilde{y}) \rangle-\langle L_1 L_2^{-1}(e^{hD})A(\tilde{y}),\ D\tilde{y} \rangle$ is a total time derivative of $(\tilde{y},\dot{\tilde{y}})$ up to $\mathcal{O}(h^N)$. Moreover, inserting the expansions from \eqref{operators-expasion} into the right-hand side of \eqref{modified-eq2}, one finds that it is of order $\mathcal{O}(h^2)$ and can be represented as a formal series in even powers of $h$. Therefore, we conclude that equation \eqref{modified-eq2} is a total derivative and that there exists a function 
\begin{equation*}
H_{h}(\tilde{y},\dot{\tilde{y}})=H(\tilde{y},\dot{\tilde{y}})+h^2H_{2}(\tilde{y},\dot{\tilde{y}})+h^4H_{4}(\tilde{y},\dot{\tilde{y}})+\cdots,
\end{equation*}
such that its truncation at the $\mathcal{O}(h^N)$ term satisfies
$\frac{\mathrm{d}}{\mathrm{d}t}H_{h}(\tilde{y},\dot{\tilde{y}})=\mathcal{O}(h^N)$ along the  solutions of equation \eqref{modified}. 

$\bullet$  \textbf{From short to long time intervals.}
We are now prepared to prove that the energy 
$H$ is nearly preserved by FVI over long time intervals. This is achieved by decomposing the whole interval into subintervals of length $h$ and patching together the short-time near-conservation results derived above. Assume that the numerical solution stays in a compact set independent of $h$.
From the short-time analysis, we have the estimates
\begin{equation*}
	\abs{H(\tilde{y}(t),\dot{\tilde{y}}(t))-H_{h}(\tilde{y}(t),\dot{\tilde{y}}(t))}\leq Ch^2,
	\qquad
	\abs{H_{h}(\tilde{y}(t),\dot{\tilde{y}}(t))-H_{h}(\tilde{y}(0),\dot{\tilde{y}}(0))}\leq cth^N,
\end{equation*}
where the constants $C$ and $c$ are independent of the final time $T$ and $h$. Using these estimates and applying a telescoping sum, we obtain
\begin{equation*}
	\begin{aligned}
		&\abs{H(\tilde{y}(nh),\dot{\tilde{y}}(nh))-H(\tilde{y}(0),\dot{\tilde{y}}(0))}=\Bigl| H(\tilde{y}(nh),\dot{\tilde{y}}(nh))-H_{h}(\tilde{y}(nh),\dot{\tilde{y}}(nh))+\sum_{j=1}^{n}\big(H_{h}(\tilde{y}(jh),\dot{\tilde{y}}(jh))\\
		&\ \ -H_{h}(\tilde{y}((j-1)h),\dot{\tilde{y}}((j-1)h))\big)+H_{h}(\tilde{y}(0),\dot{\tilde{y}}(0))-H(\tilde{y}(0),\dot{\tilde{y}}(0))\Bigr|\leq Ch^2+cnh^{N+1}.
	\end{aligned}
\end{equation*}
Therefore, for any fixed integer $N \geq 3$ and for $T = nh$ satisfying $nh \leq c h^{-N+2}$, one has
\begin{equation*}
	\abs{H(\tilde{y}(nh),\dot{\tilde{y}}(nh))-H(\tilde{y}(0),\dot{\tilde{y}}(0))}
	\leq Ch^2,
\end{equation*}
where  $C$ is independent of the final time $T$.

We next prove that the energy is nearly preserved along the filtered two-step variational integrator.
Based on the idea of the backward error analysis, we know that the numerical solution $x_n$ and the exact solution $p(nh)$ of the truncated modified equation satisfy $x_n=p(nh)+\mathcal{O}(h^N)$. Using the expression $\tilde{y}(nh)=(p(nh)+p(nh+h)
)/2$, this directly leads to $x_{n+1/2}=\tilde{y}(nh)+\mathcal{O}(h^N)$. 
For the velocity approximation, one can write
$$v_{n+1/2}=\Phi(x_{n+1}-x_{n})/h=\Phi(p(nh+h)-p(nh))/h+\mathcal{O}(h^N).$$
Applying the operator $h^{-1}L_1L_2^{-1}(e^{hD})$ to this expression yields
$$
v_{n+1/2}
=\Phi\big(h^{-1}L_1L_2^{-1}(e^{hD})\tilde{y}(nh)\big)
+\mathcal{O}(h^N)
= \Phi\Big(\dot{\tilde{y}}(nh)
-\frac{1}{12}h^2\dddot{\tilde{y}}(nh)
+\frac{1}{120}h^4\tilde{y}^{(5)}(nh)
+\cdots\Big)
+\mathcal{O}(h^N).
$$
In view of the definition of $\Phi$, this reduces to
$$
v_{n+1/2}=\dot{\tilde{y}}(nh)+\mathcal{O}(h^2).
$$ Using the Lipschitz continuity of $H$, it follows that
\begin{equation*}
\abs{H(x_{n+1/2},v_{n+1/2})-H(\tilde{y}(nh),\dot{\tilde{y}}(nh))}\leq Ch^2, \qquad \abs{H(x_{1/2},v_{1/2})-H(\tilde{y}(0),\dot{\tilde{y}}(0))}\leq Ch^2.
\end{equation*}
Therefore, for any fixed integer $N \geq 3$ and for $T = nh$, we have
$$\abs{H(x_{n+1/2},v_{n+1/2})-H(x_{1/2},v_{1/2})}\leq Ch^2,$$
where $C$ is independent of the final time $T$ as long as $nh \leq c h^{-N+2}$. This establishes the statement of Theorem \ref{conservation-H1}.

\subsubsection{Long-time  near-conservation of momentum (Proof of Theorem \ref{conservation-M1})} \label{sec4.3} 
Similar to the analysis in Section~\ref{sec4.2}, we only show that the modified differential equation \eqref{modified} admits a formal first integral that remains close to the momentum $M$,
and the long-time near-conservation of momentum then follows in the same way as for the energy.

Using the invariance conditions \eqref{U-A}, one readily verifies that $A'(\tilde{y})S\tilde{y}=SA(\tilde{y})$ and $\tilde{y}^{\intercal}S\nabla_{\tilde{y}} U(\tilde{y})=0$.
Multiplying \eqref{modified-eq}  with $(S\tilde{y})^{\intercal}$ yields
\begin{equation}\label{modified-eq3}
	\begin{aligned}
		-h^{-2}(S\tilde{y})^{\intercal}\Psi^{-1} (L_1^2L_2^{-2}(e^{hD})\tilde{y})
		&=h^{-1}(S\tilde{y})^{\intercal}A'(\tilde{y})^{ \intercal}(L_1 L_2^{-1}(e^{hD})\tilde{y})
		-h^{-1}(S\tilde{y})^{\intercal}L_1 L_2^{-1}(e^{hD})A(\tilde{y})
		+(S\tilde{y})^{\intercal}F(\tilde{y}).
	\end{aligned}
\end{equation}	
We will show that both sides of \eqref{modified-eq3} can be written as total derivatives, and there exist functions $M_h(\tilde{y},\dot{\tilde{y}})$ and $M_{2j}(\tilde{y},\dot{\tilde{y}})$ such that $$M_{h}(\tilde{y},\dot{\tilde{y}})=M(\tilde{y},\dot{\tilde{y}})+\sum\limits_{j=1}^{\infty}h^{2j}M_{2j}(\tilde{y},\dot{\tilde{y}}).$$This result is established in the following three steps.

\textsl{(i)} The last term of the formula \eqref{modified-eq3} cancels by using $F(\tilde{y})=-\nabla_{\tilde{y}}U(\tilde{y})$ and $\tilde{y}^{\intercal}S\nabla_{\tilde{y}} U(\tilde{y})=0$. Based on the antisymmetric property of $S$ and the fact that $A'(\tilde{y})S\tilde{y}=SA(\tilde{y})$, the right-hand side of the formula \eqref{modified-eq3} becomes
\begin{equation*}
	\begin{aligned}
		&h^{-1}\big((S\tilde{y})^{\intercal}A'(\tilde{y})^{ \intercal}(L_1 L_2^{-1}(e^{hD})\tilde{y})-(S\tilde{y})^{\intercal}L_1 L_2^{-1}(e^{hD})\tilde{y}A(\tilde{y})\big)
		 =h^{-1}\big((A'(\tilde{y})S\tilde{y})^{\intercal}(L_1 L_2^{-1}(e^{hD})\tilde{y})\\
		 &\ \ +\tilde{y}^{\intercal}L_1 L_2^{-1}(e^{hD})SA(\tilde{y})\big)
	 =h^{-1}\big((SA(\tilde{y}))^{\intercal}(L_1 L_2^{-1}(e^{hD})\tilde{y})+\tilde{y}^{\intercal}L_1 L_2^{-1}(e^{hD})SA(\tilde{y})\big).
	\end{aligned}
\end{equation*}
 By the definition of Euclidean inner product, we obtain
\begin{equation*}
	\begin{aligned}
	\tilde{y}^{\intercal}L_1 L_2^{-1}(e^{hD})SA(\tilde{y})=(L_1 L_2^{-1}(e^{hD})SA(\tilde{y}))^{\intercal}\tilde{y}
		=  \langle L_1 L_2^{-1}(e^{hD})SA(\tilde{y}) ,\ \tilde{y}  \rangle=-\langle L_1 L_2^{-1}(e^{-hD})SA(\tilde{y}) ,\ \tilde{y}  \rangle.
	\end{aligned}
\end{equation*}
Since $L_1 L_2^{-1}(e^{hD})$ is an odd operator, the right-hand side of \eqref{modified-eq3} can be rewritten as
\begin{equation*}
	 h^{-1}\big(\langle SA(\tilde{y}) ,\ L_1 L_2^{-1}(e^{hD})\tilde{y} \rangle - \langle L_1 L_2^{-1}(e^{-hD})SA(\tilde{y}) ,\ \tilde{y}  \rangle\big),
\end{equation*}
which is a total derivative up to  $\mathcal{O}(h^{N-1})$ for arbitrary $N$.

\textsl{(ii)}  Combining the expression of $\Psi^{-1}$ with the  expansion \eqref{operators-expasion}, the left side of the formula \eqref{modified-eq3} becomes
\begin{equation*}\label{M-total}
	-{\tilde{y}}^{\intercal}S\Big(\ddot{\tilde{y}}-\frac{1}{6}h^2\tilde{y}^{(4)}+\frac{17}{720}h^4\tilde{y}^{(6)}+\cdots \Big)
		-{\tilde{y}}^{\intercal}S\Big(1-\mathrm{tanc}\Big(\frac{h}{2 }\Big)^{-1}\Big){\tilde{B}_0}^2\Big(\ddot{\tilde{y}}-\frac{1}{6}h^2\tilde{y}^{(4)}+\frac{17}{720}h^4\tilde{y}^{(6)}+\cdots \Big).
\end{equation*}	
Note that the term $\tilde{y}^{\intercal}S\tilde{y}^{(2l)}$ and $\tilde{y}^{\intercal}S{\tilde{B}_0}^2\tilde{y}^{(2l)}$ are total derivatives as
$$\tilde{y}^{\intercal}S\tilde{y}^{(2l)}=\frac{\mathrm{d}}{\mathrm{d}t}\big(\tilde{y}^{\intercal}S\tilde{y}^{(2l-1)}-\dot{\tilde{y}}^{\intercal}S\tilde{y}^{(2l-2)}
+\cdots\mp\tilde{y}^{(l-1)}S\tilde{y}^{(l)}\big),$$
and $$\tilde{y}^{\intercal}S{\tilde{B}_0}^2\tilde{y}^{(2l)}=\frac{\mathrm{d}}{\mathrm{d}t}\big(\tilde{y}^{\intercal}S{\tilde{B}_0}^2\tilde{y}^{(2l-1)}-\dot{\tilde{y}}^{\intercal}S{\tilde{B}_0}^2\tilde{y}^{(2l-2)}
+\cdots\mp\tilde{y}^{(l-1)}S{\tilde{B}_0}^2\tilde{y}^{(l)}\big),$$
which is satisfied under the assumption $S{\tilde{B}_0}^2={\tilde{B}_0}^{2}S$. We then get that the left side of the formula \eqref{modified-eq3} is a total derivative.

\textsl{(iii)} From the results of (i) and (ii), it follows that \eqref{modified-eq3} can be expressed as a total derivative. We now proceed to identify the terms that constitute the time derivative of $M(\tilde{y},\dot{\tilde{y}})$.
 Based on the definition of $\Psi$ and the expansions of $L_1^2L_2^{-2}(e^{hD})$ and $L_1L_2^{-1}(e^{hD})$, the formula \eqref{modified-eq3} can be rewritten as 
\begin{equation}\label{M-total}
	\begin{aligned}
		&-{\tilde{y}}^{\intercal}S\Big(\ddot{\tilde{y}}-\frac{1}{6}h^2\tilde{y}^{(4)}+\frac{17}{720}h^4\tilde{y}^{(6)}+\cdots \Big)
		-{\tilde{y}}^{\intercal}S\Big(1-\mathrm{tanc}\Big(\frac{h}{2 }\Big)^{-1}\Big){\tilde{B}_0}^2\Big(\ddot{\tilde{y}}-\frac{1}{6}h^2\tilde{y}^{(4)}+\frac{17}{720}h^4\tilde{y}^{(6)}+\cdots \Big)\\
		&=-{\tilde{y}}^{\intercal}S\big(A'(\tilde{y})^{\intercal}\dot{\tilde{y}}-A'(\tilde{y})\dot{\tilde{y}}\big)-\frac{h^{2}}{12}{\tilde{y}}^{\intercal}S\big(A'(\tilde{y})^{\intercal}\tilde{y}^{(3)}+(A(\tilde{y}))^{(3)}\big)-\frac{h^{4}}{120}{\tilde{y}}^{\intercal}S\big(A'(\tilde{y})^{\intercal}\tilde{y}^{(5)}-(A(\tilde{y}))^{(5)}\big)+\cdots.
	\end{aligned}
\end{equation}	
With the help of $\dot{\tilde{y}}\times B(\tilde{y})=\big(A'(\tilde{y})^{\intercal}-A'(\tilde{y})\big)\dot{\tilde{y}}$, $\tilde{y}^{\intercal}S\big(\dot{\tilde{y}}\times B(\tilde{y})\big)=-\frac{\mathrm{d}}{\mathrm{d}t}\big(\tilde{y}^{\intercal}SA(\tilde{y})\big)$ and
 ${\tilde{y}}^{\intercal}S\ddot{\tilde{y}}=\frac{\mathrm{d}}{\mathrm{d}t}(\tilde{y}^{\intercal}S\dot{\tilde{y}})$, the formula \eqref{M-total} turns to
 \begin{equation*}
 	\begin{aligned}
 		\frac{\mathrm{d}}{\mathrm{d}t}\big(\overbrace{A(\tilde{y})^{\intercal}S\tilde{y}+\dot{\tilde{y}}^{\intercal}S\tilde{y}}^{M(\tilde{y},\dot{\tilde{y}})}\big)
 	&={\tilde{y}}^{\intercal}S\Big(1-\mathrm{tanc}\Big(\frac{h}{2 }\Big)^{-1}\Big){\tilde{B}_0}^2\Big(\ddot{\tilde{y}}-\frac{1}{6}h^2\tilde{y}^{(4)}+\frac{17}{720}h^4\tilde{y}^{(6)}+\cdots \Big)\\
 		&\quad	+h^{2}\Big(-\frac{1}{6}{\tilde{y}}^{\intercal}S\tilde{y}^{(4)}+\frac{1}{12}{\tilde{y}}^{\intercal}SA'(\tilde{y})^{\intercal}\tilde{y}^{(3)}+\frac{1}{12}{\tilde{y}}^{\intercal}S(A(\tilde{y}))^{(3)}\Big)\\
 		&\quad
 		+h^{4}\Big(\frac{17}{720}{\tilde{y}}^{\intercal}S\tilde{y}^{(6)}+\frac{1}{120}{\tilde{y}}^{\intercal}SA'(\tilde{y})^{\intercal}\tilde{y}^{(5)}-\frac{1}{120}{\tilde{y}}^{\intercal}S(A(\tilde{y}))^{(5)}\Big)+\cdots.
 	\end{aligned}
 \end{equation*}	
Therefore, based on the Taylor expansion of $\mathrm{tanc}(h/2)^{-1}$, we obtain $h$-independent functions $M_{2j}(\tilde{y},\dot{\tilde{y}})$
such that the function
	\begin{equation*}
	M_{h}(\tilde{y},\dot{\tilde{y}})=M(\tilde{y},\dot{\tilde{y}})+h^2M_{2}(\tilde{y},\dot{\tilde{y}})+h^4M_{4}(\tilde{y},\dot{\tilde{y}})+\cdots,
\end{equation*}
truncated at the $\mathcal{O}(h^{N})$ term, satisfies $\frac{\mathrm{d}}{\mathrm{d}t}M_{h}(\tilde{y},\dot{\tilde{y}})=\mathcal{O}(h^N)$ along solutions of the modified differential equation \eqref{modified}.  A similar argument as in Section~\ref{sec4.2} then  leads to the result of Theorem \ref{conservation-M1}.

\section{Error bounds and long-term analysis in a strong magnetic field (Proof of Theorems \ref{error bound 2}--\ref{conservation-M2})  } \label{sec5}
This section is devoted to analyzing  the error bounds and long-time near-conservation of energy and magnetic moment under a strong magnetic field given in Section \ref{sec3.2}.  The primary analytical tool is modulated Fourier expansion which was developed in  \cite{hairer2000,hairer2006}, it is a powerful  tool for studying highly oscillatory differential equations. 
The basic idea of modulated Fourier expansion is to decompose both the exact and the numerical solutions into a slowly evolving component and highly oscillatory terms, where the oscillations are characterized by trigonometric functions with slowly varying modulation functions. To illustrate this idea, we represent the solution of \eqref{charged-particle}
in the form of a modulated Fourier expansion
\begin{equation}\label{exact-MFE-fram}
	x(t)
	=
	y^{0}(t)
	+
	\sum_{0<\abs{k}<N} \mathrm{e}^{\mathrm{i} k t / \varepsilon} \, y^{k}(t)
	+
	R_{N}(t),
\end{equation}
where $N \ge 2$ is the truncation index,
$y^{k}(t)$ are slowly varying modulation functions,
and $R_{N}(t)$ denotes the defect that results from substituting the truncated expansion into equation~\eqref{charged-particle}.
Here $y^{0}(t)$ is a real-valued function, while $y^{k}(t)$   are complex-valued modulation functions.
Since the solution $x(t)$ is real-valued, it follows that 
$
y^{-k}(t)=\overline{y^{k}(t)}.$ In addition, the modulation functions $y^{k}(t)$ for $0 \leq \abs{k} <N$  are required to be smooth in the sense that all their time derivatives are bounded independently of $\varepsilon$.

To determine the modulation functions, the truncated expansion
\eqref{exact-MFE-fram} without the remainder is inserted into equation~\eqref{charged-particle},
the nonlinearity is expanded around $y^{0}(t)$, and the coefficients of $\mathrm{e}^{\mathrm{i} k t / \varepsilon}$ are compared. This procedure leads to a system of asymptotic differential and algebraic equations for the coefficient functions $y^{k}(t)$, $0 \leq |k| < N$,  together with the corresponding initial conditions. It then follows that the modulation functions are uniquely determined up to the truncation order, and bounds for both the modulation functions and the remainder term can be obtained. This construction of the modulation system provides an analytical framework for the error analysis and for investigating the long-time behavior of the numerical method.  Within this framework, we rigorously establish error bounds by comparing the modulated Fourier expansions of the exact and numerical solutions. Moreover, the modulation system is shown to admit two almost-invariants, one associated with the total energy and the other with the magnetic moment. In the following subsections, we first derive the modulated Fourier expansions of the exact and numerical solutions and then establish the corresponding error bounds and long-time near-conservation results.

\subsection{Modulated Fourier expansions and proof of error bounds} \label{sec5-MFE-exac}
To establish the error bounds stated in Section~\ref{sec3.2}, we employ the modulated Fourier expansion technique.
For clarity, the analysis is organized into the following three parts:
\begin{itemize}
    \item Section \ref{MEF-exact} presents the modulated Fourier expansion of the exact solution.
    \item Section \ref{MFE-method} derives the modulated Fourier expansion corresponding to the numerical solution produced by the FVI method.
    \item Section \ref{error-proof} establishes error bounds for FVI by comparing the modulated Fourier expansions of the exact and numerical solutions.
\end{itemize}

\subsubsection{Modulated Fourier expansion of the exact solution}\label{MEF-exact}
To analyze the highly oscillatory structure induced by the strong magnetic field more accurately, we first diagonalize the linear operator $v \mapsto v \times B_0$. Within the modulated Fourier expansion framework developed for charged-particle dynamics (see, e.g., \cite{Hairer2018}), we denote the
eigenvalues and corresponding normalized eigenvectors of this operator by $\lambda_{-1} = -i$, $\lambda_0 = 0$, and $\lambda_1 = i$, with $v_{-1} = \overline{v_1}$, $v_0 = B_0$, and $v_1$. Let $P_j=v_{j}v_{j}^{*}$ denote the orthogonal projection onto each eigenspace, where $v_{j}^{*}=(\overline{v_{j}})^{\intercal}$. These projections satisfy the identities $P_{-1} + P_0 + P_1 = I$, $P_0(v \times B_0) = 0$, and $P_{\pm 1}(v \times B_0) = \pm i P_{\pm 1} v$. Furthermore, the coefficient functions $y^{k}$ (with the time variable $t$ omitted) in the expansion~\eqref{exact-MFE-fram} can be expressed in the basis $(v_j)$ associated with the projections $P_j$  
as 
$$y^{k}=y^{k}_{-1}+y^{k}_{0}+y^{k}_{1},\quad y^{k}_{j}=P_{j}y^{k},\quad \textmd{for}\ \ j=-1,0,1.$$
The following theorem presents the modulated Fourier expansion for the exact solution of the CPD \eqref{charged-particle}.

\begin{mytheo}\label{Theoexact}(\cite{Lubich2022})
	Let $x(t)$ be an exact solution of \eqref{charged-particle}, it is assumed that initial velocity is bounded as \eqref{velocity} and $x(t)$ stays in a compact set $K$ which is independent of $\varepsilon$ for $0 \leq t \leq T $. Then consider its modulated Fourier expansion
	\begin{equation}\label{exact-MFE}
		x(t)=\sum_{\abs{k} \leq N}e^{ikt/  \varepsilon}y^{k}(t)+R_N(t)
	\end{equation}
	with an arbitrary truncation index $N \geq 1$. 
	The modulated Fourier expansion \eqref{exact-MFE} has the following properties:
	
	\textsl{(a)} The coefficient functions $y^{k}$ together with their derivatives (up to order $N$) are bounded as $y_{j}^{0} = \mathcal{O}(1)\ \text{for}\ j \in \{-1,0,1\}$,
	$y_{\pm 1}^{\pm 1}=\mathcal{O}(\varepsilon)$,
	and for the remaining $(k, j)$ with $\abs{k} \leq N$,
	$y_{j}^{k}=\mathcal{O}(\varepsilon^{\abs{k}+1}).$
	They are unique up to $\mathcal{O}(\varepsilon^{N+1})$ and are chosen to satisfy
	$y_{-j}^{-k}=\overline{y_{j}^{k}}$. Moreover, $\dot{y}_{\pm 1}^{0}$ together with its derivatives is bounded as $\dot{y}_{\pm 1}^{0}=\mathcal{O}(\varepsilon)$.
	
	\textsl{(b)} The remainder term and its derivatives are bounded by
	\begin{equation*}
		R_N(t)=\mathcal{O}(t^2\varepsilon^{N}),\ \ \ \dot{R}_N(t)=\mathcal{O}(t\varepsilon^{N})\quad \textmd{for} \ \ 0 \leq t \leq T.
	\end{equation*}

	\textsl{(c)} The modulation functions $y_0^0$, $y_{\pm 1}^0$, $y_{1}^{1}$, $y_{-1}^{-1}$ satisfy the following  expressions
	\begin{equation*}
		\begin{aligned}
			&\ddot{y}_{0}^{0}=P_{0}\big(\dot{y}^0 \times B_1(y^0)+F(y^0)\big)+2P_{0}Re\Big(\frac{i}{\varepsilon}y^1 \times B'_{1}(y^0)y^{-1}\Big)+\mathcal{O}(\varepsilon^2),\\
			&\dot{y}_{\pm 1}^0=\pm i  \varepsilon  P_{\pm 1}\big(\dot{y}^0 \times B_1(y^0)+F(y^0)\big)+\mathcal{O}(\varepsilon^2),\quad \
            \dot{y}_{\pm 1}^{\pm 1}=P_{\pm 1}\big(y_{\pm 1}^{\pm 1} \times B_1(y^0)\big)+\mathcal{O}( \varepsilon^2).
		\end{aligned}
	\end{equation*}	
	All other  functions  $y_{j}^{k}$ are obtained from algebraic expressions depending on
	$y^0$, $\dot{y}_{0}^0$, $y_{1}^{1}$, $y_{-1}^{-1}$.
	
	\textsl{(d)} Initial values for the differential equations of item \textsl{(c)} are derived by
	\begin{equation*}
		\begin{aligned}
			&y^{0}(h/2)=x(h/2)+\varepsilon \dot{x}(h/2) \times B_{0}+\mathcal{O}(\varepsilon^2),\quad \ y_{\pm 1}^{\pm 1}(h/2)=\mp i\varepsilon e^{\mp ih/2\varepsilon}P_{\pm 1}\dot{x}(h/2)+\mathcal{O}(\varepsilon ^2),\\
			&\dot{y}_{0}^0(h/2)=P_0\dot{x}(h/2)
			-\varepsilon P_0 \big((\dot{x}(h/2) \times B_0)\times B_1(x(h/2))\big)
			+\mathcal{O}(\varepsilon^2).
		\end{aligned}
	\end{equation*}		
	The constants symbolized by the $\mathcal{O}$-notation  depend on $K$, $T$, $N$, on the velocity bound $\hat{C}$ and  on the bounds of derivatives of $B_1$ and $F$, but they are 	independent of $\varepsilon$	and $t$ with $ 0 \leq t \leq T$.
	\end{mytheo}

\begin{proof} The result is stated in \cite{Lubich2022}, 
but a detailed proof is not given there. Since its structure is instructive for the subsequent analysis, we outline the main arguments here.

\textsl{(a)} and \textsl{(b)}
The bounds in \textsl{(a)} and \textsl{(b)} follow from an argument similar to that
of Theorem~4.1 in \cite{Hairer2018}, which establishes the existence of the modulated Fourier expansion together with bounds for the modulation functions and the remainder term.
While \cite{Hairer2018} treats a strong magnetic field with state-dependent frequencies and projections, the present work considers a mildly strong magnetic field and uses a constant frequency and fixed projections $P_j$. The details are therefore omitted and we refer to \cite{Hairer2018}. For the error bounds in Theorem~\ref{error bound 2}, explicit differential
equations for the dominant modulation functions $y_0^0$, $y_{\pm 1}^0$ and $y_{\pm 1}^{\pm 1}$ are required, and hence only the differential equations in \textsl{(c)} and the corresponding initial values in \textsl{(d)} are stated.

\textsl{(c)} Inserting the expansion \eqref{exact-MFE} with removed reminder term into the differential equation \eqref{charged-particle} and comparing the coefficients of $e^{ikt/  \varepsilon}$ yields
\begin{equation}\label{CPD-MFE}
\ddot{y}^{k}+2\frac{ik}{\varepsilon}\dot{y}^{k}-\frac{k^2}{\varepsilon^2}{y}^{k}=\sum_{k_1+k_2=k}\Big(\dot{y}^{k_1}+\frac{ik_1}{\varepsilon}{y}^{k_1}\Big)\times\sum_{ \underset{s(\alpha)=k_2 }{m \geq 0  }}\frac{1}{m!}B^{(m)}(y^0){\pmb{y}}^{\alpha}+\sum_{ \underset{s(\alpha)=k }{m \geq 0  }}F^{(m)}(y^0){\pmb{y}}^{\alpha},
\end{equation}
where
$\alpha=(\alpha_1,\ldots, \alpha_m )$ is a multi-index with
$\alpha_j \in \mathbb{Z} \backslash \{0\} $, $s(\alpha)=\alpha_1+\cdots +\alpha_m $,
$\abs{\alpha}=\abs{\alpha_1}+\cdots +\abs{\alpha_m} $ and $\pmb{y}^{\alpha}=(y^{\alpha_1},\ldots, y^{\alpha_m})$.

In view of \eqref{CPD-MFE} for $k=0$ and the bound of $y^k$, we have
\begin{equation*}
\ddot{y}^{0}=\dot{y}^{0}\times B(y^0)+2Re\big(i\varepsilon^{-1}y^1 \times B'(y^0)y^{-1}\big)+F(y^0)+\mathcal{O}(\varepsilon ^2).
\end{equation*}
Then taking the projection $P_0$ on both sides and noting that $P_0(\dot{y}^0\times B_0)=0$, we arrive at the stated first differential equation for $y_0$. Moreover, from $P_{\pm 1}(\dot{y}^0\times B_0)=\pm i \dot{y}_{\pm 1}^{0}$, it follows that
\begin{equation*}
	\begin{aligned}
	P_{\pm 1}\ddot{y}^{0}
	&=P_{\pm 1}\big(\dot{y}^{0}\times \varepsilon^{-1}B_0\big)+P_{\pm 1}\big(\dot{y}^{0}\times B_{1}(y^0)+F(y^0)\big)+\mathcal{O}(\varepsilon)\\
	&=\pm i\varepsilon^{-1}\dot{y}_{\pm 1}^{0}+P_{\pm 1}\big(\dot{y}^{0}\times B_{1}(y^0)+F(y^0)\big)+\mathcal{O}(\varepsilon).
	\end{aligned}
\end{equation*}
We then have the first-order differential equation for $y_{\pm 1}^{0}$
under the bounds $\dot{y}_{\pm 1}^{0}=\mathcal{O}(\varepsilon)$.

Inserting $k=\pm 1$ into \eqref{CPD-MFE} gives
\begin{equation}\label{CPD-MFE2}
\pm 2i\varepsilon^{-1}\dot{y}^{\pm 1}-\varepsilon^{-2}y^{\pm 1} =\big(\dot{y}^{\pm 1}\pm i\varepsilon^{-1}y^{\pm 1} \big)\times B(y^0)
+\mathcal{O}(\varepsilon).
\end{equation}
Multiplying  \eqref{CPD-MFE2} with $P_{\pm1}$, we find
\begin{equation*}
	\pm 2i\varepsilon^{-1}\dot{y}_{\pm 1}^{\pm 1}-\varepsilon^{-2}y_{\pm 1}^{\pm 1} =\pm i\varepsilon^{-1}P_{\pm1}\big(\dot{y}^{\pm 1}\pm i\varepsilon^{-1}y^{\pm 1} \big)+ P_{\pm1}\big(\dot{y}^{\pm 1}\pm i\varepsilon^{-1}y^{\pm 1} \big)\times B_{1}(y^0)
	+\mathcal{O}(\varepsilon).
\end{equation*}
Since the $\varepsilon^{-2}$-terms cancel on both sides,  the $\varepsilon^{-1}$-terms then become dominant, which leads to
\begin{equation*}
\dot{y}_{\pm 1}^{\pm 1}=\mp i \varepsilon \big( P_{\pm1}\big(\dot{y}^{\pm 1}\pm i\varepsilon^{-1}y^{\pm 1} \big)\times B_{1}(y^0)
	+\mathcal{O}(\varepsilon)\big)= P_{\pm1}(y^{\pm 1}\times B_{1}(y^0))	+\mathcal{O}(\varepsilon^2),
\end{equation*}
which is the presented third equation in (c).

\textsl{(d)} In contrast to \cite{Lubich2022}, where the initial values are given at $t=0$, we state them at $t=h/2$. From \eqref{exact-MFE}, we obtain
\begin{equation}\label{CPD-MFE3}
	\begin{aligned}
	&x(h/2)=y^0(h/2)+e^{ih/2\varepsilon}y^1(h/2)+e^{-ih/2\varepsilon}y^{-1}(h/2)+\mathcal{O}(\varepsilon^3),\\
	&\dot{x}(h/2)=\dot{y}^0(h/2)+e^{ih/2\varepsilon}\dot{y}^1(h/2)+e^{-ih/2\varepsilon}\dot{y}^{-1}(h/2)+i\varepsilon^{-1}\big(e^{ih/2\varepsilon}y^1(h/2)-e^{-ih/2\varepsilon}y^{-1}(h/2)\big)+\mathcal{O}(\varepsilon^2).
	\end{aligned}
\end{equation}
We express the vectors in the basis $(v_j)$, use $P_0(\dot{y}^0\times B_0)=0$, $P_{\pm 1}(\dot{y}^0\times B_0)=\pm i \dot{y}_{\pm 1}^{0}$, and the bounds of $y^k$ given in part (a), then the above second formula yields
\begin{equation*}
		e^{ih/2\varepsilon}y_{1}^1(h/2)+e^{-ih/2\varepsilon}y_{-1}^{-1}(h/2)=-\varepsilon\dot{x}(h/2)\times B_0+\mathcal{O}(\varepsilon^2).
\end{equation*}
 Thus, we have $y^{0}(h/2)=x(h/2)+\varepsilon \dot{x}(h/2) \times B_{0}+\mathcal{O}(\varepsilon^2)$. Multiplying the second formula of \eqref{CPD-MFE3} with $P_0$ and considering the differential equation for $y_{\pm1}^{\pm1}$ in part (c), the expression for $\dot{y}_0^0(h/2)$ is then obtained. Taking the projection $P_{1}$ on both sides of the second equation of \eqref{CPD-MFE3} gives
\begin{equation*}
	\begin{aligned}
		P_{1}\dot{x}(h/2)&=	\dot{y}_{1}^0(h/2)+e^{ih/2\varepsilon}\dot{y}_{1}^1(h/2)+e^{-ih/2\varepsilon}\dot{y}_{1}^{-1}(h/2)+i\varepsilon^{-1}\big(e^{ih/2\varepsilon}y_{1}^1(h/2)-e^{-ih/2\varepsilon}y_{1}^{-1}(h/2)\big)+\mathcal{O}(\varepsilon^2)\\
&=i\varepsilon^{-1}e^{ih/2\varepsilon}y_{1}^1(h/2)+\mathcal{O}(\varepsilon),
\end{aligned}
\end{equation*}
which we have used the bounds $\dot{y}_{1}^{0}=\mathcal{O}(\varepsilon)$ and $y_{j}^{k}=\mathcal{O}(\varepsilon^{\abs{k}+1})$ for $j=1$ and  $k=\pm1$.  This leads to the expression for $y_{1}^{1}(h/2)$ in part (d). Similarly, the case $j=-1$ is derived after multiplying the second equation of \eqref{CPD-MFE3} by $P_{-1}$.
\end{proof}

\subsubsection{Modulated Fourier expansion of the numerical solution} \label{MFE-method}
In this subsection, we shall consider the modulated Fourier expansion of the two-step formulation \eqref{Method}. 
Our analysis will be carried out for the regime $h^2 > C^{*}\varepsilon$, while the corresponding results for $ c^{*}\varepsilon^2 \leq h^2 \leq C^{*}\varepsilon$ can be derived in the same way.

\begin{mytheo}\label{thm: bounds} Under the  conditions of  Theorem \ref{error bound 2} with $h^2 > C^{*}\varepsilon$, suppose that  initial velocity $\dot{x}(0)$  is bounded as in \eqref{velocity} and the numerical solution $x_n$ remains in a compact set $K$ for $0 \leq nh \leq T$, where both $K$ and $T$ are independent of $\varepsilon$ and $h$. Then the numerical result $x_{n+1/2} :=(x_n+x_{n+1})/2$ admits the following modulated Fourier expansion
	\begin{equation*}\label{MFE}
		x_{n+1/2}=\sum_{\abs{k} \leq N}e^{ikt/  \varepsilon}\xi^{k}(t)+R_N(t),	\quad  \ t=(n+1/2)h,
	\end{equation*}
	with the following properties:
	
\textsl{(a)} The coefficient functions $\xi^k$ and their derivatives up to order $N$ are bounded as follows:
\begin{equation*}
\begin{aligned}
&\xi_{j}^{0} = \mathcal{O}(1)\ \text{for}\ j \in \{-1,0,1\}, \ \xi_0^k = \mathcal{O}(h\varepsilon^{|k|}),\  \text{for}\abs{k}\geq 1,
\\  &\xi_{\pm 1}^{\pm 1} = \mathcal{O}(\varepsilon), \ \xi_j^k = \mathcal{O}(\varepsilon^{|k|+1}) \ \text{for all other}\ (k,j) \ \text{with} \abs{k}\leq N.
\end{aligned}
\end{equation*}
These coefficients are unique up to $\mathcal{O}(\varepsilon^{N+1})$ and are chosen to satisfy $\xi_{-j}^{-k} = \overline{\xi_j^k}$. Moreover,  the time derivatives of $\xi_{\pm 1}^0$, as well as their higher-order derivatives, are bounded by $\dot{\xi}_{\pm 1}^0 = \mathcal{O}(\varepsilon)$.

\textsl{(b)} For arbitrary $M > 1$, the remainder $R_N(t)$   is bounded by
\begin{equation*}
P_0R_N(t)=\mathcal{O}(t^2 h^{M})+\mathcal{O}(t^2 \varepsilon^{N}),\ \ P_{\pm1}R_N(t)=\mathcal{O}(t^2\varepsilon h^{M-1})+\mathcal{O}(t^2 \varepsilon^{N}), \ \  0\leq t=(n+1/2)h \leq T.
\end{equation*}	

\textsl{(c)} The functions $\xi_0^0$, $\xi_{\pm 1}^0$, $\xi_{1}^{1}$, $\xi_{-1}^{-1}$ satisfy the following differential equations
	\begin{equation*}
		\begin{aligned}
			&\ddot{\xi}_{0}^{0}=P_{0}\big(\dot{\xi}^0 \times B_1(\xi^0)+F(\xi^0)\big)+\mathcal{O}(h^2),\quad \ 
            \dot{\xi}_{\pm 1}^0=\pm i  \varepsilon  P_{\pm 1}\big(\dot{\xi}^0 \times B_1(\xi^0)+F(\xi^0)\big)+\mathcal{O}(\varepsilon h),\\
			&\dot{\xi}_{\pm 1}^{\pm 1}={\varepsilon}{h}^{-1}\sin(h/ \varepsilon) P_{\pm 1}\big(\xi_{\pm 1}^{\pm 1} \times B_1(\xi^0)\big)+\mathcal{O}( \varepsilon^2).
		\end{aligned}
	\end{equation*}	
	All other modulation functions  $\xi_{j}^{k}$ are given by  algebraic expressions depending on
	$\xi^0$, $\dot{\xi}_{0}^0$, $\xi_{1}^{1}$, $\xi_{-1}^{-1}$.
	
\textsl{(d)} For the differential equations of item \textsl{(c)}, their initial values  are given by
	\begin{equation*}
\xi^{0}(h/2)=x_{1/2}+\mathcal{O}(h^2),\
	\dot{\xi}_{0}^0(h/2)=P_0v_{1/2}+\mathcal{O}( h^2),\
    \xi_{\pm 1}^{\pm 1}(h/2)=\mp i\varepsilon e^{\mp ih/2\varepsilon}\cos(h/2\varepsilon)P_{\pm 1}v_{1/2}+\mathcal{O}(\varepsilon h).
	\end{equation*}	
The constants symbolized by the $\mathcal{O}$-notation  depend on $K$, $T$, $N$, on the velocity bound $\hat{C}$ and  on the bounds of derivatives of $B_1$ and $F$, but they are 	independent of $\varepsilon$	and $n$ with $ 0 \leq nh \leq T$.
\end{mytheo}

\begin{proof} \textsl{(a)} and \textsl{(b)}
The bounds for the modulation functions \textsl{(a)} and the remainder term \textsl{(b)} can be obtained through a similar analysis of the exact solution, provided that the step size satisfies $h^2> C^{*}\varepsilon$. Here, we emphasize the main differences in extending the modulated Fourier expansion to the two-step formulation \eqref{Method}.
	
 In order to establish the modulated Fourier expansion of the two-step formulation \eqref{Method}, we first consider the
 operators \eqref{oper} defined in Section \ref{sec4.2} and use the fact that $v \times B(x)=\big(A'(x)^{\intercal}-A'(x)\big)v$. This yields
	\begin{equation}\label{opera-method}
		\begin{aligned}
			h^{-2}L_1^2 L_2^{-2}(e^{hD})x_{n+\frac{1}{2}}
			=\Psi & \big(h^{-1}L_1 L_2^{-1}(e^{hD})x_{n+\frac{1}{2}} \times B\big(x_{n+\frac{1}{2}}\big)
			+h^{-1}A'\big(x_{n+\frac{1}{2}}\big)\big(L_1 L_2^{-1}(e^{hD})x_{n+\frac{1}{2}}\big)\\
			&\ \ - h^{-1} L_1 L_2^{-1}(e^{hD}) A\big(x_{n+\frac{1}{2}}\big)+F\big(x_{n+\frac{1}{2}}\big)           \big).
		\end{aligned}
	\end{equation}	
	For the solution $x_{n+1/2}$,  the  modulated Fourier expansion is given by
	\begin{equation}\label{N-MFE}
		x_{n+1/2} \approx \sum_{\abs{k} \leq N}e^{ikt/  \varepsilon}\xi^{k}(t)
		=\sum_{\abs{k} \leq N}z^{k}(t),
	\end{equation}
	where $z^{k}(t)=e^{ikt/  \varepsilon}\xi^{k}(t)$,  $t=t_{n+1/2}:=(t_{n}+t_{n+1})/2$ and the functions $\xi^{k}$ depend on the step size $h$ and $\eta=h/\varepsilon$.

Using the Taylor expansion \eqref{operators-expasion} of the operators, it follows that
	\begin{equation*}
		\begin{aligned}
			h^{-2}L_{1}^{2}L_{2}^{-2}(e^{hD})z^k(t)=e^{ikt/\varepsilon}\sum_{l \geq 0}\varepsilon^{l-2} d_{l}^{k}\frac{\mathrm{d}^l}{{\mathrm{d}t}^l}\xi^{k},\ \quad
			h^{-1}	L_{1}L_{2}^{-1}(e^{hD})z^k(t)=e^{ikt/\varepsilon}\sum_{l \geq 0}\varepsilon^{l-1} c_{l}^{k}\frac{\mathrm{d}^l}{{\mathrm{d}t}^l}\xi^{k},
		\end{aligned}
	\end{equation*}
	where $c_{2j}^{0}=d_0^0=d_{2j+1}^{0}=0$, $c_{2j+1}^{0}=\alpha_{2j+1}\eta^{2j}$  and $d_{2j}^{0}=\beta_{2j}\eta^{2j-2}$. And $\alpha_j$ and $\beta_j$ are defined through the expansions
	\begin{equation*}
		2\tanh(t/2)=\sum_{j \geq 0}^{\infty}\alpha_{j}t^j, \quad 4\tanh^2(t/2)=\sum_{j \geq 0}^{\infty}\beta_{j}t^j.
	\end{equation*}
For $k\neq0$, we have
	\begin{equation*}
		\begin{aligned}
			&c_0^k=\frac{2i}{\eta}\tan\Big(\frac{1}{2}k\eta\Big), \quad \quad
			c_1^k=\sec^2\Big(\frac{1}{2} k \eta \Big),\qquad\qquad\qquad  	c_2^k=\frac{-i\eta}{2}\tan\Big(\frac{1}{2}k\eta\Big)\sec^2\Big(\frac{1}{2} k \eta \Big),\\
			&d_0^k=\frac{-4}{\eta^2}\tan^2\Big(\frac{1}{2}k\eta\Big),\quad
			d_1^k=\frac{4i}{\eta}\tan\Big(\frac{1}{2}k\eta\Big)\sec^2\Big(\frac{1}{2} k \eta \Big),\quad  d_2^k=\big(2-\cos(k\eta)\big)\sec^4\Big(\frac{1}{2} k \eta \Big).
		\end{aligned}
	\end{equation*}
Inserting expansion \eqref{N-MFE}	into the formulation  \eqref{opera-method} and comparing the coefficients of $e^{ikt/  \varepsilon}$ yields
	\begin{equation}\label{fund}
		\sum_{l \geq 0}\varepsilon^{l-2}d_l^{k}\frac{\mathrm{d}^l}{{\mathrm{d}t}^l}  \xi^{k}=\Upsilon^k,
	\end{equation}	
	with
	\begin{equation*} 
		\begin{aligned}
			\Upsilon^k
			=&	\Psi\Big(\sum_{k_1+k_2=k}\Big(\sum_{l \geq   0}\varepsilon^{l-1} c_{l}^{k_1}\frac{\mathrm{d}^l}{{\mathrm{d}t}^l}\xi^{k_1}\Big)
			\times \sum_{ \underset{s(\alpha)=k_2 }{0 \leq m \leq N  }}\frac{1}{m!}B^{(m)}(\xi^{0}){\pmb{\xi}}^{\alpha}\\
			&+\sum_{k_1+k_2=k}\Big(\sum_{ \underset{s(\alpha)=k_1 }{0 \leq m \leq N  }}\frac{1}{m!}A^{(m+1)}(\xi^{0})\pmb{\xi}^{\alpha}\Big)
			\Big(\sum_{ l \geq 0}\varepsilon^{l-1} c_{l}^{k_2}\frac{\mathrm{d}^l}{{\mathrm{d}t}^l}\xi^{k_2}\Big) \\
			&-\sum_{ l \geq 0}\varepsilon^{l-1} c_{l}^{k}\frac{\mathrm{d}^l}{{\mathrm{d}t}^l}
			\Big(\sum_{ \underset{s(\alpha)=k }{0 \leq m \leq N  }}\frac{1}{m!}A^{(m)}(\xi^{0})\pmb{\xi}^{\alpha}\Big)
			+\sum_{ \underset{s(\alpha)=k }{0 \leq m \leq N  }}\frac{1}{m!}F^{(m)}(\xi^{0})\pmb{\xi}^{\alpha}\Big).
		\end{aligned}
	\end{equation*}	
Here, we do not provide the details of the bounds for the coefficient functions $\xi^k$  and the remainder term $R_{N}(t)$, since they can be derived using arguments similar to those used for the exact solution.

\textsl{(c)} In order to derive the second-order error bounds, we shall show the differential equations for the dominant coefficient functions $\xi_0^0$, $\xi_{\pm 1}^0$ and $\xi_{\pm1}^{\pm 1}$ as follows.   Multiplying \eqref{fund} with $P_{j}$ gives
	\begin{equation}\label{Pfund}
		\varepsilon^{-2}d_{0}^{k}\xi_{j}^{k}
		+\varepsilon^{-1}d_{1}^{k}\dot{\xi}_{j}^{k}
		+P_{j}\sum_{l \geq 2}\varepsilon^{l-2} d_{l}^{k}\frac{\mathrm{d}^l}{{\mathrm{d}t}^l}\xi^{k}
		=P_{j}\Upsilon^k.
	\end{equation}	
	
	\textsl{i)} 	For $k=0$ and $j=0$, we obtain
		\begin{equation*}\label{}
   \varepsilon^{-2}d_{0}^{0}\xi_{0}^{0}	+\varepsilon^{-1}d_{1}^{0}\dot{\xi}_{0}^{0}
   	+d_{2}^{0}\ddot{\xi}_{0}^{0}	+P_{0}\sum_{l \geq 3}\varepsilon^{l-2} d_{l}^{0}\frac{\mathrm{d}^l}{{\mathrm{d}t}^l}\xi^{0}=P_0\Upsilon^0.
		\end{equation*}	
By combining the definitions of coefficients $d_{l}^{k}$,  $c_{l}^{k}$ and $\Upsilon^k$ with the bounds for the coefficients $\xi^k$, it follows that
\begin{equation*}\label{}
	\begin{aligned}
		\ddot{\xi}_{0}^{0}+\mathcal{O}(h^2)
		=P_0\Psi\bigg(&\dot{\xi}^{0} \times B(\xi^{0})+F(\xi^{0})
		+\frac{4i}{h}\tan\Big(\frac{h}{2\varepsilon}\Big)\xi^{1}\times B_{1}'(\xi^{0})\xi^{-1}\\
		&-\frac{4i}{h}\tan\Big(\frac{h}{2\varepsilon}\Big)\big(A''({\xi}^{0})\xi^{1}\big)\xi^{-1}
		+\mathcal{O}(h^2)+\mathcal{O}(\varepsilon^2)\bigg).
	\end{aligned}
\end{equation*}	
Using the identity $P_0\Psi = P_0$, which follows from the definition of $\Psi$ and the orthogonality of the projection, together with the fact that $P_0(\alpha \times B_0)=0$ for all $\alpha\in\mathbb{R}^3$, the above relation simplifies to
\begin{equation*}\label{}
\ddot{\xi}_{0}^{0}
=P_0\Big(\dot{\xi}^{0} \times B_1(\xi^{0})+F(\xi^{0})
+\mathcal{O}(h^2)+\mathcal{O}(\varepsilon^2/h)\Big).
\end{equation*}	
Under the condition $h^2 > C^{*}\varepsilon$, the term $\mathcal{O}(\varepsilon^2/h)$ is dominated by $\mathcal{O}(h^2)$, which yields the first equation in \textsl{(c)}.

\textsl{ii)}  Inserting $k=0$ and $j=1$ into \eqref{Pfund}, using the expressions of  $d_{l}^{k}$, $\Upsilon^k$, $\Psi$ and the bounds of coefficients $\xi^k$,  this yields
\begin{equation*}\label{}
	\begin{aligned}
		\ddot{\xi}_{0}^{0}+\mathcal{O}(\varepsilon h^2)
		=P_1\Psi\bigg(&\dot{\xi}^{0} \times B(\xi^{0})+F(\xi^{0})
		+\frac{4i}{h}\tan\Big(\frac{h}{2\varepsilon}\Big)\xi^{1}\times B_{1}'(\xi^{0})\xi^{-1}\\
		&-\frac{4i}{h}\tan\Big(\frac{h}{2\varepsilon}\Big)\big(A''({\xi}^{0})\xi^{1}\big)\xi^{-1}
		+\mathcal{O}(h^2)+\mathcal{O}(\varepsilon^2)\bigg).
	\end{aligned}
\end{equation*}	
From the definition of $\Psi$ and the orthogonality of the projection, it follows that $P_1\Psi = \mathrm{tanc}(h/2 \varepsilon)P_1$. Moreover, using the identity $P_{\pm 1}(\alpha \times B_0)=\pm i P_{\pm 1}\alpha$ for all $\alpha\in\mathbb{R}^3$, the above relation can be rewritten as 
\begin{equation*}\label{}
\ddot{\xi}_{0}^{0}+\mathcal{O}(\varepsilon h^2)
=\frac{2i}{h}\tan\Big(\frac{h}{2\varepsilon}\Big)\dot{\xi}_{1}^{0}+\frac{2\varepsilon}{h}\tan\Big(\frac{h}{2\varepsilon}\Big)P_1\Big(\dot{\xi}^{0} \times B_1(\xi^{0})+F(\xi^{0})+\mathcal{O}(h^2)+\mathcal{O}(\varepsilon^2/h)\Big).
\end{equation*}	
We now derive the equation satisfied by $\dot{\xi}_{1}^{0}$, which appears in the dominant term with a factor $h^{-1}$. By the non-resonance condition \eqref{non-resonance conditions}, we have $\abs{\tan(h/2\varepsilon)} \geq c>0$. Furthermore, recalling that the bound $	\ddot{\xi}_{1}^{0}=\mathcal{O}(\varepsilon)$, we conclude that
	\begin{equation*}
		\dot{\xi}_{1}^{0}= 	i \varepsilon P_1\big(\dot{\xi}^{0} \times B_1(\xi^{0})+F(\xi^{0})\big)+\mathcal{O}(\varepsilon h),
	\end{equation*}
	which is the differential equation for $\xi_{1}^{0}$ stated in (c). The case $j=-1$ is obtained by taking complex conjugates.
	
\textsl{iii)} 	For $k=1$ and $j=1$, considering the expressions of the  leading coefficients $c_{l}^{k}$, $d_{l}^{k}$ and the bounds of $\xi^k$,  we obtain
	\begin{equation*}
	\begin{aligned}
&-\frac{4}{h^2}\tan^{2}\Big(\frac{h}{2 \varepsilon}\Big)\xi_{1}^{1}
+\frac{4 i}{h}\tan\Big(\frac{h}{2 \varepsilon}\Big)\sec^{2}\Big(\frac{h}{2 \varepsilon}\Big)\dot{\xi}_{1}^{1}
+\mathcal{O}(\varepsilon)\\ 
&=
P_1\Psi\bigg(\Big(\frac{2i}{h}\tan\Big(\frac{h}{2\varepsilon}\Big)\xi^{1}+\sec^2\Big(\frac{h}{2\varepsilon}\Big)\dot{\xi}_{1}
+\mathcal{O}(\varepsilon h)\Big)\times \Big(B(\xi^{0})+\mathcal{O}(\varepsilon^2)\Big) \\
&\qquad\qquad +\frac{i}{h}\tan\Big(\frac{h}{2 \varepsilon}\Big)A'''(\xi^0)(\xi^1,\xi^{-1})\xi^1+\mathcal{O}(\varepsilon)\bigg)
.
\end{aligned}
\end{equation*}
With the help of $P_{\pm 1}\Psi=\mathrm{tanc}(h/2\varepsilon)P_{\pm 1}$ and $P_{\pm 1}(\alpha \times B_0)=\pm i P_{\pm 1}\alpha$ for all $\alpha\in\mathbb{R}^3$, the above formula gives 
	\begin{equation*}
		\begin{aligned}
			&-\frac{4}{h^2}\tan^{2}\Big(\frac{h}{2 \varepsilon}\Big)\xi_{1}^{1}
			+\frac{4 i}{h}\tan\Big(\frac{h}{2 \varepsilon}\Big)\sec^{2}\Big(\frac{h}{2 \varepsilon}\Big)\dot{\xi}_{1}^{1}
			+\mathcal{O}(\varepsilon)\\ 
			&=-\frac{4}{h^2}\tan^{2}\Big(\frac{h}{2 \varepsilon}\Big)\xi_{1}^{1}
			+\frac{2 i}{h}\tan\Big(\frac{h}{2 \varepsilon}\Big)\sec^{2}\Big(\frac{h}{2 \varepsilon}\Big)\dot{\xi}_{1}^{1}
			+\frac{4 i \varepsilon }{h^2}\tan^{2}\Big(\frac{h}{2 \varepsilon}\Big)P_1\big(\xi_{1}^{1} \times B_1(\xi^{0}) \big)+\mathcal{O}(\varepsilon)+\mathcal{O}(\varepsilon^2/h).
		\end{aligned}
	\end{equation*}
It is worth noting that the elimination observed here results from defining $\Psi$
as in \eqref{MVF-phi}, which ensures that the first term on the right-hand side becomes identical to the first term on the left.  Consequently, the terms involving $\dot{\xi}_{1}^{1}$ become dominant. Then, under the non-resonance condition \eqref{non-resonance conditions}, the following result for $\xi_{1}^{1}$ is determined
	\begin{equation*}
		\dot{\xi}_{1}^{1}=\varepsilon h^{-1} \sin (h/\varepsilon)P_1\big(\xi_{1}^{1} \times B_1(\xi^{0}) \big)+\mathcal{O}(\varepsilon^2).
	\end{equation*}
Under the condition $h^2 >C^{*}\varepsilon$, this modulation equation coincides
with the corresponding equation satisfied by the exact coefficient $y_1^1$
in the modulated Fourier expansion of the exact solution up to $\mathcal{O}(h^2)$, which plays a key role in deriving the second-order error bound in the position.  A detailed error analysis based on the comparison of the modulated Fourier expansions
of the exact and numerical solutions can be found in Section~\ref{error-proof}.
Moreover, the case $j=-1$ is then given by taking the complex conjugates.
	
\textsl{(d)}	The initial value is given by
	\begin{equation*}
		\sum_{\abs{k} \leq N}e^{ikh/2\varepsilon}\xi^{k}(h/2) =x_{1/2},
	\end{equation*}
	which is  a result of \eqref{N-MFE}. In view of the bounds of $\xi^{\pm 1}$, the stated expression for $\xi^{0}(h/2)$ can be obtained immediately from this formula.
	
	By the definition of the operator  $h^{-1}L_1 L_2^{-1}(e^{hD})$, the approximate velocity in \eqref{AV1} can be transformed into
	\begin{equation}\label{initialv}
		v_{n+1/2}=\Phi		h^{-1}L_{1}L_{2}^{-1}(e^{hD})x_{n+1/2}.
			\end{equation}		
 Inserting \eqref{N-MFE} into the above formula and  multiplying  it by  $P_{0}$, combining the  coefficients $c_{l}^{k}$	with the bounds of $\xi^k$, we arrive at
			\begin{equation*}
				\begin{aligned}
					P_0	v_{1/2}
			&=P_0\Phi\Big(\sum_{\abs{k}  \leq N}	h^{-1}L_{1}L_{2}^{-1}(e^{hD})e^{ikh/  2\varepsilon}\xi^{k}(h/2)\Big)
			=P_0\Phi\Big(\sum_{\abs{k}  \leq N}e^{ikh/2\varepsilon}\sum_{l \geq 0}\varepsilon^{l-1} c_{l}^{k}\frac{\mathrm{d}^l}{{\mathrm{d}t}^l}\xi^{k}(h/2)\Big)\\
			&=P_0\Phi\big(\dot{\xi}^{0}(h/2)+2ih^{-1}\tan(h/2\varepsilon)
			\big(e^{ih/2\varepsilon}\xi_{1}^{1}(h/2)-e^{-ih/2\varepsilon}\xi_{-1}^{-1}(h/2)\big)
			+\mathcal{O}(h^2)\big)\\
			&=\dot{\xi}_{0}^{0}(h/2)+\mathcal{O}(h^2).
			\end{aligned}
		\end{equation*}		
Similarly,	a  multiplication of   \eqref{initialv} with $P_{\pm 1}$ gives
	\begin{equation*}
	\begin{aligned}
P_{\pm 1}v_{1/2}
&=P_{\pm1}\Phi\big(\dot{\xi}^{0}(h/2)+2ih^{-1}\tan(h/2\varepsilon)
\big(e^{ih/2\varepsilon}\xi_{1}^{1}(h/2)-e^{-ih/2\varepsilon}\xi_{-1}^{-1}(h/2)\big)
+\mathcal{O}(h^2)\big)\\
&=\pm i\varepsilon^{-1}e^{\pm ih/2\varepsilon}\sec(h/2\varepsilon) \xi_{\pm 1}^{\pm 1}(h/2)+\mathcal{O}(h).
\end{aligned}
\end{equation*}		
Here, we use the fact that $P_{\pm 1} \Phi \alpha=\mathrm{sinc} (h/2\varepsilon)^{-1}P_{\pm 1}\alpha$, which holds for any vector $\alpha \in \mathbb{C}^3$, together with the bound $\dot{\xi}_{\pm 1}^0 = \mathcal{O}(\varepsilon)$. The above analysis leads to the second and third equations of the initial values in \textsl{(d)}. The proof is complete.
\end{proof}

\subsubsection{Error bounds (Proof of Theorem \ref{error bound 2})}\label{error-proof}
We first consider the regime $h^2 > C^{*}\varepsilon$. From the  modulated Fourier expansion of the exact solution in Theorem \ref{Theoexact}, we have $x(t)=y^0(t)+\mathcal{O}(\varepsilon)$,
and Theorem \ref{thm: bounds} gives the numerical solution of
the two-step formulation \eqref{Method} for a stepsize $h^2 > C^{*}\varepsilon$ as $x_{n+1/2}=\xi^0(t_{n+1/2})+\mathcal{O}(h^2).$
Moreover, Theorems \ref{Theoexact} and \ref{thm: bounds} show  that the differential equations for $y^0(t)$ and $\xi^0(t_{n+1/2})$ are the same up to $\mathcal{O}(h^2)$. At the same time, their initial values  are the same up to $\mathcal{O}(h^2)$, which follows from comparing the modulated Fourier expansion of the exact solution $x(t_1)$ with that of the numerical solution $x_1$. The bound of the velocity $\dot{x}(t_1)$ and $v_1$ can be derived in the same way. We do not provide the details here, since it uses the same kinds of arguments as in Theorem \ref{thm: bounds}.
Thus,  $y^0(t)$ and $\xi^0(t)$ differ by $\mathcal{O}(h^2)$ on the time interval $0\leq t \leq T$.
This yields $\abs{x_{n+1/2}-x(t_{n+1/2})} \leq  Ch^2.$
Combining the initial condition $x(t_0) = x_0$ with the result above yields $\abs{x_{n+1}-x(t_{n+1})} \leq  Ch^2.$
These results provide the error bounds for the position, as stated in Theorem~\ref{error bound 2}.

We now derive the error bound for the velocity. By Theorem \ref{Theoexact},  the modulated Fourier expansion of the velocity is shown as
\begin{equation*}\label{vt}
	v(t)=\dot{x}(t)=\dot{y}^{0}(t)+i\varepsilon^{-1}\big(y_{1}^{1}(t)e^{it/\varepsilon}-y_{-1}^{-1}(t)e^{-it/\varepsilon}\big)+\mathcal{O}(\varepsilon).
\end{equation*}
From the modulated Fourier expansion of the velocity  \eqref{initialv}, it follows that
\begin{equation*}
		v_{n+\frac{1}{2}}
		=\Phi\sum_{\abs{k}  \leq N}e^{ikt/\varepsilon}\sum_{l \geq 0}\varepsilon^{l-1} c_{l}^{k}\frac{\mathrm{d}^l}{{\mathrm{d}t}^l}\xi^{k}(t)
        =\Phi\Big(\dot{\xi}^{0}+\frac{2i}{h}\tan\Big(\frac{h}{2\varepsilon}\Big)	\big(\xi_{1}^{1}e^{it/\varepsilon}-\xi_{-1}^{-1}e^{-it/\varepsilon}\big)
		+\mathcal{O}(h^2)\Big).
\end{equation*}
Splitting $\Phi \cdot \dot{\xi}^{0}$ into $\dot{\xi}^{0}+(\Phi -I)\dot{\xi}^{0}$ and inserting the definition of $\Phi$ in \eqref{MVF-phi} into the above formula, we have  
\begin{equation*}
		v_{n+\frac{1}{2}}
		=\dot{\xi}^{0}+i\varepsilon^{-1}\sec(h/2\varepsilon)\big(\xi_{1}^{1}e^{it/\varepsilon}-\xi_{-1}^{-1}e^{-it/\varepsilon}\big)+(\Phi -I)\dot{\xi}^{0}
		+\mathcal{O}(h^2).
\end{equation*}
After multiplication  with $P_j$ for $j=-1,0,1$, it leads to
\begin{equation*}
	\begin{aligned}
		P_0 v_{n+\frac{1}{2}}=\dot{\xi}^{0}_{0}+\mathcal{O}(h^2),\quad \ P_{\pm1} v_{n+\frac{1}{2}}=\dot{\xi}^{0}_{\pm1}\pm i\varepsilon^{-1}\sec(h/2\varepsilon)e^{\pm it/\varepsilon} \xi_{\pm 1}^{\pm 1}+\mathcal{O}(h).
	\end{aligned}
\end{equation*}
Comparing the numerical and exact expansions yields the stated velocity error bounds. 

We next consider the regime $c^{*}\varepsilon^2 \le h^2 \le C^{*}\varepsilon$. 
In this case, the errors in the position and the parallel velocity satisfy an $\mathcal{O}(\varepsilon)$ bound. 
This estimate follows from a modulated Fourier expansion argument analogous to that employed for $h^2 > C^{*}\varepsilon$.
This completes the proof.

\subsection{Conservation properties (Proof of Theorems \ref{conservation-H2}--\ref{conservation-M2})} \label{sec5-Conservation}
Based on the idea of modulated Fourier expansions presented in Chap.~\uppercase\expandafter{\romannumeral13} of \cite{hairer2006}, we demonstrate that the modulated Fourier expansion of the numerical solution admits two almost-invariants. 
These almost-invariants are then used to establish long-time near-conservation of the energy and the magnetic moment. We first present the proof of Theorem~\ref{conservation-H2} under the condition $h^2 > C^{*}\varepsilon$. To simplify the derivation in this regime, we introduce the assumption $h^{m} \leq \varepsilon$ for some fixed $ m > 2 $ and choose $ M \geq mN $. This simplifies the treatment of remainder terms and enables the analysis on the time scale $\varepsilon^{-N}$. The same assumption will also be used in the proof of long-time near-conservation of magnetic moment. In the absence of this assumption, the result still holds for $\textmd{min}(h^{-M},\varepsilon^{-N})$.
For the case $h^2 \le C^{*}\varepsilon$, the argument follows the same overall strategy, with an additional discussion for sufficiently small step sizes $h^2 < c^{*}\varepsilon^2$, where backward error analysis is employed to establish the corresponding energy estimate.

\subsubsection{Long-time  near-conservation of energy (Proof of Theorem \ref{conservation-H2})} \label{sec5.2}
The proof proceeds by first establishing  the existence
of an almost-invariant for the modulated functions of the filtered two-step variational integrator that is close to the total energy in the following lemma. Based on this result, the long-time near-conservation of energy stated in Theorem~\ref{conservation-H2} is then derived.
\begin{lem}\label{lem1}
	Under the conditions of Theorem \ref{thm: bounds}, there exists an almost-invariant  $\mathcal{H}[\bold{z}](t)$ with  $\bold{z}=(z^{k})_{k \in \mathbb{Z}}$, such that for $0\leq t \leq T$
	\begin{equation*}
\mathcal{H}[\bold{z}](t)=\mathcal{H}[\bold{z}](0)
+\mathcal{O}(t\varepsilon^N),\ \
\mathcal{H}[\bold{z}](t_{n+\frac{1}{2}})=H(x_{n+\frac{1}{2}}, v_{n+\frac{1}{2}}) +\mathcal{O}(h) \ \ \textmd{for} \ \ 0\leq(n+1/2)h \leq T.
	\end{equation*}		
The constants symbolised by $\mathcal{O}$-notation are independent of $n$, $h$ and $\varepsilon$, but depend on  $N$.
\end{lem}

\begin{proof}
 { $\bullet$  \textbf{Proof of the first statement.}}
With the following extended potentials
	\begin{equation*}
	\begin{aligned}
		\mathcal{U}(\bold{z})= \sum_{ \underset{s(\alpha)=0 }{0 \leq m \leq N  }}
		\frac{1}{m!}U^{(m)}(z^0)\bold{z}^{\alpha},\ \quad
		\mathcal{A}(\bold{z})=(\mathcal{A}_k(\bold{z}))_{k \in \mathbb{Z} }
		=\Big(\sum_{ \underset{s(\alpha)=k }{0 \leq m \leq N+1  }}
		\frac{1}{m!}A^{(m)}(z^0)\bold{z}^{\alpha}\Big)_{k \in \mathbb{Z} },
	\end{aligned}
\end{equation*}
which are  given by \cite{Hairer2018} and satisfy the
following  invariance properties
\begin{equation*}\label{}
	\mathcal{U}(S(\lambda)\bold{z})=\mathcal{U}(\bold{z}),\quad	\mathcal{A}(S(\lambda)\bold{z})=S(\lambda)\mathcal{A}(\bold{z})
\end{equation*}	
 for all $\lambda \in \mathbb{R}$ and  $S(\lambda)\bold{z}:=(e^{ik\lambda}z^k)_{k \in \mathbb{Z} }$.
Using the functions $\xi^k$ constructed in Theorem~\ref{thm: bounds}, we insert the modulated Fourier expansion of the numerical solution into \eqref{opera-method} and obtain
	\begin{equation}\label{modulation functions}	
		\Psi^{-1}\Big(\frac{1}{h^2}L_1^2 L_2^{-2}(e^{hD})z^k\Big)
		=  \sum_{j \in \mathbb{Z}}\Big(\frac{\partial{ \mathcal{A}_j}}{\partial{ z^k}}(\bold{z})\Big)^{\ast}\frac{L_1 L_2^{-1}(e^{hD})}{h}z^j
		-\frac{L_1 L_2^{-1}(e^{hD})}{h} \mathcal{A}_k(\bold{z})
		-\Big(\frac{\partial{ \mathcal{U}}}{\partial{ z^k}}(\bold{z})\Big)^{\ast}+\mathcal{O}(\varepsilon^N),
	\end{equation}	
where we let
$z^k(t)=e^{ikt/  \varepsilon}\xi^{k}(t)$	and assume that $z^k(t)=0$ for $\abs{k} > N$.

	Multiplication \eqref{modulation functions} with $(\dot{z}^{k})^{\ast}$ and summation over $k$ gives	
	\begin{equation}\label{total derivatives}	\begin{aligned}
		&\sum_{k}	(\dot{z}^{k})^{\ast}\Psi^{-1}\frac{L_1^2 L_2^{-2}(e^{hD})}{h^2}z^k
		-\sum_{k}\Big( \frac{\mathrm{d}}{\mathrm{d}t}\mathcal{A}_k(\bold{z})^{\ast}\frac{L_1 L_2^{-1}(e^{hD})}{h}z^k
		-(\dot{z}^{k})^{\ast}\frac{L_1 L_2^{-1}(e^{hD})}{h}\mathcal{A}_k(\bold{z})\Big)
		\\&\ +\frac{\mathrm{d}}{\mathrm{d}t}\mathcal{U}(\bold{z})
		=\mathcal{O}(\varepsilon^N).\end{aligned}
	\end{equation}
The analysis provided in the proof of Theorem 3 in \cite{wangwu2020} indicates that each of the three terms
on the left-hand side is a total differential up to $\mathcal{O}(\varepsilon^N)$. Therefore, there exists a function $\mathcal{H}[\bold{z}](t)$ such that $\frac{\mathrm{d}}{\mathrm{d}t}\mathcal{H}[\bold{z}](t)=\mathcal{O}(\varepsilon^N).$ The first statement of the lemma is shown.
	
{ $\bullet$  \textbf{Proof of the second statement.}}	In what follows, we shall show the leading term of
	\begin{equation*}\label{}
		\mathcal{H}[\bold{z}](t)=	\mathcal{K}[\bold{z}](t)+	\mathcal{M}[\bold{z}](t)+	\mathcal{U}[\bold{z}](t),	
	\end{equation*}	
	where the time derivatives of $\mathcal{K}$, $\mathcal{M}$ and $\mathcal{U}$ equal the three corresponding terms on the left-hand side of \eqref{total derivatives}.	
	
	For $\mathcal{K}[\bold{z}]$, we can get
		\begin{equation*}
		\begin{aligned}
			&\frac{\mathrm{d}}{\mathrm{d}t}\mathcal{K}[\bold{z}]
			=\sum_{k}(\dot{z}^{k})^{\ast}\Psi^{-1}h^{-2}L_1^2 L_2^{-2}(e^{hD})z^k
			=\sum_{k}\Big(\dot{\xi}^{k}+\frac{ik}{\varepsilon}\xi^{k}\Big)^{\ast}\Psi^{-1}
			\Big(\sum_{l \geq 0}\varepsilon^{l-2} d_{l}^{k}\frac{\mathrm{d}^l}{{\mathrm{d}t}^l}\xi^{k}\Big)\\
			&\ =(\dot{\xi}^{0})^{\ast}\Psi^{-1}\ddot{\xi}^{0}
			+\sum_{k=\pm 1}\Big(\dot{\xi}^{k}+\frac{ik}{\varepsilon}\xi^{k}\Big)^{\ast}\Psi^{-1}
			\Big(\sum_{l \geq 0}\varepsilon^{l-2} d_{l}^{k}\frac{\mathrm{d}^l}{{\mathrm{d}t}^l}\xi^{k}\Big)+\mathcal{O}(h)\\
			&\ =(\dot{\xi}^{0})^{\ast}\Psi^{-1}\ddot{\xi}^{0}+\Big(\dot{\xi}^{1}+\frac{i}{\varepsilon}\xi^{1}\Big)^{\ast}\Psi^{-1}
			\Big(-\frac{4}{h^2}\tan^2\Big(\frac{h}{2\varepsilon}\Big)\xi^{1}+\frac{4i}{h}\tan\Big(\frac{h}{2\varepsilon}\Big)\sec^2\Big(\frac{h}{2\varepsilon}\Big)\dot{\xi}^{1}+\mathcal{O}(\varepsilon)\Big)\\
			&\quad\ +\Big(\dot{\xi}^{-1}-\frac{i}{\varepsilon}\xi^{-1}\Big)^{\ast}\Psi^{-1}
			\Big(-\frac{4}{h^2}\tan^2\Big(\frac{h}{2\varepsilon}\Big)\xi^{-1}-\frac{4i}{h}\tan\Big(\frac{h}{2\varepsilon}\Big)\sec^2\Big(\frac{h}{2\varepsilon}\Big)\dot{\xi}^{-1}+\mathcal{O}(\varepsilon)\Big)+\mathcal{O}(h)\\
			&\ =(\dot{\xi}^{0})^{\ast}\Psi^{-1}\ddot{\xi}^{0}- \frac{8}{h^2}\tan^{2}\Big(\frac{h}{2\varepsilon}\Big)(\dot{\xi}^{1})^{\ast}\Psi^{-1}\xi^{1}
			+\frac{8}{\varepsilon h}\tan\Big(\frac{h}{2\varepsilon}\Big) \sec^{2}\Big(\frac{h}{2\varepsilon}\Big)(\xi^{1})^{\ast}\Psi^{-1}\dot{\xi}^{1}
			+\mathcal{O}(h).
		\end{aligned}
	\end{equation*}
	This yields
	\begin{equation*}\label{}
		\mathcal{K}[\bold{z}]=\frac{1}{2}	(\dot{\xi}^{0})^{\ast}\Psi^{-1}\dot{\xi}^{0}- 4\Big(\frac{1}{h^2}\tan^{2}\Big(\frac{h}{2\varepsilon}\Big)-\frac{1}{\varepsilon h}\tan\Big(\frac{h}{2\varepsilon}\Big) \sec^{2}\Big(\frac{h}{2\varepsilon}\Big)\Big)(\xi^{1})^{\ast}\Psi^{-1}\xi^{1}
		+\mathcal{O}(h).	
	\end{equation*}	
Considering further the fact that $\Psi^{-1}
	=I+\big(1-\mathrm{tanc}(h/2\varepsilon)^{-1}\big){\tilde{B}_0}^2$,  one arrives at
	\begin{equation*}
		\begin{aligned}
			(\dot{\xi}^{0})^{\ast}\Psi^{-1}\dot{\xi}^{0}=(\dot{\xi}^{0})^{\ast}\dot{\xi}^{0}+\mathcal{O}(\varepsilon^2), \ \quad
			(\xi^{1})^{\ast}\Psi^{-1}\xi^{1}=
			\mathrm{tanc}(h/2\varepsilon)^{-1}(\xi_{1}^{1})^{\ast}\xi_{1}^{1}	+\mathcal{O}(\varepsilon^2),
		\end{aligned}
	\end{equation*}
where $\xi^{k}=\xi_{1}^{k}+\xi_{0}^{k}+\xi_{-1}^{k}$, the orthogonal property and the bounds of $\xi^k$ are used here. So, we obtain
	\begin{equation*}\label{}
		\mathcal{K}[\bold{z}]=\frac{1}{2}\abs{\dot{\xi}^{0}}^2- 2\Big(\frac{1}{h\varepsilon}\tan\Big(\frac{h}{2\varepsilon}\Big)-\frac{1}{\varepsilon^2} \sec^{2}\Big(\frac{h}{2\varepsilon}\Big)\Big)\abs{\xi_{1}^{1}}^2
		+\mathcal{O}(h).	
	\end{equation*}

	We now consider $\mathcal{M}[\bold{z}]$.  It is easy to see that the term $k=0$ is of size $\mathcal{O}(h^2)$,  while the dominating terms arise from  $k=\pm 1$. With some calculation, it is deduced that
\begin{equation*}
\begin{aligned}
&\sum_{k=\pm1}\Big( \frac{\mathrm{d}}{\mathrm{d}t}\mathcal{A}_k(\bold{z})^{\ast}h^{-1}L_1 L_2^{-1}(e^{hD})z^k-(\dot{z}^{k})^{\ast}h^{-1}L_1 L_2^{-1}(e^{hD})\mathcal{A}_k(\bold{z})\Big)\\
&=\sum_{k=\pm1}\varepsilon^{-1}c_{0}^{k}\Big(\frac{\mathrm{d}}{\mathrm{d}t}\big(A'(\xi^0)\xi^{k}\big)^{\ast}\xi^{k}
+(\dot{\xi}^{k})^{\ast}A'(\xi^0)\xi^{k}\Big)
+\mathcal{O}(\varepsilon).
		\end{aligned}
	\end{equation*}	
With the help of  $P_{\pm 1}(v \times B_0)=\pm i P_{\pm 1}v$, we then have $\mathcal{M}[\bold{z}]=-\frac{2}{\varepsilon h}\tan\big(\frac{h}{2\varepsilon}\big)\abs{\xi_{1}^{1}}^2
		+\mathcal{O}(h^2).$	
For $\mathcal{U}[\bold{z}]$, it is clear that $\mathcal{U}[\bold{z}]=U(\xi^{0})+\mathcal{O}(h)$.
	Therefore, $\mathcal{H}[\bold{z}]$ is determined by
	\begin{equation}\label{almost-H}
		\mathcal{H}[\bold{z}]=	\frac{1}{2}\abs{\dot{\xi}^{0}}^2+\frac{1}{\varepsilon^2}\sec^{2}\Big(\frac{h}{2\varepsilon}\Big)
		\abs{\xi_{1}^{1}}^2+U(\xi^{0})+\mathcal{O}(h).	
	\end{equation}
    
On the other hand, we now turn to consider the modulated Fourier expansion of $v_{n+1/2}$.  The combination of the  operator  $\frac{1}{h}L_1 L_2^{-1}(e^{hD})$, the definition of $\Phi$, and the bounds of the coefficient functions $\xi^k$ gives
	\begin{equation}\label{v12}
		\begin{aligned}
			v_{n+\frac{1}{2}}
			&=\Phi\sum_{\abs{k}  \leq N}e^{ikt/\varepsilon}\sum_{l \geq 0}\varepsilon^{l-1} c_{l}^{k}\frac{\mathrm{d}^l}{{\mathrm{d}t}^l}\xi^{k}(t)
			=\Phi\Big(\dot{\xi}^{0}+\frac{2i}{h}\tan\Big(\frac{h}{2\varepsilon}\Big)
			\big(\xi_{1}^{1}e^{it/\varepsilon}-\xi_{-1}^{-1}e^{-it/\varepsilon}\big)+\mathcal{O}(h^2)
			\Big)\\
			&=\dot{\xi}^{0}+\frac{i}{\varepsilon}\sec\Big(\frac{h}{2\varepsilon}\Big)
			\big(\xi_{1}^{1}e^{it/\varepsilon}-\xi_{-1}^{-1}e^{-it/\varepsilon}\big)
			+\mathcal{O}(h).\\
		\end{aligned}
	\end{equation}
It then follows that 
	\begin{equation}\label{eaxctH-MFE}
		\begin{aligned}
			H(x_{n+\frac{1}{2}}, v_{n+\frac{1}{2}})
			=\frac{1}{2}\abs{v_{n+\frac{1}{2}}}^2+U(x_{n+\frac{1}{2}})=
			\frac{1}{2}\abs{\dot{\xi}^{0}}^2+\frac{1}{\varepsilon^2}\sec^{2}\Big(\frac{h}{2\varepsilon}\Big)\abs{\xi_{1}^{1}}^2+U(\xi^{0})+\mathcal{O}(h).
		\end{aligned}	
	\end{equation}
Therefore, we have
	\begin{equation}\label{energyH}
 \mathcal{H}[\bold{z}](t_{n+\frac{1}{2}})=H(x_{n+\frac{1}{2}}, v_{n+\frac{1}{2}})+\mathcal{O}(h).
	\end{equation}
Note that the filter function $\Phi$ acting in the velocity component is chosen such that the second term in the right-hand side of \eqref{eaxctH-MFE} coincides with the corresponding second term of the almost-invariant $\mathcal H[\bold{z}]$ given in \eqref{almost-H}. This matching property induced by $\Phi$ is essential for establishing \eqref{energyH} and underlies the subsequent long-term analysis of the energy. Moreover, this specific form of $\Phi$ is crucial for the subsequent analysis of the long-time near-conservation of the magnetic moment.
\end{proof}	

We are now in a position to prove the long-time near-conservation of the energy. To this end, we decompose the time interval into subintervals of length $h$ and define $t_{n+1/2}=(n+1/2)h$. By repeatedly applying the relations of Lemma~\ref{lem1} on  intervals of length $h$, we consider the modulated Fourier expansions associated with different starting values. On each subinterval $[t_{n-1/2}, t_{n+1/2}]$, Theorem~\ref{thm: bounds} provides modulated Fourier expansion coefficients $\bold{z}_{n+1/2}(t)$ corresponding to starting values $(x_{n+1/2},v_{n+1/2})$. The uniqueness of the modulation system
up to $\mathcal{O}(\varepsilon^{N+1})$ implies that $\bold{z}_{n-1/2}(h)=\bold{z}_{n+1/2}(0)+\mathcal{O}(\varepsilon^{N+1})$. The first statement of Lemma \ref{lem1} then gives
$ 
\mathcal{H}[\bold{z}_{n+1/2}](0)=\mathcal{H}[\bold{z}_{n-1/2}](h)+\mathcal{O}(\varepsilon^{N+1} )=\mathcal{H}[\bold{z}_{n-1/2}](0)+\mathcal{O}(h\varepsilon^N).	$
Iterating this relation yields
\begin{equation*}
	\begin{aligned}
		\mathcal{H}[\bold{z}_{n+1/2}](0)=\mathcal{H}[\bold{z}_{1/2}](0)+\mathcal{O}(nh\varepsilon^N).		
	\end{aligned}
\end{equation*}
Moreover, by the second statement of Lemma~\ref{lem1}, the modulated coefficients associated with the starting values
$(x_{n+1/2},v_{n+1/2})$ and $(x_{1/2},v_{1/2})$ satisfy
$\mathcal{H}[\bold{z}_{n+1/2}](0)=H(x_{n+1/2}, v_{n+1/2}) +\mathcal{O}(h)$
and $\mathcal{H}[\bold{z}_{1/2}](0)=H(x_{1/2}, v_{1/2}) +\mathcal{O}(h).$ 	 Combining the above estimates, we obtain	 		
\begin{equation*}
	\begin{aligned}
H(x_{n+1/2}, v_{n+1/2})-H(x_{1/2}, v_{1/2})=\mathcal{H}[\bold{z}_{n+1/2}](0)-\mathcal{H}[\bold{z}_{1/2}](0) +\mathcal{O}(h)=\mathcal{O}(h)+\mathcal{O}(nh\varepsilon^N).	 		
	\end{aligned}
\end{equation*} 
This completes the proof of the long-time near-conservation of energy stated in Theorem~\ref{conservation-H2} for the regime $h^2 >C^{*}\varepsilon$.

We next consider the remaining step-size regime $h^2 \leq C^{*}\varepsilon$,  which is divided into two subregimes.
For $c^{*}\varepsilon^2 \leq h^2 \leq C^{*}\varepsilon$, the modulated Fourier expansion yields long-time near-conservation of the energy with an $\mathcal{O}(\varepsilon)$ bound, using the same analytical framework as in the regime $h^2 > C^{*}\varepsilon$. For sufficiently small step sizes $h^2 < c^{*}\varepsilon^2$, two analytical approaches are considered. 
The modulated Fourier expansion yields long-time near-conservation  of the energy with an $\mathcal{O}(\varepsilon)$ bound,  while backward error analysis, carried out as in the moderate magnetic field case, 
gives an estimate of order $\mathcal{O}(h^2/\varepsilon^2)$. 
For brevity, the details are omitted. Combining these two results leads to an energy error of order $\mathcal{O}(\min\{h^2/\varepsilon^2,\varepsilon\})$, 
which is valid throughout the regime $h^2 \leq C^{*}\varepsilon$.

\subsubsection{ Long-time  near-conservation of magnetic moment (Proof of Theorem \ref{conservation-M2})} \label{sec5.3}
Similar to the proof of Theorem~\ref{conservation-H2}, an almost-invariant associated with the magnetic moment is first established in the following lemma, which is then used to prove its long-time near-conservation for the filtered variational integrator.
\begin{lem}
	Under the conditions of Theorem \ref{thm: bounds}, there exists an almost-invariant  $\mathcal{I}[\bold{z}](t)$, such that for $0\leq t \leq T$
    \begin{equation*}
\mathcal{I}[\bold{z}](t)=\mathcal{I}[\bold{z}](0)+\mathcal{O}(t\varepsilon^N),\ \
\mathcal{I}[\bold{z}](t_{n+\frac{1}{2}})=I(x_{n+\frac{1}{2}}, v_{n+\frac{1}{2}}) +\mathcal{O}(h) \ \ \textmd{for} \ \ 0\leq(n+1/2)h \leq T,
	\end{equation*}		
where the constants symbolised by $\mathcal{O}$-notation depend on $N$, but are independent of $n$, $h$ and $\varepsilon$.
\end{lem}
\begin{proof}
{ $\bullet$  \textbf{Proof of the first statement.}}
The identities
	\begin{equation*}\label{}
		\begin{aligned}
			\sum_{k \in \mathbb{Z}}ik\frac{\partial{ \mathcal{U}}}{\partial{ z^k}}(\bold{z})z^k=0,\ \quad
				\sum_{j \in \mathbb{Z}}i j \frac{\partial{ \mathcal{A}_k}}{\partial{ z^j}}(\bold{z})z^j=i k\mathcal{A}_k(\bold{z}) \quad \textmd{for} \ \ k \in \mathbb{Z}	
		\end{aligned}	
	\end{equation*}
follow from the invariance properties of  $\mathcal{U}(\bold{z})$ and $\mathcal{A}(\bold{z})$, as discussed in \cite{Hairer2018}. Based on these formulae, multiplying  \eqref{modulation functions} with $-ik(z^{k})^{\ast}/\varepsilon$ and summing over $k$,  it follows that
	\begin{equation}\label{I}			\begin{aligned}
&\varepsilon^{-1}\sum_{k}ik\big(\mathcal{A}_k(\bold{z})^{\ast}h^{-1}L_1 L_2^{-1}(e^{hD})z^k-(z^{k})^{\ast}h^{-1}L_1 L_2^{-1}(e^{hD})\mathcal{A}_k(\bold{z})\big)\\
&-\varepsilon^{-1}\sum_{k}ik(z^{k})^{\ast}\Psi^{-1}\big(h^{-2}L_1^2 L_2^{-2}(e^{hD})z^k\big)=  \mathcal{O}(\varepsilon^{N}).		\end{aligned}
	\end{equation}
	Similar to the  analysis in Lemma \ref{lem1}, it can be seen that the real part of this left-hand side is a total time derivative. Thus, there exists a function $\mathcal{I}[\bold{z}](t)$ such that $\frac{\mathrm{d}}{\mathrm{d}t}\mathcal{I}[\bold{z}](t)=\mathcal{O}(\varepsilon^{N})$.
The first statement of the lemma is shown immediately from this result.

{ $\bullet$  \textbf{Proof of the second statement.}}	Considering the dominant terms of $\mathcal{I}[\bold{z}]$ as follows.  For $k=0$, we find that the left-hand side of \eqref{I} is zero and it is easy to check that the first sum is of size $\mathcal{O}(\varepsilon)$ for $k=\pm 1$. Therefore, we turn to the second term for  $k=\pm 1$ to find the dominant part of $\mathcal{I}[\bold{z}]$. From the following  “magic formula" (see Chap. \uppercase\expandafter{\romannumeral13} of  \cite{hairer2006})
\begin{equation*}
\Im \overline{z}^{\intercal}z^{(2l+2)}=\Im\frac{\mathrm{d}}{\mathrm{d}t}\big(\overline{z}^{\intercal}z^{(2l+1)}-\dot{\overline{z}}^{\intercal}z^{(2l)}+\cdots \pm (\overline{z}^{(l)})^{\intercal}z^{(l+1)} \big),
\end{equation*}
one arrives at
\begin{equation*}
\begin{aligned}
&\Im \big((z^{k})^{\ast}\Psi^{-1}h^{-2}L_1^2 L_2^{-2}(e^{hD})z^k\big)=\sum_{l \geq 0}\beta_{2l}h^{2l}\Im\big((z^{k})^{\ast}\Psi^{-1}(z^k)^{2l+2}\big)\\
&\ =\sum_{l \geq 0}\beta_{2l}h^{2l}\Im\frac{\mathrm{d}}{\mathrm{d}t}\big((z^{k})^{\ast}\Psi^{-1}(z^{k})^{(2l+1)}-(\dot{z}^{k})^{\ast}\Psi^{-1}(z^{k})^{(2l)}
			+\cdots \pm \big((z^{k})^{(l)}\big)^{\ast}\Psi^{-1}(z^{k})^{(l+1)}\big)\\
 &\ =\frac{\mathrm{d}}{\mathrm{d}t}\sum_{l \geq 0}\beta_{2l}h^{2l}\Im\big((z^{k})^{\ast}\Psi^{-1}(z^{k})^{(2l+1)}-(\dot{z}^{k})^{\ast}\Psi^{-1}(z^{k})^{(2l)}
			+\cdots \pm \big((z^{k})^{(l)}\big)^{\ast}\Psi^{-1}(z^{k})^{(l+1)}\big).
\end{aligned}
\end{equation*}
Thanks to the Lemma 5.1 in \cite{Hairer2016}, we have
	\begin{equation*}
		\frac{1}{m!}\frac{\mathrm{d}^m}{\mathrm{d}t^m}{z}^{k}(t)=\frac{1}{m!}\xi^{k}(t)\Big(\frac{ik}{\varepsilon}\Big)^{m}e^{ikt/\varepsilon}
		+\mathcal{O}\Big(\frac{1}{(m/M)!}\Big(\frac{c}{\varepsilon}\Big)^{m-1-\abs{k}}\Big),
	\end{equation*}	
	where $c$ and the notation $\mathcal{O}$ are independent of $m \geq 1$ and $\varepsilon$. Combine this formula with the term
	$(-1)^r \frac{\mathrm{d}^r}{\mathrm{d}t^r}\big({\xi}^{k}(t)\big)^{\ast}\Psi^{-1}\frac{\mathrm{d}^s}{\mathrm{d}t^s}{\xi}^{k}(t)$, we find that the dominant term is  same for $r+s=2l+1.$ Then, it can be deduced that
	\begin{equation*}
		\begin{aligned}
			&(z^{k})^{\ast}\Psi^{-1}(z^{k})^{(2l+1)}-(\dot{z}^{k})^{\ast}\Psi^{-1}(z^{k})^{(2l)}
			+\cdots \pm \big((z^{k})^{(l)}\big)^{\ast}\Psi^{-1}(z^{k})^{(l+1)}\\
		   &= (l+1)\big(\frac{ik}{\varepsilon}\big)^{2l+1}(z^{k})^{\ast}\Psi^{-1}z^{k}	+\mathcal{O}\Big(\frac{1}{(l/M)!}\Big(\frac{c}{\varepsilon}\Big)^{2l-2\abs{k}}\Big),
		\end{aligned} 
	\end{equation*}
which further implies that the total derivative of the second  term on the left-hand side of \eqref{I} is given by
	\begin{equation*}
		\begin{aligned}
			&-\frac{i}{\varepsilon}\sum_{k}\frac{ik}{h} \sum_{l \geq 0}\Big((-1)^l\beta_{2l}(l+1) \Big(\frac{kh}{\varepsilon}\Big)^{2l+1}({\xi}^{k})^{\ast}\Psi^{-1}{\xi}^{k}\Big)
			+\mathcal{O}(\varepsilon)\\
			&\ =\frac{1}{\varepsilon h}\sum_{k}\frac{k}{2} \sum_{l \geq 0}\Big((-1)^l\beta_{2l}(2l+2) \Big(\frac{kh}{\varepsilon}\Big)^{2l+1}({\xi}^{k})^{\ast}\Psi^{-1}{\xi}^{k}\Big)
			+\mathcal{O}(\varepsilon)\\	
			&\ =\frac{1}{\varepsilon h}\sum_{k}\frac{k}{2} 2\tan\Big(\frac{kh}{2\varepsilon}\Big) \sec^{2}\Big(\frac{kh}{2\varepsilon}\Big)({\xi}^{k})^{\ast}\Psi^{-1}{\xi}^{k}
			+\mathcal{O}(\varepsilon)\\
			&\ =\frac{2}{\varepsilon h}\tan\Big(\frac{h}{2\varepsilon}\Big) \sec^{2}\Big(\frac{h}{2\varepsilon}\Big)({\xi}^{1})^{\ast}\Psi^{-1}{\xi}^{1}
			+\mathcal{O}(\varepsilon)\\
			 &\ =\frac{2}{\varepsilon h}\tan\Big(\frac{h}{2\varepsilon}\Big) \sec^{2}\Big(\frac{h}{2\varepsilon}\Big)\Big(\mathrm{tanc}\Big(\frac{h}{2 \varepsilon}\Big)^{-1}\abs{\xi_{1}^{1}}^{2}	+\mathcal{O}((\varepsilon^2)\Big)
			 +\mathcal{O}(\varepsilon)\\	
			&\ =\frac{1}{\varepsilon^2}\sec^2\Big(\frac{h}{2\varepsilon}\Big)	\abs{\xi_{1}^{1}}^2+\mathcal{O}(h).
		\end{aligned}
	\end{equation*}

	On the other hand, by combining the expansion of $v_{n+1/2}$ in \eqref{v12} with the relations $\dot{\xi}^{0} \times B_0=\mathcal{O}(\varepsilon)$ and $P_{\pm 1}(v \times B_0)=\pm i P_{\pm 1}v$, we obtain $v_{n+1/2} \times B_0
=-\varepsilon^{-1}\sec(h/2\varepsilon)\big(\xi_{1}^{1}e^{it/\varepsilon}+\xi_{-1}^{-1}e^{-it/\varepsilon}\big)+\mathcal{O}(h).$
The following relation then holds
	\begin{equation*}
		\begin{aligned}
			&I(x_{n+1/2}, v_{n+1/2} )
			=\frac{1}{2}\abs{v_{n+1/2} \times B_0}^2+\mathcal{O}(\varepsilon)\\
			&\ =\frac{1}{2\varepsilon^2}\sec^2\Big(\frac{h}{2\varepsilon}\Big)\big((\xi_{1}^{1})^{\ast}\xi_{1}^{1}
			+(\xi_{1}^{1})^{\ast}\xi_{-1}^{-1}e^{-2it/\varepsilon}+(\xi_{-1}^{-1})^{\ast}\xi_{1}^{1}e^{2it/\varepsilon} 
			+(\xi_{-1}^{-1})^{\ast}\xi_{-1}^{-1}\big)+\mathcal{O}(h)\\
			&\ =\frac{1}{\varepsilon^2}\sec^2\Big(\frac{h}{2\varepsilon}\Big)	\abs{\xi_{1}^{1}}^2+\mathcal{O}(h).
		\end{aligned}
	\end{equation*}
Therefore, it follows that $\mathcal{I}[\bold{z}](t_{n+1/2})=I(x_{n+1/2}, v_{n+1/2})+\mathcal{O}(h).$ 
\end{proof}

By combining the above results with the detailed analysis carried out in Theorem \ref{conservation-H2}, 
we are able to establish the desired long-time near-conservation  for the magnetic moment.  This concludes the proof of Theorem~\ref{conservation-M2}.

\section{Conclusions} \label{sec6}	
In this paper, we study charged-particle dynamics (CPD) in both the moderate magnetic field regime and the near-uniform strong magnetic field regime, and propose a novel filtered two-step variational integrator together with its theoretical analysis. For the moderate regime $\varepsilon = 1$, the proposed method is shown to achieve second-order accuracy, and its long-time near-conservation of energy and momentum is established by backward error analysis. For the strong magnetic field regime $0 < \varepsilon \ll 1$,  the analysis is carried out within the framework of modulated Fourier expansion.  This yields an $\mathcal{O}(h^2)$ error bound for $h^2 \ge C^{*}\varepsilon$ and an $\mathcal{O}(\varepsilon)$ accuracy for $c^{*}\varepsilon^2 \le h^2 \le C^{*}\varepsilon$.  The long-time near-conservation of the energy and the magnetic moment is also established by modulated Fourier expansion in these step-size regimes.  For the smaller-step-size regime $h^2 < c^{*}\varepsilon^2$, the long-time behavior is further analyzed by combining modulated Fourier expansion with backward error analysis.

Finally, it is worth noting that applying a large-stepsize method to CPD in a strong non-uniform magnetic field is an interesting but very challenging topic. The work \cite{Xiao21} succeeded in proposing a large-stepsize modified Boris method for solving CPD in a strong non-uniform magnetic field, but without error analysis. The authors  gave a rigorous error analysis by  means of modulated Fourier expansion in a  recent paper (see \cite{Lubich-Shi2023}). We hope to make some progress on the formulation and analysis of a new class of large-stepsize methods for CPD in a strong non-uniform magnetic field in our future work.

\section*{Acknowledgement}
We are sincerely grateful to the two anonymous reviewers
for their valuable comments and helpful suggestions.
This work was supported by NSFC (12371403).

\clearpage

\end{document}